\pdfoutput=1
\documentclass{scrartcl}
\usepackage[intlimits]{amsmath}
\usepackage{amsthm,amssymb,mathtools,bm,ifpdf,authblk}
\usepackage{url}
\usepackage[abbrev]{amsrefs}
\usepackage{xargs,xifthen}
\usepackage[T1]{fontenc}
\usepackage{lmodern}
\usepackage{microtype}
\usepackage[usenames,dvipsnames]{color}
\usepackage{tocstyle}
\usepackage{hyperref}

\allowdisplaybreaks[1]

\newtheorem{thm}{Theorem}[section]
\newtheorem*{thm*}{Theorem}
\newtheorem{lemma}[thm]{Lemma}
\newtheorem*{lemma*}{Lemma}
\newtheorem{prop}[thm]{Proposition}
\newtheorem*{prop*}{Proposition}

\newtheorem*{conj*}{Conjecture}
\newtheorem{cor}[thm]{Corollary}

\theoremstyle{definition}
\newtheorem{ex}[thm]{Example}
\newtheorem*{ex*}{Example}

\newtheorem{defn}[thm]{Definition}
\newtheorem*{defn*}{Definition}

\newtheorem{rem}[thm]{Remark}
\newtheorem*{rem*}{Remark}
\newtheorem{qu}[thm]{Question}

\numberwithin{equation}{section}
\allowdisplaybreaks[1]

\renewcommand{\le}{\leqslant}
\renewcommand{\ge}{\geqslant}

\def\emptyset{\varnothing}

\makeatletter
\DeclareOldFontCommand{\rm}{\normalfont\rmfamily}{\mathrm}
\DeclareOldFontCommand{\sf}{\normalfont\sffamily}{\mathsf}
\DeclareOldFontCommand{\bf}{\normalfont\bfseries}{\mathbf}
\DeclareOldFontCommand{\it}{\normalfont\itshape}{\mathit}
\makeatother

\def\emph{}
\DeclareTextFontCommand{\bfemph}{\bf}
\DeclareTextFontCommand{\itemph}{\it}
\def\emph{\bfemph}

\makeatletter
\def\blankfootnote{\xdef\@thefnmark{}\@footnotetext}

\newcommand*{\textlabel}[2]{%
  \edef\@currentlabel{#1}
  \phantomsection
  #1\label{#2}
}
\makeatother

\newcommand{\rind}{\ensuremath{\mathfrak{e}}}

\newcommand{\cc}{\ensuremath{\mathsf{cc}}}
\newcommand{\oc}{\ensuremath{\mathsf{oc}}}
\newcommand{\ak}{\ensuremath{\mathsf{ask}}}

\newcommand{\idx}[1]{\lvert#1\rvert}

\newcommand{\ndivides}[2]{\ensuremath{ {#1} \nmid {#2} }}
\newcommand{\onto}{\twoheadrightarrow}
\newcommand{\incl}{\hookrightarrow}

\let\AA\undefined
\newcommand{\RF}{\mathfrak{K}}

\newcommand{\AA}{\mathbf{A}}
\newcommand{\PP}{\mathbf{P}}
\newcommand{\acts}{\ensuremath{\curvearrowright}}
\newcommand{\on}{\ensuremath{\mid}}
\newcommand{\QQ}{\mathbf{Q}}
\newcommand{\FF}{\mathbf{F}}
\newcommand{\GG}{\mathbf{G}}
\newcommand{\HH}{\mathbf{H}}
\newcommand{\NN}{\mathbf{N}}
\newcommand{\ZZ}{\mathbf{Z}}
\newcommand{\CC}{\mathbf{C}}
\newcommand{\RR}{\mathbf{R}}
\newcommand{\fa}{\ensuremath{\mathfrak a}}
\newcommand{\fb}{\ensuremath{\mathfrak b}}
\newcommand{\fg}{\ensuremath{\mathfrak g}}

\newcommand{\Places}{\ensuremath{\mathcal V}}

\newcommand{\Zeta}{\ensuremath{\mathsf{Z}}}

\newcommand{\omicron}{\ensuremath{o}}

\newcommand{\ee}{\ensuremath{\bm e}}
\newcommand{\ww}{\ensuremath{\bm w}}
\newcommand{\xx}{\ensuremath{\bm x}}
\newcommand{\yy}{\ensuremath{\bm y}}

\newcommand{\fo}{\mathfrak{o}}
\newcommand{\fO}{\mathfrak{O}}
\newcommand{\fP}{\mathfrak{P}}

\newcommand{\sG}{\mathsf{G}}
\newcommand{\sH}{\mathsf{H}}

\newcommand{\cW}{\mathcal{W}}
\newcommand{\cR}{\mathcal{R}}

\newcommand{\Var}{\mathbf{V}}

\newcommand{\descent}{\operatorname{d}_{\operatorname{B}}}
\DeclareMathOperator{\Des}{Des}
\newcommand{\negative}{\operatorname{N}}

\DeclareMathOperator{\concnt}{k}

\DeclareMathOperator{\Sym}{Sym}

\DeclareMathOperator{\BB}{B}

\DeclareMathOperator{\Char}{char}

\newcommand{\XX}{\bm{X}}
\newcommand{\ff}{\bm{f}}

\newcommand{\YY}{\bm{Y}}

\DeclareMathOperator{\disc}{disc}

\DeclareMathOperator{\GL}{GL}

\newcommand{\Nil}{\ensuremath{\mathfrak{n}}}
\newcommand{\Gl}{\ensuremath{\mathfrak{gl}}}
\newcommand{\Sl}{\ensuremath{\mathfrak{sl}}}
\newcommand{\So}{\ensuremath{\mathfrak{so}}}
\newcommand{\Sp}{\ensuremath{\mathfrak{sp}}}

\newcommand{\Diag}{\ensuremath{\mathfrak{d}}}

\newcommand{\tr}{\ensuremath{\mathfrak{tr}}}

\DeclareMathOperator{\Ker}{Ker}
\DeclareMathOperator{\iso}{iso}
\DeclareMathOperator{\Stab}{St}

\DeclareMathOperator{\cent}{c}
\DeclareMathOperator{\Fix}{Fix}
\DeclareMathOperator{\diag}{diag}
\DeclareMathOperator{\End}{End}
\DeclareMathOperator{\Hom}{Hom}
\DeclareMathOperator{\Uni}{U}
\DeclareMathOperator{\Mat}{M}
\DeclareMathOperator{\Spec}{Spec}
\DeclareMathOperator{\Proj}{Proj}
\DeclareMathOperator{\ad}{ad}
\DeclareMathOperator{\Ad}{Ad}

\DeclareMathOperator{\dd}{d\!}

\DeclareMathOperator{\kersize}{K}
\DeclareMathOperator{\orbsize}{O}

\newcommand{\normal}{\triangleleft}
\newcommand{\dtimes}{\ensuremath{\,\cdotp}}
\newcommand{\card}[1]{\lvert#1\rvert}

\DeclarePairedDelimiter{\abs}{\lvert}{\rvert}
\DeclarePairedDelimiter{\norm}{\lVert}{\rVert}

\DeclareMathOperator{\len}{len}
\DeclareMathOperator{\rank}{rk}
\DeclareMathOperator{\trace}{trace}
\DeclareMathOperator{\genrank}{grk}
\DeclareMathOperator{\genidim}{gor}
\DeclareMathOperator{\Real}{Re}

\DeclareMathOperator{\Img}{Im}

\newcommand{\llb}{\ensuremath{[\![ }}
\newcommand{\rrb}{\ensuremath{]\!] }}
\newcommand{\llp}{\ensuremath{(\!( }}
\newcommand{\rrp}{\ensuremath{)\!) }}

\newcommand{\ask}[2]{\operatorname{ask}( {#1} \ifthenelse{\isempty{#2}{}}{}{\on{#2}})}

\ifpdf
  \hypersetup{pdftitle={The average size of the kernel of a matrix and orbits of linear groups},
    pdfauthor={Tobias Rossmann}
  }
\fi
\title{The average size of the kernel of a matrix and orbits of linear groups}
\author{Tobias Rossmann}
\affil{\small Department of Mathematics\\University of Auckland\\Auckland\\
  New Zealand\\
  \vspace*{1em}
  \textit{Current address:}\\
  School of Mathematics, Statistics and Applied Mathematics\\
  National University of Ireland, Galway\\Galway\\
Ireland}
\date{}

\begin{document}

\maketitle
\thispagestyle{empty}

\vspace*{-4em}
\begin{abstract}
  \small
  Let $\mathfrak{O}$ be a compact discrete valuation ring
  of characteristic zero.
  Given a module~$M$ of matrices over $\mathfrak{O}$,
  we study the generating function encoding the average sizes of the kernels of
  the elements of $M$ over finite quotients of $\mathfrak{O}$.
  We prove rationality and establish fundamental properties of these
  generating functions and determine them explicitly for various natural
  families of modules~$M$.
  Using $p$-adic Lie theory, we then show that special cases of these
  generating functions enumerate orbits and conjugacy classes of suitable
  linear pro-$p$ groups.
\end{abstract}

\blankfootnote{\noindent{\itshape 2010 Mathematics Subject Classification.}
  15B33, 05A15, 11M41, 11S80, 20D60, 20D15, 20E45


  \medskip
  {\noindent
  The author gratefully acknowledges the support of the
  \href{https://www.humboldt-foundation.de}{Alexander von
    Humboldt Foundation} in the form of a Feodor Lynen Research Fellowship.
}}

\tableofcontents

\section{Introduction}
\label{s:intro}

This article is devoted to certain generating functions $\Zeta^{\ak}_M(T)$ 
(``ask zeta functions'')
attached to modules $M$ of matrices over compact discrete valuation rings.
The coefficients of $\Zeta^{\ak}_M(T)$ encode the average sizes of the kernels of
the elements of $M$ over finite quotients of the base ring.

Prior to formally defining these functions and stating our main results,
we briefly indicate how our study of the functions $\Zeta^{\ak}_M(T)$ is motivated by
questions from (both finite and infinite) group theory and
probabilistic linear algebra. 

\paragraph{Conjugacy classes of finite groups.}
Given a finite group $G$, let $\concnt(G)$ denote the number of its conjugacy
classes.
It is well-known that $\concnt(G)$ coincides with
the number of the (ordinary) irreducible characters of $G$.
Let $\Uni_d \le \GL_d$ be the group scheme of upper unitriangular $d\times d$
matrices.
Raised as a question in \cite{Hig60a}, ``Higman's conjecture'' asserts that
$\concnt(\Uni_d(\FF_q))$ is given by a polynomial in $q$ for each
fixed $d \ge 1$.

Numerous people have contributed to confirming Higman's conjecture for small~$d$.
In particular, 
building on a long series of papers, Vera-L\'opez and Arregi~\cite{VLA03} 
established Higman's conjecture for $d \le 13$.
Using a different approach, Pak and Soffer~\cite{PS15} 
recently provided a confirmation for $d\le 16$.
While Higman's conjecture remains open in general and despite
some evidence suggesting that it may fail to hold for large $d$ (see
\cite{PS15}), it nonetheless influenced and inspired numerous results on
related questions; see, in particular, work of 
Isaacs~~\cite{Isa95,Isa07} on character degrees of so-called 
algebra groups and work of Goodwin and R\"ohrle~\cite{GR09a,GR09b,GR09c,GR10}
on conjugacy classes of unipotent elements in groups of Lie~type.

\paragraph{Orbits of linear groups.}
All rings in this article are assumed to be commutative and unital.
Let $R$ be a ring and let $V$ be an $R$-module with $\card V < \infty$.
Given a linear group $G \le \GL(V)$, it is a classical problem (for
$R = \FF_q$) to relate arithmetic properties of the number of orbits of $G$ on
its natural module $V$ to geometric and group-theoretic properties of $G$; see
e.g.\ \cite{GLPST16} and the references therein. 
This problem is closely related to the enumeration of 
irreducible characters (and hence of conjugacy classes).
In particular, if $G$ is a finite $p$-group of nilpotency class less than $p$,
then the Kirillov orbit method establishes a bijection between the ordinary
irreducible characters of $G$ and the coadjoint orbits of $G$ on the dual of
its associated Lie ring; cf.\ \cite{GS09} and see~\cite{O'BV15} for
applications of such techniques to the enumeration of characters and
conjugacy classes.

\paragraph{Rank distributions and the average size of a kernel.}
In addition to group-theoretic problems such as those indicated above,
this article is also inspired by questions and results from probabilistic linear algebra.
For an elementary example, to the author's knowledge,
the number
\begin{equation}
  \label{eq:landsberg}
\prod_{i=0}^{r-1} (q^e-q^i)\frac{q^{d-i}-1}{q^{i+1}-1}
\end{equation}
of $d\times e$ matrices of rank $r$ with entries in
a finite field $\FF_q$ was first recorded by Landsberg~\cite{Lan93}.
More recently, probabilistic questions surrounding the distribution of ranks
in sets of matrices over finite fields have been studied, see e.g.\
\cite{Bal68,BKW97,Coo00,DGMS10} and \cite[Ch.~3]{Kol99}; for applications,
see \cite{SB10,LMT11}.

Let $R$ be a ring, let $V$ and $W$ be $R$-modules with $\card V, \card W < \infty$, 
and let $M \subset \Hom(V,W)$ be a submodule.
In the following, we are primarily interested in the case that $R$ is finite and $M \subset
\Mat_{d\times e}(R)$ acts by right-multiplication on $V = R^d$.
The \emph{a}verage \emph{s}ize of the \emph{k}ernel of the elements of~$M$ 
is
\[
\ask M {} := \ask M {V} := \card{M}^{-1} \sum_{a \in M}\card{\Ker(a)}.
\]

Linial and Weitz~\cite{LW00} gave the following formula for
$\ask{\Mat_{d\times e}(\FF_q)}{}$; 
the same result appeared (with a different proof) in a recent paper by
Fulman and Goldstein~{\cite[Lem.~3.2]{FG15}}
which also contains further examples of $\ask M {}$.

\begin{prop}
  \label{prop:FulmanGoldstein}
  $\ask {\Mat_{d\times e}(\FF_q)}{} = 1 + q^{d-e} - q^{-e}$.
\end{prop}

As we will see later, for a linear $p$-group $G \le \GL(V)$ with a
sufficiently strong Lie theory, $\card{V/G}$ and $\concnt(G)$
are both instances of $\ask {\fg} {}$ for suitable linear Lie algebras
$\fg$.

\paragraph{Orbit-counting and conjugacy class zeta functions.}
In the literature, numbers of the form $\card{V/G}$, $\concnt(G)$, and $\ask M
V$ for $R$-modules $V$ and $W$, a linear group $G \le \GL(V)$, and $M \subset
\Hom(V,W)$ were primarily studied in the case that $R = \FF_q$ is a finite field.
Instead of individual numbers, we consider families of such numbers obtained 
by replacing $\FF_q$ by the finite quotients of suitable rings.

We will use the following notation throughout this article.
Let $K$ be a non-Archimedean local field and let $\fO$ be its valuation
ring---equivalently, $\fO$ is a compact discrete valuation ring with field of
fractions~$K$;
we occasionally write $\fO_K$ instead of $\fO$ and similarly below.
For example, $K$ could be the field $\QQ_p$ of $p$-adic numbers (in which case
$\fO = \ZZ_p$ is the ring of $p$-adic integers) or the field $\FF_q\llp z\rrp$
of formal Laurent series over $\FF_q$ (in which case $\fO = \FF_q\llb z\rrb$).
Let $\fP$ denote the maximal ideal of~$\fO$.
Let $\RF := \fO/\fP$ be the residue field of $K$ and
let $q$ and $p$ denote the size and characteristic of $\RF$, respectively.
We write $\fP^n = \fP \dotsb \fP$ for the $n$th ideal power
of $\fP$; on the other hand, $\fO^n = \fO\times\dotsb\times\fO$ denotes the
$n$th Cartesian power of $\fO$.
Let $\fO_n = \fO/\fP^n$ and $\fO_\infty = \fO$.

\begin{defn}
  \label{d:cc_oc}
  Let $G \le \GL_d(\fO)$ be a subgroup.
\begin{enumerate}
\item 
Let $G_n \le \GL_d(\fO_n)$ denote the image of $G$ under the natural map
$\GL_d(\fO) \onto \GL_d(\fO_n)$.
The \emph{conjugacy class zeta function} of $G$ is
  $\Zeta_G^{\cc}(T) := \sum\limits_{n=0}^\infty \concnt(G_n) T^n$.
\item
  \label{d:cc_oc2}
 The \emph{orbit-counting zeta function} of $G$ is
  $\Zeta_G^{\oc}(T) := \sum\limits_{n=0}^\infty \card{\fO_n^d/G}T^n$.
\end{enumerate}
\end{defn}

Referring to these generating functions as ``zeta functions'' is
justified by various properties recalled or established in the following (e.g.\ the
existence of meromorphic continuation) for the associated  Dirichlet series
$\Zeta_G^{\cc}(q^{-s})$ and $\Zeta_G^{\oc}(q^{-s})$, at least in
characteristic zero.
Conjugacy class zeta functions were introduced by du~Sautoy~\cite{dS05} who
established their rationality for $\fO = \ZZ_p$.
Berman~et~al.~\cite{BDOP13} investigated $\Zeta_{\sG(\fO)}^{\cc}(T)$ for
Chevalley groups~$\sG$.
Lins~\cite{Lin18} recently determined $\Zeta_{\sG(\fO)}^{\cc}(T)$ for
certain families of unipotent group schemes $\sG$.
Special cases of the functions $\Zeta_G^{\oc}(T)$ have previously appeared in the literature.
In particular, 
Avni~et~al.~\cite[Thms~E, A.5]{AKOV16b} determined
orbit-counting zeta functions associated with the coadjoint representation
of $\GL_d$ and group schemes of the form $\mathrm{GU}_d$ for $d = 2,3$.

Conjugacy class and orbit-counting zeta  functions are natural analogues of
the numbers of conjugacy classes and orbits of finite groups from above.
For example, it is a natural generalisation of Higman's conjecture to ask,
for each fixed $d$,
whether $\Zeta^\cc_{\Uni_d(\fO_K)}(T)$ is given by a rational function in
$q_K$ and $T$ as a function of $K$.

\paragraph{The definition of $\Zeta^{\ak}_M(T)$.}
\label{p:ZM}
We now introduce the protagonist of this article.
Let $V$ and $W$ be finitely generated $\fO$-modules.
We frequently write $V_n = V \otimes \fO_n$ and $W_n = W \otimes \fO_n$,
where, in the absence of subscripts, tensor products are always taken over $\fO$.
Given a submodule $M \subset \Hom (V,W)$, 
we let $M_n$ denote the image of $M$ under the natural map
$\Hom(V,W) \to \Hom(V_n,W_n), \,a \mapsto a \otimes
\operatorname{id}_{\fO_n}$.
Crucially, the module $M_n$ does not merely depend on the abstract
module $M$ but rather on the given embedding of $M$ into $\Hom(V,W)$.
In particular, the natural surjection $M \otimes \fO_n
\onto M_n$ need not be an isomorphism;
for example, if $M = \fP \subset \fO = \End(\fO)$,
then $M \otimes \fO_1$ is isomorphic to $\fO_1$ but $M_1 = 0$.
In terms of matrices, for a submodule $M \subset \Mat_{d\times e}(\fO)$, 
we obtain $M_n \subset\Mat_{d\times e}(\fO_n)$ by reducing the entries of all matrices in $M$ modulo $\fP^n$.
This article is devoted to generating functions of the following form.
\begin{defn}
  \label{d:ZM}
  Let $M \subset \Mat_{d\times e}(\fO)$ be a submodule and $V = \fO^d$.
  Define the \emph{ask zeta function} of $M$ to be
  \[
  \Zeta_M(T) := \Zeta_M^\ak(T) := \Zeta^\ak_{M \acts V}(T) := \sum_{n=0}^\infty \ask{M_n}{V_n} T^n \in \QQ\llb T\rrb.
  \]
\end{defn}

In contrast to the probabilistic flavour of the work on the numbers $\ask M V$
cited above, our investigations of the functions $\Zeta^\ak_M(T)$ draw upon
results and techniques that have been previously applied in asymptotic group
theory and, specifically, the theory of  zeta functions 
(representation zeta functions, in particular) of groups and other
algebraic structures; for recent surveys of this area, see
\cite{Vol11,Klo13,Vol15}. 
Conversely, our study of ask zeta functions contributes to asymptotic group
theory: we will see that orbit-counting and conjugacy class zeta functions of
suitable groups are instances of ask zeta functions.

\paragraph{Results I: fundamental properties and examples of ask zeta functions.}

Our central structural result on the functions $\Zeta^\ak_M(T)$ is the following.

\begin{thm}
  \label{thm:rational}
  Let $\fO$ be the valuation ring of a non-Archimedean local field of
  characteristic zero.
  Let $M \subset \Mat_{d\times e}(\fO)$ be a submodule.
  Then $\Zeta^\ak_M(T)$ is rational, i.e.\ $\Zeta^\ak_M(T) \in \QQ(T)$.
\end{thm}

For example, $\Zeta^\ak_{\{0_{d\times e}\}}(T) = 1/(1-q^dT)$.
At the other extreme, we will obtain the following generalisation of
Proposition~\ref{prop:FulmanGoldstein}.

\begin{prop}
  \label{prop:Mdxe}
  Let $\fO$ be the valuation ring of a non-Archimedean local field of arbitrary
  characteristic.
  Let~$q$ be the residue field size of $\fO$.
  Then 
  \begin{equation}
    \label{eq:Mdxe}
  \Zeta^\ak_{\Mat_{d\times e}(\fO)}(T) = \frac{1-q^{-e}T}{(1-T)(1-q^{d-e}T)}.
  \end{equation}
\end{prop}

Note that since $\Zeta^\ak_{\Mat_{d\times e}(\fO)}(T) = 1 + (1 + q^{d-e} - q^{-e})
T + \mathcal O(T^2)$,
Proposition~\ref{prop:Mdxe} indeed generalises
Proposition~\ref{prop:FulmanGoldstein}.
Apart from proving Proposition~\ref{prop:Mdxe}, in \S\ref{s:examples_max}, we will
also determine $\Zeta^\ak_M(T)$ for traceless (Corollary~\ref{cor:sld}), symmetric
(Proposition~\ref{prop:Symd}), anti-symmetric (Proposition~\ref{prop:so}),
upper triangular (Proposition~\ref{prop:n_tr}),
and diagonal (Corollary~\ref{cor:diagonal}) matrices.
We will also explain why many of our formulae are of the same shape
as \eqref{eq:Mdxe}.
Despite this wealth of explicit examples in arbitrary characteristic,
the author does not know if Theorem~\ref{thm:rational} remains true in
positive characteristic; see \S\ref{sss:positive}.

Our proofs of Theorem~\ref{thm:rational}, Proposition~\ref{prop:Mdxe}, and various
other results in this article rest upon expressing the functions $\Zeta^\ak_M(T)$
in terms of suitable integrals (Theorem~\ref{thm:twoint}).
These integrals can then be studied using powerful techniques 
developed over the past decades, primarily in the context of
Igusa's local zeta function 
(see \cite{Den91a,Igu00} for introductions).
Our use of these techniques is similar to and inspired by their
applications in the theory of zeta functions of groups 
and, in particular, the study of representation growth; see
\cite{GSS88,JZ06,Vol10,AKOV13,SV14}.
In particular, Theorem~\ref{thm:rational} follows from rationality
results going back to Igusa and Denef.
Using a theorem of
Voll~\cite{Vol10}
we will furthermore see that the identity
\begin{equation}
  \label{eq:Mde_FEqn}
  \Zeta^\ak_{\Mat_{d\times
    e}(\fO)}(T) \Biggm\vert_{(q,T)\to (q^{-1},T^{-1})} = (-q^d
T)\dtimes \Zeta^\ak_{\Mat_{d\times e}(\fO)}(T)
\end{equation}
is no coincidence (Theorem~\ref{thm:feqn}).
Our $p$-adic formalism is also compatible with our previous
computational work (summarised in \cite{spp1489}) which allows us to
explicitly compute numerous further examples of~$\Zeta^\ak_M(T)$;
see \S\ref{s:zeta} for some of these.

While ``random matrices'' over local fields have been studied before
(see e.g.\ \cite{Eva02}), the author is not aware of previous
applications of the particular techniques employed 
(and the point of view taken) here.

\paragraph{Results II: ask zeta functions and asymptotic group theory.}
We say that a formal power series $F(T) \in \QQ\llb T\rrb$ has 
\emph{bounded denominators} if there exists a non-zero $a \in \ZZ$ such that
$aF(T) \in \ZZ\llb T\rrb$.
As usual, for a ring $R$ and $R$-module $V$, let $\Gl(V)$ 
denote the Lie algebra associated with the associative algebra
$\End(V)$;
that is, $\Gl(V) = \End(V)$ as $R$-modules 
and the Lie bracket of $\Gl(V)$ is defined in terms of 
multiplication in $\End(V)$ via $[a,b] = ab - ba$.

\begin{thm}
  \label{thm:bounded_denominators}
  Let $\fO$ be the valuation ring of a non-Archimedean local field of
  characteristic zero.
  Let $\fg \subset \Gl_d(\fO)$ be a Lie subalgebra.
  Then $\Zeta^\ak_{\fg}(T)$ has bounded denominators.
\end{thm}

Theorem~\ref{thm:bounded_denominators}
is based on a connection between $\Zeta^\ak_{\fg}(T)$ and 
orbit-counting zeta functions.
For a sketch, let $G\le \GL_d(\fO)$ act on $V = \fO^d$.
As before, we write $\fO_n = \fO/\fP^n$ and $V_n = V \otimes \fO_n$.
Then $G$ acts on each of the finite sets $V_n$
and, extending our previous definition of the orbit-counting zeta function
$\Zeta_G^\oc(T)$ (Definition~\ref{d:cc_oc}(\ref{d:cc_oc2})), we let
\[
\Zeta^{\oc,m}_G(T) = \sum_{n=m}^\infty \card{V_n/G} \dtimes T^{n-m} \in \ZZ\llb T\rrb;
\]
hence, $\Zeta^\oc_G(T) = \Zeta_G^{\oc,0}(T)$.
In the setting of Theorem~\ref{thm:bounded_denominators},
by linearising the orbit-counting lemma using $p$-adic Lie theory,
we will see that for sufficiently large $m$, 
there exists $G^m \le \GL_d(\fO)$ with
$q^{dm}\Zeta^\ak_{\fg}(T) = \Zeta_{G^m}^{\oc,m}(T)$.
Theorem~\ref{thm:bounded_denominators} then follows immediately.

In addition to using group theory to deduce properties of ask zeta functions
such as Theorem~\ref{thm:bounded_denominators},
we will see that, conversely, our methods for studying ask zeta functions
allow us to deduce results on both orbit-counting and conjugacy class zeta
functions.
As we will now sketch,
this direction is particularly fruitful for unipotent groups.
For a Lie algebra $\fg$ over a ring $R$, let $\ad\colon \fg \to \Gl(\fg)$
denote its adjoint representation given by $\ad(a)\colon b\mapsto [b,a]$ for
$a\in \fg$.
Let $\Nil_d(R) \subset \Gl_d(R)$ denote the Lie algebra of strictly upper triangular $d\times
d$ matrices.

\begin{thm}
  \label{thm:unipotent}
  Let $\fO$ be the valuation ring of a local field 
  of characteristic zero and residue characteristic $p$.
  Let $\fg \subset \Nil_d(\fO)$ be a Lie subalgebra
  and let $G := \exp(\fg) \le \Uni_d(\fO)$.
  Suppose that $p \ge d$ and that $\fg$ is an isolated submodule of
  $\Nil_d(\fO)$ (i.e.\ the $\fO$-module quotient $\Nil_d(\fO)/\fg$ is torsion-free).
  Then $\Zeta_G^\oc(T) = \Zeta^\ak_{\fg}(T)$ and $\Zeta_G^\cc(T) = \Zeta^\ak_{\ad(\fg)}(T)$.
\end{thm}

We will apply Theorem~\ref{thm:unipotent} and the methods for
computing $\Zeta^\ak_M(T)$ developed below in order to compute ``generic''
conjugacy class zeta functions arising from all unipotent algebraic groups of
dimension at most $5$ over a number field (see \S\ref{ss:ex_cc}).

Due to the heavy reliance of the proofs of
Theorems~\ref{thm:bounded_denominators}--\ref{thm:unipotent} on $p$-adic Lie theory, it is unclear to the author whether these results have
analogues over local fields of positive characteristic.

\paragraph{Outline.}
In \S\S\ref{s:elementary_ask}--\ref{s:elementary_zeta}, we collect 
elementary facts on $\ask M V$ and $\Zeta^\ak_M(T)$.
We then derive expressions for $\Zeta^\ak_M(T)$ in terms of suitable integrals
in \S\ref{s:ratint}. In \S\ref{s:examples_max}, we use these to compute 
explicit formulae for $\Zeta^\ak_M(T)$ for various modules $M$.
Next, in \S\ref{s:examples_min},
we discuss a relationship between the functions $\Zeta^\ak_M(T)$ and
``constant rank spaces'' studied extensively in the literature.
A geometric source of interesting examples of ask zeta functions,
determinantal hypersurfaces, is considered in \S\ref{s:det}.
In \S\ref{s:bounded_denominators}, we explore the aforementioned connection
between ask, conjugacy class, and orbit-counting zeta functions in
characteristic zero.
In particular, we prove
Theorems~\ref{thm:bounded_denominators}--\ref{thm:unipotent}.
Finally, given that most of the explicit formulae for $\Zeta^\ak_M(T)$ obtained in
\S\S\ref{s:examples_max}--\ref{s:examples_min} are quite tame, \S\ref{s:zeta}
contains a number of  more complicated examples of $\Zeta^\ak_M(T)$ and
$\Zeta_G^\cc(T)$.

\subsection*{Acknowledgement}

The author is grateful to Angela Carnevale and Christopher Voll for various
discussions, to Eamonn O'Brien for helpful comments, to Christoph Hanselka for
pointing out a flaw in an earlier version, and to the referee for numerous
insightful suggestions. 

\section{Elementary properties of average sizes of kernels}
\label{s:elementary_ask}

We collect some elementary observations on average sizes of kernels.
Throughout, unless otherwise stated, let $R$ be a ring, let $V$ and
$W$ be $R$-modules with $\card V, \card W < \infty$, and let $M
\subset \Hom(V,W)$ be a submodule. 

\subsection{Rank varieties}

In the case of a finite field $R = \FF_q$, $\ask M
V$ admits a natural geometric interpretation.
Namely, by choosing a basis of $M$, we may identify $M =
\AA^\ell_{\FF_q}(\FF_q)$,
where $\ell = \dim_{\FF_q}(M)$ and $\AA^\ell_{\FF_q} = \Spec(\FF_q[X_1,\dotsc,X_\ell])$.
We may then decompose $\AA^\ell_{\FF_q} = \coprod\limits_{i=0}^{r} V_i$, where $V_i$ is the
subvariety of maps of rank $i$.
(Note that if $M = \Mat_{d\times e}(\FF_q)$, then $\#
V_r(\FF_q)$ is given by \eqref{eq:landsberg}.)
Then
$\ask {M}{V} = \sum_{i=0}^{d} \# V_i(\FF_{q}) \dtimes q^{d-i-\ell}$;
in fact, by replacing $q$ by $q^f$ on the right-hand side,
we express $\ask{M\otimes_{\FF_q}{\FF_{q^f}}}{V\otimes_{\FF_q}{\FF_{q^f}}}$
using a formula which is valid for any $f \ge 1$.
However, even for $M = \Mat_{d\times e}(\FF_q)$,
this approach yields a fairly complicated interpretation of
Proposition~\ref{prop:FulmanGoldstein}.

\subsection{Kernels and orbits}
\label{ss:duality}

One simple yet crucial observation contained in the proof of
\cite[Thm~1.1]{LW00} is the following connection between the sizes of the
kernels $\Ker(a)$ ($a \in M$) and those of the ``orbits''
$\xx M := \{ \xx a : a \in M\} \approx M/\cent_M(\xx)$,
where $\xx \in V$ and $\cent_M(\xx) := \{ a \in M : \xx a = 0\}$;
note that in contrast to orbits under group actions, the sets $\xx M$ always overlap.

\begin{lemma}[{Cf.\ \cite[Thm~1.1]{LW00}}]
  \label{lem:WxM}
  $\ask M V = \sum\limits_{\xx\in V} \card{\xx M}^{-1}$.
\end{lemma}

We give two proofs of this lemma. 
The first is a combinatorial version of a probabilistic argument in
the proof of \cite[Thm~1.1]{LW00}.
We include it here since our terminology is different
from theirs; similar arguments appear in~\cite{Loc07}. 

\begin{proof}[First proof of Lemma~\ref{lem:WxM}]
  By computing $\#\{ (\xx,a) \in V\times M : \xx a = 0\}$ in two ways,
  we obtain
  $\sum\limits_{a \in M} \card{\Ker(a)} = \sum\limits_{\xx \in V}\card{\cent_M(\xx)}$.
  Since $\xx M \approx M/\cent_M(\xx)$ as $R$-modules, we have
  $\card{\cent_M(\xx)} = \card M /\card{\xx M}$ and thus
  $\ask M V = \card{M}^{-1} \sum\limits_{\xx \in V} \card{\cent_M(\xx)} =
  \sum\limits_{\xx \in V} \card{\xx M}^{-1}$.
\end{proof}

Our second proof of Lemma~\ref{lem:WxM} already hints at the
connection between average sizes of kernels and orbits of linear
groups, a subject further explored in \S\ref{s:bounded_denominators}.
Recall that for a finite group $G$ acting on a finite set $X$, the
orbit-counting lemma asserts that
$\card{X/G} = \card{G}^{-1} \sum\limits_{g \in G} \card{\Fix(g)}$,
where $\Fix(g) = \{ x \in X : xg = x\}$.

\begin{proof}[Second proof of Lemma~\ref{lem:WxM}]
  The rule
  $a \mapsto a^* := \bigl[\begin{smallmatrix} 1 & a
      \\ 0 & 1\end{smallmatrix}\bigr]$ yields an isomorphism of $(M,+)$ onto a subgroup $M^*$
  of $\GL(V\oplus W)$.
  We claim that the natural bijection $V\oplus W \to
  \coprod\limits_{V} W$ induces a bijection
  $(V\oplus W)/M^* \to \coprod\limits_{\xx\in V} W/xM$.
  Indeed, $(\xx,\yy)a^* = (\xx,\yy + \xx a)$ for $(\xx,\yy) \in V \oplus W$.
  As $\Fix(a^*) = \Ker(a) \oplus W$,
  the orbit-counting lemma yields
  \[
  \card W \sum_{\xx \in V} \card{\xx M}^{-1} = 
  \card{(V\oplus W)/M^*} =
  \card{M^*}^{-1}\sum_{a \in M}\card{\Fix(a^*)} = \card W \dtimes \ask{M}{V}. \qedhere
  \]
\end{proof}

In order to deduce Proposition~\ref{prop:FulmanGoldstein}
from Lemma~\ref{lem:WxM}, note that $\xx \Mat_{d\times e}(\FF_q) =
\FF_q^e$ for each non-zero $\xx \in \FF_q^d$ whence
$\ask{\Mat_{d\times e}(\FF_q)}{} = 1 + (q^d-1)q^{-e}$.

\subsection{Interlude: rank distributions and hyperoctahedral groups}
\label{ss:permstat}

We discuss combinatorial consequences of Proposition~\ref{prop:FulmanGoldstein}.

\paragraph{Reminder: the hyperoctahedral groups $\BB_n$.}
For background and details on the following, see \cite[\S 3]{Bre94} or \cite[\S 8.1]{BB05}.
The \emph{hyperoctahedral group} $\BB_n = \{ \pm 1\} \wr \operatorname{S}_n$ is the group of signed
permutations on $n$ letters; we regard $\BB_n$ as a subgroup of the symmetric
group on $\{ \pm 1,\dotsc,\pm n\}$ and as a Coxeter group in the usual way
(see \cite[p.\ 246]{BB05}).
For $\sigma \in \BB_n$, we write $\sigma = [1^{\sigma},\dotsc,n^{\sigma}]$. 
For $\sigma\in\BB_n$, let $\len(\sigma)$ denote the (Coxeter) \emph{length} of $\sigma$
(see \cite[Prop.~3.1]{Bre94} for a combinatorial description),
let $\negative(\sigma) := \#\bigl\{ i \in \{1,\dotsc,n\} : i^\sigma <
0\bigr\}$, and  let $\Des(\sigma) := \bigl\{ i \in \{0,\dotsc,n-1\} : i^\sigma >
(i+1)^\sigma\bigr\}$ (where we wrote $0^\sigma = 0$)
denote the \emph{descent set} of $\sigma$.
For $I\subset\{0,\dotsc,n\}$, define the \emph{quotient} $\BB_n^{I^c}
:= \{ \sigma \in \BB_n : \Des(\sigma) \subset I\}$
(see \cite[\S 2.4]{BB05}).
The identity $1\in\BB_n$ is the unique element of length $0$.
Moreover, since $\Des(1) = \emptyset$, the identity is contained in each
set $\BB_n^{I^c}$.
\vspace*{1em}

\noindent
Let $\Mat_d^{\rank=i}(\FF_q) := \{ a\in \Mat_d(\FF_q) : \rank(a) = i\}$.
As explained in \cite[\S 3.2]{CSV17}, for $i = 0,\dotsc,d$,
\begin{equation}
  \label{eq:Mdrki}
  \left\lvert {\Mat_d^{\rank=d-i}(\FF_q)}\right\rvert = 
  q^{d^2-i^2} \sum_{\sigma \in \BB_d^{\{i\}^c}}(-1)^{\negative(\sigma)} q^{-\len(\sigma)}
\end{equation}
whence
\begin{align}
  \ask {\Mat_d(\FF_q)}{}
  & = q^{-d^2} \sum_{i=0}^d \left\lvert{\Mat_d^{\rank=d-i}(\FF_q)}\right\rvert
    q^i \nonumber
  \\ &
       = \sum_{i=0}^d q^{i-i^2} \sum_{\sigma\in \BB_d^{\{i\}^c}}
       (-1)^{\negative(\sigma)} q^{-\len(\sigma)}.
       \label{eq:Mdq_permstat}
\end{align}

On the other hand, by Proposition~\ref{prop:FulmanGoldstein},
$\ask{\Mat_d(\FF_q)}{} = 2-q^{-d}$.
The right-hand side of \eqref{eq:Mdq_permstat} is a polynomial in $\ZZ[q^{-1}]$,
the constant term, $2$, of which arises from $\sigma = 1\in \BB_n^{\{i\}^c}$ and $i = 0,1$.
However, the fact that the other terms of \eqref{eq:Mdq_permstat} add up to
$-q^{-d}$ seems much less transparent.
Consider, for example, the case $d = 2$.
For $i = 0,1$, we have $q^{i-i^2} = 1$ and the contributions to the right-hand
side of \eqref{eq:Mdq_permstat} are exactly the terms
$(-1)^{\negative(\sigma)}q^{-\len(\sigma)}$ indicated in the following tables. 

\begin{center}
  \begin{tabular}{r|l}
    $\sigma \in \BB_2^{\{0\}^c}$ & $(-1)^{\negative(\sigma)}q^{-\len(\sigma)}$ \\\hline
    $1$ & $+1$ \\ 
    $[-1,2]$ & $-q^{-1}$ \\
    $[-2,1]$ & $-q^{-2}$ \\
    $[-2,-1]$ &$+q^{-3}$
  \end{tabular}
  \hspace*{1em}
  \begin{tabular}{r|l}
    $\sigma \in \BB_2^{\{1\}^c}$ & $(-1)^{\negative(\sigma)}q^{-\len(\sigma)}$ \\\hline
    $1$ & $+1$ \\ 
    $[2,1]$ & $+q^{-1}$\\
    $[2,-1]$ & $-q^{-2}$\\
    $[1,-2]$ & $-q^{-3}$
  \end{tabular}
\end{center}

For $i = 2$, $\BB_2^{\{2\}^c} = \{ 1\}$ and the contribution to the right-hand
side of \eqref{eq:Mdq_permstat} is a single summand $q^{-2}$.
While we can therefore confirm that
\[
2-q^{-2} = \ask{\Mat_2(\FF_q)}{} = (1 - q^{-1} - q^{-2} + q^{-3}) + (1 + q^{-1} - q^{-2} - q^{-3}) + q^{-2},
\]
the author is unable to provide a combinatorial explanation of this numerical coincidence.

A further source of such examples is given by
analogues of \eqref{eq:Mdrki} for the numbers of traceless, antisymmetric, and
symmetric $d\times d$ matrices over $\FF_q$, respectively, due to Carnevale et
al.~\cite[\S 3.2]{CSV17}; cf.\ \cite{SV14}.
The average sizes of the kernels of all these spaces of matrices are 
known or can be deduced as by-products or our investigations in
\S\ref{ss:classical_Lie} below.

\subsection{Direct sums}

\begin{lemma}
  \label{lem:oplus}
  Let $V'$ and $W'$ be $R$-modules with $\card{V'},\card{W'} < \infty$.
  Let $M' \subset \Hom(V',W')$ be a submodule.
  We regard $M \oplus M'$ as a submodule of $\Hom(V\oplus V',W \oplus
  W')$ in the natural~way.
  Then $\ask{M\oplus M'}{V \oplus V'} = \ask M V \dtimes \ask{M'}{V'}$.
\end{lemma}
\begin{proof}
  \begin{align*}
    \ask{M\oplus M'}{V\oplus V'} 
    & = \card{M \oplus M'}^{-1} \dtimes \sum\limits_{(a,a')\in M\oplus M'}
  \card{\Ker(a\oplus a')} \\
    & = \card{M}^{-1} \card{M'}^{-1} \dtimes \sum\limits_{(a,a')\in M\oplus M'}
  \card{\Ker(a)} \dtimes \card{\Ker(a')} \\
    & = \ask{M}{V} \dtimes \ask{M'}{V'}.\qedhere
  \end{align*}
\end{proof}

\begin{cor}
  \label{cor:ask_zcol}
  Let $R$ be finite, $M \subset \Mat_{d\times e}(R)$ be a
  submodule, and $\tilde M \subset \Mat_{(d+1)\times e}(R)$ be
  obtained from $M$ by adding a zero row to the elements of $M$ in some fixed
  position.
  Then $\ask{\tilde M}{} = \ask{M}{} \dtimes \card R$.
  \qed
\end{cor}

\subsection{Matrix transposes}

Following Kaplansky~\cite{Kap49}, we call $R$ an \emph{elementary divisor
  ring} if for each $a \in \Mat_{d\times e}(R)$ (and all $d,e \ge 1$), there
exist $u\in \GL_d(R)$ and $v \in \GL_e(R)$ such that $uav$ is a diagonal
matrix (padded with zeros according to the shape of $a$).
For example, any quotient of a principal ideal domain is an elementary divisor ring
(regardless of whether the quotient is an integral domain or not).
Write $a^\top$ for the transpose of $a$.

\begin{lemma}
  \label{lem:ask_transpose}
  Let $R$ be a finite elementary divisor ring and let $M \subset
  \Mat_{d\times e}(R)$ be a submodule.
  Write $V = R^d$ and $W = R^e$.
  Then $\ask {M^\top} W = \ask M V \dtimes \card{R}^{e-d}$.
\end{lemma}
\begin{proof}
  Let $a = \left[\begin{smallmatrix} \diag(a_1,\dotsc,a_r) & 0 \\ 0 & 0 \end{smallmatrix}\right]
  \in \Mat_{d\times e}(R)$. Then $\Ker(a)$ consists of those
  $\xx \in V$ with $a_i x_i = 0$ for $1\le i \le r$ and $\Ker(a^\top)$
  consists of those $\yy \in W$ with $a_i y_i = 0$ for $1\le i \le r$.
\end{proof}

\subsection{Reduction modulo $\fa$ and base change $R \to R/\fa$}
\label{ss:reduction_avg}

Let $V$ and $W$ be finitely generated $R$-modules, the underlying sets of
which need not be finite. As before, let $M \subset \Hom(V,W)$ be a submodule.
Let $\fa\normal R$ with $\card{R/\fa} < \infty$. 
Define $V_{\fa} := V \otimes_R R/\fa$,
$W_{\fa} := W \otimes_R R/\fa$, and let $M_{\fa}$
be the image of the natural map $M \incl \Hom(V,W) \to \Hom(V_\fa,W_\fa)$.
In general, the natural surjection $M \otimes_R R/\fa \onto
M_{\fa}$ need not be injective (see the example on p.\ \pageref{p:ZM}).
However, if $M$ is finitely generated, then $M \otimes_R R/\fa$ is
finite and we obtain the following expression for $\ask{M_\fa}{V_\fa}$.

\begin{lemma}
  \label{lem:alt_ask}
  Suppose that $M$ is finitely generated.
  Then 
  \[
  \pushQED{\qed}
  \ask{M_\fa}{V_{\fa}} = \card{M \otimes_R R/\fa}^{-1}
  \sum_{\bar a \in M \otimes_R R/\fa} \card{\Ker(\bar a \on
    V_\fa)}. \qedhere \popQED
  \]
\end{lemma}

\section{Basic algebraic and analytic properties of $\Zeta_M(T)$ and $\zeta_M(s)$} 
\label{s:elementary_zeta}

\subsection{Average sizes of kernels and Dirichlet series: $\zeta_M(s)$}

While our main focus is on the generating functions $\Zeta_M(T)$ from
the introduction, it is natural to also consider a global analogue.
First suppose that $R$ is a ring which contains only finitely many
ideals $\fa$ of a given finite norm $\card{R/\fa}$.
Given a submodule $M \subset \Mat_{d\times e}(R)$ acting on $V = R^d$
and an ideal $\fa \normal R$,
let $V_{\fa} = (R/\fa)^d$, $W_{\fa} = (R/\fa)^e$,
and let $M_{\fa}$ denote the image of the natural map $M \incl
\Mat_{d\times e}(R) \onto \Mat_{d\times e}(R/\fa)$ (cf.~\S\ref{ss:reduction_avg}).

\begin{defn}
  \quad
  \begin{enumerate}
  \item
    Define a formal Dirichlet series
    \[
    \zeta_M(s) = \sum_\fa \ask{M_\fa}{V_\fa} \dtimes \card{R/\fa}^{-s},
    \]
    where the sum extends over the ideals of finite norm of $R$
    and $s$ denotes a complex variable.
  \item
    Let $\alpha_M \in [-\infty,\infty]$ denote the abscissa of convergence of
    $\zeta_M(s)$. 
  \end{enumerate}
\end{defn}

\subsection{Abscissae of convergence: local case}

Let $K$ be a local field of arbitrary characteristic with valuation
ring $\fO$ and residue field size $q$.
Let $M \subset \Mat_{d\times e}(\fO)$ be a submodule acting on $V =
\fO^d$.
Then $\zeta_M(s) = \Zeta_M(q^{-s})$.
Moreover, if $\fO$ has characteristic zero, then
Theorem~\ref{thm:rational} 
(proved in \S\ref{sss:rationality})
implies that $\alpha_M$ is precisely the
largest real pole of (the meromorphic continuation of) $\zeta_M(s)$.

Recall that, unless otherwise indicated, tensors products are taken over $\fO$.

\begin{defn}
  \label{d:genidim}
  The \emph{generic orbit rank} of $M$ is $\genidim(M) :=
  \max\limits_{\xx \in V} \dim_K(\xx M \otimes K)$. 
\end{defn}

Our choice of terminology will be justified by Proposition~\ref{prop:genidim}.

\begin{prop}
  \label{prop:local_alpha}
  $\max\bigl(d - \genidim(M),0\bigr) \le \alpha_M  \le d$.
\end{prop}
\begin{proof}
  The upper bound follows since $\ask{M_n}{V_n} \le \card{V_n} =
  q^{nd}$ and $\sum_{n=0}^\infty q^{n(d-s)}$ converges for $\Real(s) >
  d$.
  Similarly, the lower bound follows from Lemma~\ref{lem:WxM} and
  $\ask{M_n}{V_n} \ge \max(\card{V_n} / q^{n \genidim(M)},1)$.
\end{proof}

Let $0 \le r \le e$.
Let $M \subset \Mat_{d\times e}(\fO)$ be obtained from $\Mat_{d\times
  r}(\fO)$ by inserting $e-r$ zero columns in some fixed
positions.
Then $\genidim(M) = r$,
$\zeta_{M}(s) = \zeta_{\Mat_{d\times r}(\fO)}(s)$,
and it will follow from Proposition~\ref{prop:Mdxe}
that $\alpha_M = \max(d - r,0)$.
In particular, the bounds in Proposition~\ref{prop:local_alpha} are optimal.
We note that Example~\ref{ex:unbounded_denom} below illustrates that
the meromorphic continuation of $\zeta_M(s)$ (cf.\
Theorem~\ref{thm:rational}) may have real poles less than $d-e$.

\subsection{Abscissae of convergence in the global case and Euler products}
\label{ss:global_alpha}

Let $k$ be a number field with ring of integers $\fo$.
Let $\Places_k$ denote the set of non-Archimedean places of $k$.
For $v \in \Places_k$, let $k_v$ be the $v$-adic completion of $k$ and let
$\fo_v$ be its valuation ring.
We let $q_v$ denote the size of the residue field $\RF_v$ of $k_v$.
For an $\fo$-module $U$ and $v \in \Places_k$, we write $U_v := U
\otimes_{\fo}\fo_v$ (regarded as an $\fo_v$-module).

Let $M \subset \Mat_{d\times e}(\fo)$ be a submodule.
For $v \in \Places_k$,
we may identify $M_v$ with the $\fo_v$-submodule of $\Mat_{d\times e}(\fo_v)$
generated by $M$.

\begin{prop}
  \label{prop:euler}
  Let $M \subset \Mat_{d\times e}(\fo)$ be a submodule. Then:
  \begin{enumerate}
  \item\label{prop:euler1} $\alpha_M \le d+1$.
  \item\label{prop:euler2} $\zeta_M(s) = \prod\limits_{v\in \Places_k} \zeta_{M_v}(s)$.
  \end{enumerate}
\end{prop}
\begin{proof}
  Let $V = \fo^d$.
  The proof of (\ref{prop:euler1}) is similar to that of
  Proposition~\ref{prop:local_alpha}.
  Namely, for each $\fa \normal \fo$, $\ask{M_\fa}{V_\fa} \le
  \card{\fo/\fa}^d$ and
  $\sum_{\fa} \card{\fo/\fa}^{d-s} = \zeta_{k}(s-d)$ converges for
  $\Real(s) > d + 1$, where $\zeta_k(s)$ is the Dedekind zeta
  function of $k$.
  For (\ref{prop:euler2}), it suffices to show that for non-zero coprime ideals
  $\fa,\fb \normal \fo$, $\ask{M_{\fa\fb}}{V_{\fa\fb}} =
  \ask{M_\fa}{V_\fa} \dtimes \ask{M_\fb}{V_{\fb}}$.

  To that end, the natural isomorphism $\fo/\fa\fb \to \fo/\fa \times
  \fo/\fb$ yields an ($\fo$-module) isomorphism
  $M\otimes_\fo \fo/\fa\fb \to (M\otimes_\fo \fo/\fa) \times (M \otimes_\fo \fo/\fb)$
  which is compatible with the corresponding isomorphism
  $V_{\fa\fb} \to V_{\fa} \times V_{\fb}$ in the evident way.
  Hence, for $\bar a \in M \otimes_\fo \fo/\fa\fb$
  corresponding to $(\bar a_\fa,\bar a_\fb) \in (M\otimes_\fo \fo/\fa)\times
  (M\otimes_\fo \fo/\fb)$,
  we obtain an isomorphism
  $\Ker(\bar a \on V_{\fa\fb}) \to \Ker(\bar a_\fa \on V_{\fa}) \times
  \Ker(\bar a_\fb \on V_{\fb})$.
  Part (\ref{prop:euler2}) thus follows from Lemma~\ref{lem:alt_ask}.
\end{proof}

\begin{ex}
  Let $\zeta_k(s)$ denote the Dedekind zeta function of $k$.
  Then Proposition~\ref{prop:Mdxe} and Proposition~\ref{prop:euler}(\ref{prop:euler2}) imply
  that
  $\zeta_{\Mat_{d\times e}(\fo)}(s) = \zeta_k(s)\zeta_k(s-d+e)/\zeta_k(s+e)$.
\end{ex}

Further analytic properties of $\zeta_M(s)$ in a global setting will be
derived in \S\ref{s:feqn}.

\subsection{Hadamard products}

Recall that the \emph{Hadamard product} $F(T) \star G(T)$ of formal power
series $F(T) = \sum_{n=0}^\infty a_nT^n$
and $G(T) = \sum_{n=0}^\infty b_n T^n$ 
with coefficients in some common ring is
\[
F(T) \star G(T) = \sum_{n=0}^\infty a_nb_n T^n.
\]

The following is an immediate consequence of Lemma~\ref{lem:oplus}.
\begin{cor}
  \label{cor:hadamard}
  Let $K$ be a local field of arbitrary characteristic with valuation ring~$\fO$.
  Let $M = A \oplus B \subset \Mat_{d+e}(\fO)$
  for submodules $A \subset \Mat_d(\fO)$ and $B \subset \Mat_e(\fO)$.
  Then $\Zeta_M(T) = \Zeta_A(T) \star \Zeta_B(T)$. \qed
\end{cor}

We note that Hadamard products of rational generating functions are
well-known to be rational (see \cite[Prop. 4.2.5]{Sta12}).

\begin{cor}
  \label{cor:square}
  Let $M \subset \Mat_{d\times e}(\fO)$ be a submodule.
  Define $f = \max(d,e)$.
  Let $\tilde M \subset \Mat_f(\fO)$ be obtained from $M$ by
  adding $f-d$ zero rows and $f-e$ zero columns to the elements of $M$
  in some fixed positions.
  Then $\Zeta_{\tilde M}(T) = \Zeta_M(q^{f-d} T)$. \qed
\end{cor}

Thus, various questions on the series $\Zeta_M(T)$ are reduced to the case of square matrices.

\subsection{Rescaling}
\label{ss:rescaling}

Let $\fO$ be the valuation ring of a non-Archimedean local field $K$
of arbitrary characteristic.
Let $M \subset \Mat_{d\times e}(\fO)$ be a submodule, $V = \fO^d$, and
$W = \fO^e$.

\begin{defn}
For $m \ge 0$, let
$\Zeta^m_M(T) := \sum\limits_{n=m}^\infty  \ask{M_n}{V_n} \dtimes T^{n-m} \in \QQ\llb T\rrb$.
\end{defn}

Note that $\Zeta_M(T) = \Zeta^0_M(T)$.
For an $\fO$-module~$U$, let $U^m =\fP^m U$
and write $U^m_n$ for the common value of $(U^m)_n$ and $(U_n)^m$.
Clearly, if $n \le m$, then $\ask{M_n^m}{V_n} = \ask{\{0\}} {V_n} = \card{V_n} = q^{nd}$.

\begin{prop}
  \label{prop:rescale}
  $\Zeta^m_{M^m}(T) = q^{dm} \dtimes \Zeta_M(T)$.
\end{prop}
\begin{proof}
  It suffices to show that
  $\ask{M^m_n}{V_n} = q^{dm} \dtimes \ask{M_{n-m}}{V_{n-m}}$
  for $n \ge m$. 
  Choose $\pi \in \fP\setminus \fP^2$.
  Observe that multiplication by $\pi^m$ induces an $\fO$-module
  isomorphism $M_{n-m} \to M^m_n$ and a monomorphism $V_{n-m} \to V_n$
  with image $V^m_n$.
  For $a \in M$,
  \begin{align*}
  \Ker(\pi^m a \on V_n)
    & = \bigl\{ \xx \in V_n : \xx (\pi^m a) \equiv 0 \pmod {W^n}\bigr\} \\
    & = \bigl\{ \xx \in V_n : \xx a \equiv 0 \pmod {W^{n-m}}\bigr\} \\
    & = \bigl\{ \xx \in V_n : \xx + V^{n-m} \in \Ker(a \on V_{n-m})\bigr\}
  \end{align*}
  has size $\card{\Ker(a \on V_{n-m})} \dtimes q^{dm}$.
  We conclude that
  \begin{align*}
    \ask{M^m_n}{V_n} & = \card{M^m_n}^{-1} \sum_{a \in M^m_n}
                     \card{\Ker(a \on V_n)} \\ &
     = \card{M_{n-m}}^{-1} \sum_{a \in M_{n-m}} \card{\Ker(a \on V_{n-m})} \dtimes
                     q^{dm} \\
                     & = q^{dm} \dtimes \ask{M_{n-m}}{V_{n-m}}.
                       \qedhere
  \end{align*}
\end{proof}

\section{Rationality of $\Zeta_M(T)$ and $p$-adic integration}
\label{s:ratint}

Unless otherwise stated, in this section, $K$ is a non-Archimedean
local field of arbitrary characteristic with valuation ring $\fO$.
Given a submodule $M \subset \Mat_{d\times e}(\fO)$,
we use the original definition of $\ask M V$ as well as the alternative
formula in Lemma~\ref{lem:WxM} to derive two types of expressions for
$\Zeta_M(T)$ in terms of $p$-adic integrals.

In \S\ref{ss:master_integral}, we describe a general setting for
rewriting certain generating functions as integrals.
By specialising to the case at hand, we obtain, in
\S\ref{ss:distributions}, 
two expressions for $\Zeta_M(T)$ (Theorem~\ref{thm:twoint})
in terms of functions $\kersize_M$ and $\orbsize_M$ that we introduce.
In \S\ref{ss:formulae_K_and_O}, we derive explicit formulae (in terms
of the absolute value of $K$ and polynomials over $\fO$) for these
functions.
These formulae serve two purposes. First, using established
rationality results from $p$-adic integration, they allow us to deduce
Theorem~\ref{thm:rational}.
Secondly, these formulae, in particular the one based on
$\orbsize_M$ (and hence on Lemma~\ref{lem:WxM}), lie at the heart of 
explicit formulae such as Proposition~\ref{prop:Mdxe} in \S\ref{s:examples_max}.

\subsection{Generating functions and $p$-adic integrals}
\label{ss:master_integral}

Let $Z$ be a free $\fO$-module of finite rank $d$ and let $U \subset Z$ be a
submodule.
By the elementary divisor theorem, there exists a unique 
sequence 
$(\lambda_1,\dotsc,\lambda_d)$ with $0 \le \lambda_1 \le \dotsb \le
\lambda_d \le \infty$ such that $Z/U
\approx \bigoplus_{i=1}^d \fO_{\lambda_i}$ as $\fO$-modules;
recall that $\fO_\infty = \fO$.
We call $(\lambda_1,\dotsc,\lambda_d)$ the \emph{submodule type} of $U$ within $Z$.
Recall that the \emph{isolator} $\iso_Z(U)$ of $U$ in $Z$ is the
preimage of the torsion submodule of $Z/U$ under the natural map $Z \onto
Z/U$.
Equivalently, $\iso_Z(U)$ is the smallest direct summand of $Z$ which
contains $U$.
Recall that $U$ is \emph{isolated} in $Z$ if and only if $U = \iso_Z(U)$;
this is equivalent to each $\lambda_i$ from above being either $0$ or $\infty$.
If $U$ is isolated in~$Z$, then we may naturally identify $U \otimes \fO_n$ 
with the image $U_n$ of $U \incl Z \onto Z_n := Z \otimes \fO_n$;
to see that this identification may fail if $U$ is not isolated,
consider e.g.\ $U = \fP \subset \fO = Z$ and $n = 1$.
If $U$ is isolated in~$Z$ and $U$ has rank $\ell$, say, then $\card{U_n} = q^{\ell n}$.
As we will now see, the general case is only slightly more complicated.

We let $\nu$ denote the valuation on $K$ with value group $\ZZ$.
Let $\abs{\dtimes}$ be the absolute value on $K$ with $\abs{\pi} = q^{-1}$
for $\pi \in \fP\setminus\fP^2$;
we write $\norm{A} = \sup(\abs a : a \in A)$.

\begin{lemma}
  \label{lem:card_Un}
  Let $U \subset Z$ be a submodule.
  Suppose that $U$ has submodule type $(\lambda_1,\dotsc,\lambda_\ell)$ within $\iso_Z(U)$.
  Let $U_n$ denote the image of the natural map $U \incl Z \onto Z_n
  := Z \otimes \fO_n$.
  Then $\card{U_n} = q^{\sum\limits_{i=1}^\ell (n-\min(\lambda_i,n))}$.
  In particular, for $y\in \fO$ with $\nu(y) = n$, $\card{U_n} 
  = \prod\limits_{i=1}^\ell \frac{\norm{\pi^{\lambda_i},y}}{\abs y}$.
\end{lemma}
\begin{proof}
  Clearly,
  $\card{U_n} = q^{\sum_{i=1}^\ell \min(n-\lambda_i,0)}$
  and the first identity follows from $\min(0,n-a) = n - \min(n,a)$;
  the second claim then follows immediately.
\end{proof}

The following result is concerned with generating functions
associated with a given family of ``weight functions'' $f_n \colon U_n \to
\RR_{\ge 0}$.
Given a free $\fO$-module $W$ of finite rank,
let $\mu_W$ denote the Haar measure on $W$ with $\mu_W(W) = 1$.

\begin{lemma}
  \label{lem:master_integral}
  Let $U \subset Z$ be a submodule.
  Suppose that $U$ has submodule type $(\lambda_1,\dotsc,\lambda_\ell)$ within $\iso_{Z}(U)$.
  Let $U_n$ denote the image of $U \incl Z \onto Z_n := Z \otimes \fO_n$
  and let $\pi_n\colon U \to U_n$ denote the natural map.
  Let $N \ge 0$ and, for $n \ge N$, let $f_n \colon U_n \to \RR_{\ge 0}$ be
  given.
  Define $$F\colon U \times \fP^N\setminus\{0\} \to \RR_{\ge 0}, \quad(\xx,y) \mapsto
  f_{\nu(y)}(\pi_{\nu(y)}(\xx))$$
  and extend $F$ to a map $U\times \fP^N\to \RR_{\ge 0}$ via
  $F(\xx,0) \equiv 0$.
  Let $$g \colon \fP^N \times \CC \to \RR_{\ge 0}, \quad (y,s) \mapsto
  \abs{y}^{s-\ell-1}
  \dtimes \prod\limits_{i=1}^\ell
  \norm{\pi^{\lambda_i},y},$$
  where we set $g(0,s) \equiv 0$.
  (Note that $g(y,s) = \abs{y}^{s-1} \dtimes \card{U_{\nu(y)}}$ for
  $y\not= 0$ by Lemma~\ref{lem:card_Un}.)

  Suppose that $\delta \ge 0$ satisfies
  $\sum\limits_{\bar\xx \in U_n}f_n(\bar\xx) = \mathcal O(q^{\delta n})$.
  Then, writing $t = q^{-s}$, for all $s \in \CC$ with $\Real(s) > \delta$, 
  \begin{equation}
    \label{eq:master_integral}
    \sum_{n=N}^\infty 
    \sum_{\bar\xx\in U_n}f_n(\bar\xx)
    t^n =
    (1-q^{-1})^{-1} \int_{U\times\fP^N} F(\xx,y)\dtimes g(y,s) \,\dd\mu_{U\times\fO}(\xx,y).
  \end{equation}
\end{lemma}
\begin{proof}
  First note that the left-hand side of \eqref{eq:master_integral}
  converges for $\Real(s) > \delta$.
  Further note that we may ignore the case $y = 0$ on the right-hand
  side as it occurs on a set of measure zero.

  Let $U^n = \Ker(\pi_n\colon U \onto U_n)$.
  Given $(\xx,y) \in U \times \fP^N \setminus \{0\}$ with $n :=
  \nu(y)$, the map $F$ is constant
  on the open set $(\xx + U^n)\times(y + \fP^{n+1})$;
  in particular, $F$ is measurable.

  Let $\cR_n \subset U$ be a complete and irredundant set of representatives
  for the cosets of $U^n$ and let $\cW_n = \fP^n\dtimes \fO^\times = \fP^n\setminus\fP^{n+1}$.
  By Lemma~\ref{lem:card_Un}, $\mu_U(U^n) = \card{U_n}^{-1} = \prod\limits_{i=1}^\ell \frac{\abs
    y}{\norm{\pi^{\lambda_i},y}}$ for any $y\in \cW_n$; moreover, $\mu_{\fO}(\cW_n) = (1-q^{-1})
  q^{-n}$.
  The claim now follows via
  \begin{align*}
    \int_{U\times \fP^N} \!\!\!\! \!\!\!
    F(\xx,y) \dtimes g(y,s) \,\dd\mu_{U\times\fO}(\xx,y)
    & =
      \sum_{n=N}^\infty \sum_{\xx \in \cR_n}
      \int_{(\xx + U^n)\times \cW_n}\!\!\! \!\!\!\!\!
      \!\!\!f_n(\pi_n(\xx)) \dtimes \underbrace{g(y,s)}_{=\card{U_n}
      \dtimes (qt)^n}\,\dd\mu_{U\times\fO}(\xx,y) \\
    & = (1-q^{-1})
      \sum_{n=N}^\infty \sum_{\bar\xx \in U_n}
      \!\!
      f_n(\bar\xx) \dtimes \mu(U^n\times \cW_n) \dtimes \card{U_n} \dtimes
      (qt)^n.\\
    & = (1-q^{-1})
      \sum_{n=N}^\infty \sum_{\bar\xx \in U_n}
      f_n(\bar\xx) t^n. \qedhere
  \end{align*}
\end{proof}

\begin{rem}
  The introduction of the additional variable $y$ to express a
  generating function as an integral in Lemma~\ref{lem:master_integral}
  mimics similar formulae of Jaikin-Zapirain~\cite[\S 4]{JZ06} and
  Voll~\cite[\S 2.2]{Vol10}. 
\end{rem}

\subsection{$\Zeta_M(T)$ and the functions $\kersize_M$ and $\orbsize_M$}
\label{ss:distributions}

Let $M \subset \Mat_{d\times e}(\fO)$ be a submodule, $V = \fO^d$, and
$W = \fO^e$.

\begin{defn}
  \label{d:OK}
  Define
  \begin{align*}
    \kersize_M\colon M \times \fO \to \NN_0 \cup \{ \infty\},
    \quad & (a,y) \mapsto \#{\Ker\Biggl(
      V \otimes \fO/(y)
      \xrightarrow{a \otimes {\fO}/{(y)}}
      W \otimes \fO/(y)
      \Biggr)}\text{ and} \\
    \orbsize_M\colon V \times \fO \to \NN_0\cup\{\infty\},
    \quad & (\xx,y) \mapsto \#\Img\Bigl(\xx M \incl W \onto W \otimes \fO/(y)\Bigr).
  \end{align*}
\end{defn}

Note that for $y \not= 0$, $1 \le \kersize_M(a,y) \le \abs{y}^{-d}$ and $1 \le \orbsize_M(\xx,y)
\le \abs{y}^{-\genidim(M)}$ (see Definition~\ref{d:genidim}); these
are the same estimates as in Proposition~\ref{prop:local_alpha}.

Using Lemmas~\ref{lem:WxM} and \ref{lem:master_integral}, we obtain
the following formulae for $\Zeta_M(t)$ (where $t = q^{-s}$).

\begin{thm}
  \label{thm:twoint}
  For $s \in \CC$ with $\Real(s) > d$,
  \begin{align*}
    (1-q^{-1}) \dtimes \Zeta_M(q^{-s})
    &=
      \!\!\!\int_{M \times \fO} \!\!\!\abs{y}^{s-1} \dtimes \kersize_M(a,y) \,\dd\mu_{M\times \fO}(a,y)
    = \!\!\!
      \int_{V \times \fO} \!\!\!
      \frac{\abs{y}^{s-d-1}} {\orbsize_M(\xx,y)} \,\dd\mu_{V \times\fO}(\xx,y).
  \end{align*}
\end{thm}
\begin{proof}
  The given formulae for $(1-q^{-1}) \dtimes\Zeta_M(q^{-s})$ 
  are based on $\ask {M_n}{V_n} = \card{M_n}^{-1}\dtimes$
  $\sum_{a\in M_n}\card{\Ker(a)}$ and 
  $\ask{M_n}{V_n} = \sum_{\bar\xx\in V_n} \card{\bar\xx M_n}^{-1}$ (Lemma~\ref{lem:WxM}), respectively.
  In detail, the first equality follows from Lemma~\ref{lem:master_integral} with
  $U = M$, $Z = \Mat_{d\times e}(\fO)$,
  $F(a,y) = \card{M_{\nu(y)}}^{-1} \kersize_M(a,y)$
  and $g(y,s) = \abs{y}^{s-1} \dtimes \card{M_{\nu(y)}}$ (for $y \not= 0$).
  For the second equality, use Lemma~\ref{lem:master_integral} with $U = Z = V$,
  $F(\xx,y) = \card{(\xx + y \fO^d) M_{\nu(y)}}^{-1} = \orbsize_M(\xx,y)^{-1}$
  and
  $g(y,s) = \abs{y}^{s-d-1}$ (for $y \not= 0$).
\end{proof}

\subsection{Explicit formulae for $\kersize_M$ and $\orbsize_M$}
\label{ss:formulae_K_and_O}

As before, let $M \subset \Mat_{d\times e}(\fO)$ be a submodule.
In order to use Theorem~\ref{thm:twoint} for theoretical
investigations or explicit computations of $\Zeta_M(T)$,
we need to produce sufficiently explicit formulae for
$\kersize_M(a,y)$ or $\orbsize_M(\xx,y)$.

\subsubsection{The sizes of kernels and images}
\label{ss:size_ker}

Let $a \in \Mat_{d\times e}(\fO)$ have rank $r$ over~$K$.
Fix $\pi \in \fP\setminus \fP^2$.
By the elementary divisor theorem, there are $0 \le \lambda_1 \le \dotsb
\le \lambda_r$, $u \in \GL_d(\fO)$, and $v\in \GL_e(\fO)$ such that
\begin{equation}
  \label{eq:snf}
  uav = \begin{bmatrix} \diag(\pi^{\lambda_1},\dotsc,\pi^{\lambda_r}) & 0 \\
    0 & 0 \end{bmatrix}.
\end{equation}
We call $(\lambda_1,\dotsc,\lambda_r)$ the \emph{equivalence type} of $a$.

For a set of polynomials $\ff(\YY)$, we write
$\ff(\yy) = \{ f(\yy) : f\in \ff\}$.

\begin{lemma}
  \label{lem:size_ker}
  Let $a \in \Mat_{d\times e}(\fO)$ have rank $r$ over $K$ and
  equivalence type $(\lambda_1,\dotsc,\lambda_r)$.
  Let $\ff_i(\YY) \subset \ZZ[\YY] := \ZZ[Y_{ij} : 1\le i\le d, \,1\le j \le e]$ be the set of non-zero $i\times i$ minors of the generic
  $d\times e$ matrix $[Y_{ij}] \in \Mat_{d\times e}(\ZZ[\YY])$.
  (Hence, $\ff_0(\YY) = \{ 1\}$.)
  Let $a_n \in \Mat_{d\times e}(\fO_n)$ be the image of the matrix $a$ under the natural
  map $\Mat_{d\times e}(\fO) \onto \Mat_{d\times e}(\fO_n)$.
Then:
  \begin{enumerate}
  \item \label{lem:size_ker1}
    $\norm{\ff_i(a)} = q^{-\lambda_1-\dotsb-\lambda_i}$ for $0 \le i \le r$.
  \item \textup{(Cf.\ \cite[\S 2.2]{Vol10})}
    \label{lem:size_ker2}
    $\frac{\norm{\ff_{i-1}(a)}}{\norm{\ff_i(a) \cup \pi^n \ff_{i-1}(a)}} =
    q^{\min(\lambda_i,n)}$ for $1 \le i \le r$.
  \item \label{lem:size_ker3}
    $\card{\Ker(a_n)} = q^{\min(\lambda_1,n) + \dotsb +
      \min(\lambda_r,n) + (d-r)n}$.
  \item \label{lem:size_ker4}
    $\card{\Img(a_n)} = q^{rn - (\min(\lambda_1,n) + \dotsb + \min(\lambda_r,n))}$.
  \end{enumerate}
\end{lemma}
\begin{proof}
  The first part is elementary linear algebra.
  Part (\ref{lem:size_ker2}) then follows from
  $$q^{-\lambda_1 - \dotsb - \lambda_{i-1} + \min(\lambda_1 + \dotsb +
    \lambda_i,\lambda_1 + \dotsb + \lambda_{i-1} + n)} = q^{\min(\lambda_i,n)}.$$
  For $0\le i \le n$, the  map $\fO_n \to \fO_n$ given by multiplication by $\pi^i$
  has kernel and image size $q^i$ and $q^{n-i}$, respectively.
  Parts (\ref{lem:size_ker3})--(\ref{lem:size_ker4}) thus follow from equation \eqref{eq:snf}.
\end{proof}

\subsubsection{A formula for $\kersize_M$}

We use Lemma~\ref{lem:size_ker} in order to derive a formula for $\kersize_M(a,y)$.

\begin{defn}
  \label{d:genrank}
  The \emph{generic element rank} of $M$ is
  $\genrank(M) := \max\limits_{a \in M} \rank_K(a)$.
\end{defn}

By the following, $\genrank(M)$ is the rank of a ``generic'' matrix
in $M$ in any meaningful sense.

\begin{prop}
  \label{prop:genrank}
  Let $(a_1,\dotsc,a_\ell)$ be an $\fO$-basis of $M$ and let
  $\lambda_1,\dotsc,\lambda_\ell$ be algebraically independent over $K$. Then:
  \begin{enumerate}
  \item \label{prop:genrank1}
    $\genrank(M) = \rank_{K(\lambda_1,\dotsc,\lambda_\ell)}(\lambda_1 a_1 + \dotsb + \lambda_\ell a_\ell)$.
  \item \label{prop:genrank2}
    $\mu_M(\{a \in M : \rank_K(a) < \genrank(M)\}) =0$.
  \end{enumerate}
\end{prop}
\begin{proof}
  Let $r$ denote the right-hand side in (\ref{prop:genrank1}).
  Then $r$ is the largest number such that some $r\times r$ minor,
  $m(\lambda_1,\dotsc,\lambda_\ell)$ say, of $\lambda_1 a_1 + \dotsb +
  \lambda_\ell a_\ell$ is non-zero.
  In particular, $\genrank(M) \le r$.
  Conversely, since $m(\lambda_1,\dotsc,\lambda_\ell) \not= 0$, we
  find $c_1,\dotsc,c_\ell \in \fO$ with
  $m(c_1,\dotsc,c_\ell) \not= 0$ whence $\genrank(M) \ge
  r$.
  Finally, the well-known fact (provable using induction and Fubini's
  theorem) that the zero locus of a non-zero polynomial over $K$ has
  measure zero implies~(\ref{prop:genrank2}).
\end{proof}

Excluding a set of measure zero, we thus obtain the following 
formula for $\kersize_M(a,y)$.

\begin{cor}
  \label{cor:formula_K}
  Let $N = \{ a \in M : \rank_K(a) < \genrank(M)\}$
  and let $\ff_i(\YY)$ be the set of non-zero $i\times i$ minors of
  the generic $d\times e$ matrix.
  Then for all $a \in M\setminus N$ and $y \in \fO\setminus\{0\}$,
  \[
  \kersize_M(a,y) =
  \abs{y}^{\genrank(M)-d} \dtimes \prod_{i=1}^{\genrank(M)}
  \frac{\norm{\ff_{i-1}(a)}}{\norm{\ff_i(a) \cup y \ff_{i-1}(a)}}.
  \]
\end{cor}
\begin{proof}
  Immediate from Lemma~\ref{lem:size_ker}(\ref{lem:size_ker2})--(\ref{lem:size_ker3}).
\end{proof}

Hence, using Theorem~\ref{thm:twoint}, we conclude that
\begin{equation}
  \label{eq:int_K_minors}
 \Zeta_M(q^{-s})
  =   (1-q^{-1})^{-1}
  \!\!\!\!\!\int_{M \times \fO} \!\!\!   
  \abs{y}^{s + \genrank(M)-d-1} \dtimes \!\!
  \prod_{i=1}^{\genrank(M)}
  \!\!\!
  \frac{\norm{\ff_{i-1}(a)}}{\norm{\ff_i(a) \cup y \ff_{i-1}(a)}}
 \,\dd\mu_{M\times \fO}(a,y).
\end{equation}

\subsubsection{Rationality and variation of the place}
\label{sss:rationality}

As in the proof of Proposition~\ref{prop:genrank}, 
we may replace the $\ff_i(\YY)$ in \eqref{eq:int_K_minors} by
polynomials in a chosen system of coordinates of $M$.
We may thus interpret the integral in \eqref{eq:int_K_minors} 
as being defined in terms of valuations of polynomial expressions in
$\dim_K(M\otimes K)+ 1$ variables.
Integrals of this form have been studied extensively.
In particular, using well-known results going back to work of Igusa
and Denef (see \cite{Den91a}), initially for a single polynomial and
later extended to families of polynomials (see, in
particular, work of du~Sautoy and Grunewald~\cite{dSG00}, 
Veys and Z{\'u}{\~n}iga-Galindo~\cite{VZG08}, and Voll~\cite{Vol10}),
we obtain the following two results,
the first of which implies and refines Theorem~\ref{thm:rational}.

\begin{thm}
  \label{thm:rational_plus}
  Let $\fO$ be the valuation ring of a local field of characteristic
  zero and let $M \subset \Mat_{d\times e}(\fO)$ be a submodule.
  Then $\Zeta_M(T) \in \QQ(T)$.
  More precisely, there exist $f(T) \in \ZZ[T]$, non-zero $(a_1,b_1),
  \dotsc, (a_r,b_r)$ with $(a_i,b_i) \in \ZZ\times \NN_0$, and $m \in \NN_0$ such that
  \[
  \pushQED{\qed}
  \Zeta_M(T) = \frac{f(T)}{q^m (1-q^{a_1}T^{b_1}) \dotsb
    (1-q^{a_r}T^{b_r})}.
  \qedhere
  \popQED
  \]
\end{thm}

Moreover, in a global setting, the dependence of Euler factors on
the place is as follows.

\begin{thm}
  \label{thm:denef}
  Let $k$ be a number field with ring of integers $\fo$.
  Recall the notation from~\S\ref{ss:global_alpha}.
  Let $M \subset \Mat_{d \times e}(\fo)$ be a submodule.
  There are separated $\fo$-schemes $V_1,\dotsc,V_r$ of finite
  type and rational functions $W_1(X,T),\dotsc,W_r(X,T) \in \QQ(X,T)$
  (which can be written over denominators of the same shape as those in
  Theorem~\ref{thm:rational_plus}) such that the following holds:
  for almost all $v \in \Places_k$,
  \begin{equation}
    \label{eq:denef}
  \pushQED{\qed}
  \Zeta_{M_v}(T) = \sum_{i=1}^r \# V_i(\RF_v) \dtimes W_i(q_v,T).
  \qedhere
  \popQED
  \end{equation}
\end{thm}

\subsubsection{Local fields of positive characteristic}
\label{sss:positive}

Suppose that $M \subset\Mat_{d\times e}(\fO)$ is a submodule as before
but that $K = \RF\llb z\rrb$ is a local field of \itemph{positive} characteristic.
Then the techniques cited above to establish rationality in
Theorem~\ref{thm:rational_plus} do not apply to $\Zeta_M(T)$ 
and indeed, the author does not know if $\Zeta_M(T)$ is necessarily rational
in positive characteristic.
By combining Igusa's original proof of the rationality of his local zeta function
(see \cite[\S 1.3]{Den91a} for a modern account)
and ideas of du~Sautoy and Grunewald~\cite[\S 2]{dSG00}, we obtain
the following sufficient condition for rationality of ask zeta functions over $\fO$:
if every hypersurface embedded inside some affine space over $K$
admits an embedded resolution of singularities over $K$, then $\Zeta_M(T)$ is
rational for all modules $M \subset \Mat_{d\times e}(\fO)$.
The status of resolution of singularities in positive characteristic is
presently unresolved; see e.g.\ \cite{Hau10}.

In contrast, to such unresolved issues, ask zeta functions in ``large''
positive characteristic arising from global models in characteristic 
zero are amenable to existing techniques. 
Namely, by applying powerful model-theoretic transfer principles such as
\cite[Thm~9.2.4]{CL10} to our integrals, we obtain the following.

\begin{thm}
  \label{thm:transfer}
  Let $k$ be a number field with ring of integers $\fo$.
  Recall the notation from~\S\ref{ss:global_alpha}.
  Let $M \subset \Mat_{d \times e}(\fo)$ be a submodule.
  Then for almost all $v \in \Places_k$,
  $\Zeta_{M_v}(T) = \Zeta_{M\otimes_{\fo}\RF_v\llb z\rrb}(T)$.
  In particular, $\Zeta_{M\otimes_{\fo}\RF_v\llb z\rrb}(T)$ is rational for
  almost all $v\in \Places_k$.
  \qed
\end{thm}

We note that Theorems~\ref{thm:denef}--\ref{thm:transfer} both behave well under local base
extensions; cf.~\cite[Thm~2.3]{stability} and \cite[Rem.~1.6]{Vol17}.

\subsubsection{A formula for $\orbsize_M$}
\label{sss:O}

As in the case of $\kersize_M$, we can produce a formula for $\orbsize_M$.

Let $\XX = (X_1,\dotsc,X_d)$ be algebraically independent over $K$.
Let $C(\XX) \in \Mat_{\ell\times e}(\fO[\XX])$ with
$\fO[\XX]^d C(\XX) =  \XX \dtimes (M\otimes\fO[\XX])$.
For example,
we may choose generators $a_1,\dotsc,a_\ell$ of $M$
as an $\fO$-module and take
\[
C(\XX) = \begin{bmatrix} \XX a_1 \\ \vdots \\ \XX a_\ell
\end{bmatrix} \in
\Mat_{\ell\times e}(\fO[\XX]).
\]
Let $\bm g_i(\XX)$ be the set of non-zero $i\times i$ minors of $C(\XX)$.
Note that if $\xx \in V$, then $\xx M$ is the row span of $C(\xx)$ over $\fO$
so that, in particular, $\dim_K(\xx M \otimes K) = \rank_K(C(\xx))$.

Recall the definition of $\genidim(M)$ in Definition~\ref{d:genidim}.
The following is proved in the same way as
Proposition~\ref{prop:genrank}.

\begin{prop}
  \label{prop:genidim}
  Let $\XX = (X_1,\dotsc,X_d)$ be algebraically independent over $K$.
  Then:
  \begin{enumerate}
  \item $\genidim(M) =  \dim_{K(\XX)}\bigl( \XX \dtimes (M \otimes
    K(\XX))\bigr) = \rank_{K(\XX)}(C(\XX))$.
  \item
    $\mu_V(\{ \xx \in V: \dim_K(\xx M \otimes K) < \genidim(M)\}) =
    0$.
    \qed
  \end{enumerate}
\end{prop}

The following analogue of Corollary~\ref{cor:formula_K} is obtained
using Lemma~\ref{lem:size_ker}(\ref{lem:size_ker2}),(\ref{lem:size_ker4}).

\begin{cor}
  \label{cor:formula_O}
  Let $Z = \{ \xx \in V : \dim_K(\xx M\otimes K) < \genidim(M)\}$.
  Then for all $\xx \in V\setminus Z$ and $y \in \fO \setminus \{0\}$,
  \begin{equation}
    \label{eq:O_polynomials}
  \pushQED{\qed}
  \orbsize_M(\xx,y) = \abs{y}^{-\genidim(M)} \prod_{i=1}^{\genidim(M)} \frac{\norm{\bm g_i(\xx) \cup y \bm
      g_{i-1}(\xx)}}{\norm{\bm g_{i-1}(\xx)}}.
  \qedhere
  \popQED
  \end{equation}
\end{cor}

Theorem~\ref{thm:twoint} thus provides us with the following
counterpart of \eqref{eq:int_K_minors}:
\begin{equation}
  \label{eq:int_O_minors}
  \Zeta_M(q^{-s})
  =   (1-q^{-1})^{-1}
  \!\!\!\!\!\int_{V \times \fO} \!\!\!   
  \abs{y}^{s + \genidim(M) - d - 1} \dtimes \!\!
  \prod_{i=1}^{\genidim(M)}
  \!\!\!
  \frac{\norm{\bm g_{i-1}(\xx)}}{\norm{\bm g_i(\xx) \cup y \bm g_{i-1}(\xx)}}
 \,\dd\mu_{V\times \fO}(\xx,y).
\end{equation}

Despite the essentially identical shapes of the integrals in
\eqref{eq:int_K_minors} and \eqref{eq:int_O_minors},
either type might be vastly more useful for explicit computations
of specific examples.
In particular,~\S\ref{s:examples_max} is concerned with examples of
$\Zeta_M(T)$ that can be easily computed using~\eqref{eq:int_O_minors}
and \S\ref{s:examples_min} considers the analogous situations for
\eqref{eq:int_K_minors}.
A very similar phenomenon was exploited by O'Brien and Voll~\cite[\S
5]{O'BV15} in their enumeration of conjugacy classes of certain relatively
free $p$-groups.

\begin{rem}
  We note that the integrals in
  \eqref{eq:int_K_minors}--\eqref{eq:int_O_minors} are almost of the
  same shape as those in \cite[Eqn~(1.4)]{AKOV13}.
  These similarities can be clarified further by rewriting our integrals
  slightly; see the proof of Theorem~\ref{thm:feqn} below.
  We further that the role of the matrix $C(\XX)$ here 
  is similar to that of $A(\mathbf X)$ in \cite[Def.\ 2.1]{O'BV15}.
\end{rem}

\subsection{Projectivisation}
\label{ss:proj}

For later applications, in the following, we record ``projective''
versions of the integrals in Theorem~\ref{thm:twoint} in the spirit of
Voll's formalism~\cite[\S 2.2]{Vol10}, as rewritten by Avni et
al.~\cite[\S 3.2]{AKOV13}.

Let $M \subset \Mat_{d\times e}(\fO)$ be a submodule and
$V = \fO^d$.
The following lemma is elementary.

\begin{lemma}
  \label{lem:OK_homg}
  Let $\xx \in V$, $a\in M$, $y\in \fO\setminus\{0\}$, and $\pi \in \fP\setminus\fP^2$.
  Then:
  \begin{enumerate}
  \item \label{lem:OK_homg1}
    $\orbsize_M(\pi \xx,\pi y) = \orbsize_M(\xx,y)$.
  \item \label{lem:OK_homg2}
    $\kersize_M(\pi a,\pi y) = q^d \dtimes \kersize_M(a,y)$.
  \end{enumerate}
\end{lemma}
\begin{proof}
  Let $W:= \fO^e$.
  \begin{enumerate}
  \item 
    For a submodule $U\subset W$ and $z\in \fO$,
    let $U_z$ denote the image of $U$ in $W\otimes \fO/(z)$.
    Since multiplication by $\pi$ induces an isomorphism
    $U_z \approx (\pi U)_{\pi z}$, the claim follows by taking $U = \xx M$ and
    $z = y$.
  \item
    Let $\xx \in V$.
    Then $\xx (\pi a) \equiv 0 \pmod {\pi y W}$ if and only if $\xx a\equiv 0
    \pmod {y W}$.
    Clearly, each residue class modulo $y W$ of such elements has precisely $q^d$ lifts
    modulo $\pi y W$.
\qedhere
  \end{enumerate}
\end{proof}

\begin{prop}
  \label{prop:proj_twoint}
  Let $\ell := \dim_K(M\otimes K)$.
  Then:
  \begin{enumerate}
  \item
  \label{prop:proj_twoint1}
  $\displaystyle  (1-q^{-s}) \dtimes \Zeta_M(q^{-s})  =
  {1 + (1-q^{-1})^{-1} \int\limits_{(V\setminus\fP V)\times
      \fP} \frac{\abs{y}^{s-d-1}}{\orbsize_M(\xx,y)} \,\dd\mu_{V\times\fO}(\xx,y)}.$
  \item
  \label{prop:proj_twoint2}
  $\displaystyle (1-q^{d-\ell-s}) \dtimes \Zeta_M(q^{-s})
 =  {1 + (1-q^{-1})^{-1} \int\limits_{(M\setminus\fP M) \times\fP} \abs{y}^{s-1} \kersize_M(a,y) \,\dd\mu_{M\times\fO}(a,y)}.$
  \end{enumerate}
\end{prop}
\begin{proof}
  Write $t := q^{-s}$.
  \begin{enumerate}
  \item
    First, $\orbsize_M(\xx,y) = 1$ for $\xx\in V$ and $y\in \fO^\times$.
    In the following,
    we write
    $\omicron$ as a shorthand for $\frac{\abs{y}^{s-d-1}} {\orbsize_M(\xx,y)} \,\dd\mu_{V \times \fO}(\xx,y)$.
    Using Lemma~\ref{lem:OK_homg} and a change of variables, 
    we find that $\int\limits_{\fP V \times \fP}\!\!\omicron = t \dtimes \int\limits_{V\times\fO} \!\!\omicron$.
    The claim then follows from Theorem~\ref{thm:twoint} and
    \begin{alignat*}{5}
      \int_{V\times \fO} \!\! \omicron & = \int_{V\times \fO^\times} \!\!\omicron  &\quad +\quad &
      \int_{(V\setminus\fP V)\times \fP} \!\!\omicron &\quad+\quad&
      \int_{\fP V\times \fP}\!\! \omicron\\
      & = (1-q^{-1})  &\quad+\quad & \int_{(V\setminus \fP V)\times \fP} \!\!
      \omicron &\quad + \quad& t \dtimes \int_{V\times\fO} \!\!\omicron.
    \end{alignat*}
  \item
     For $a \in M$ and $y \in \fO^\times$,
     $\kersize_M(a,y) = 1$.
     One may then proceed similarly to~(\ref{prop:proj_twoint1}) using 
     $$\int\limits_{\fP M\times \fP} \kappa = q^{d-\ell}t\dtimes
     \int\limits_{M\times\fO}\kappa,$$
     where $\kappa = \abs{y}^{s-1} \kersize_M(a,y) \,\dd\mu_{M\times \fO}(a,y)$.
     \qedhere
  \end{enumerate}
\end{proof}

\subsection{Local functional equations and global analytic properties}
\label{s:feqn}

Functional equations under ``inversion of the prime'' 
are a common (but not universal) phenomenon in the theory of local
zeta functions.
Denef and Meuser~\cite{DM91} showed that for a homogeneous polynomial over
a number field, almost all of its associated local Igusa zeta
functions satisfy such a functional equation.
Vastly generalising their result, Voll~\cite{Vol10} established
functional equations for numerous types of zeta functions arising in
asymptotic algebra and expressible in terms of $p$-adic integrals.
For further positive results establishing such functional equations,
see, in particular, work of du~Sautoy and Lubotzky~\cite{dSL96},
Voll~\cite{Vol17}, Avni et al.~\cite[\S 4]{AKOV13}, and
Stasinski and Voll~\cite[Thm~A]{SV14}.
Using the formalism developed above, we may deduce the following; recall the
notation from \S\ref{ss:global_alpha}.

\begin{thm}
  \label{thm:feqn}
  Let $k$ be a number field with ring of integers $\fo$.
  Let $M \subset \Mat_{d\times e}(\fo)$ be a submodule.
  Then for almost all $v \in \Places_k$,
  \[
  \Zeta_{M_v}(T) \Biggm\vert_{(q_v,T) \to (q_v^{-1},T^{-1})} = (-q_v^d T) \dtimes \Zeta_{M_v}(T).
  \]
\end{thm}
\begin{rem}
  \quad
  \begin{enumerate}
  \item
    The operation ``$q_v \to q_v^{-1}$'' can be unambiguously defined in
    terms of an arbitrary formula of the form~\eqref{eq:denef};
    see \cite{DM91,Vol10} and  cf.\ \cite[Cor.\ 4.3]{stability}.
    If $\Zeta_{M_v}(T) = W(q_v,T)$ for almost all
    $v \in \Places_k$ and some $W(X,T) \in \QQ(X,T)$,
    then Theorem~\ref{thm:feqn} asserts that
    $W(X^{-1},T^{-1}) = (-X^d T) \dtimes W(X,T)$;
    see~\cite[\S 4]{stability} and cf.\ \eqref{eq:Mde_FEqn}.
  \item 
    Using Theorem~\ref{thm:denef}--\ref{thm:transfer},
    Theorem~\ref{thm:feqn} also establishes, for
    almost all $v\in\Places_k$, a functional equation for
    $\Zeta_{M\otimes_{\fo}\RF_v\llb z\rrb}(T)$;
    cf.\ \cite[Cor.\ 1.3]{Vol17}.
  \end{enumerate}
\end{rem}
\begin{proof}[Proof of Theorem~\ref{thm:feqn}]
  We use Voll's results from \cite[\S 2.1]{Vol10}.
  Let $K = k_v$ for $v \in \Places_k$ and let $\fO$ be as before.
  Let $H_v(s)$ denote the right-hand side in
  Proposition~\ref{prop:proj_twoint}(\ref{prop:proj_twoint1}).
  Using the surjection $\GL_d(\fO) \to V\setminus \fP V$ which sends a
  matrix to its first column, we rewrite $H_v(s)$ in terms of an integral over
  $\GL_d(\fO)\times \fP$; cf.\ \cite[\S 4]{AKOV13}.
  In the setting of the explicit formula for $\orbsize_{M_v}(\xx,y)$
  derived in \S\ref{sss:O}, we may assume that $C(\XX)$ is a matrix of
  linear forms whence each $\bm g_i(\XX)$ consists of homogeneous
  polynomials of degree $i$.
  This allows us to use
  \cite[Cor.\ 2.4]{Vol10} which shows that
  $H_v(s) \!\bigm\vert_{q_v\to q_v^{-1}} = q^d H_v(s)$
  for almost all $v\in \Places_k$.
\end{proof}

Based on work of du~Sautoy and Grunewald~\cite{dSG00},
Duong and Voll~\cite{DV17} studied analytic properties of Euler products of 
functions of the same form as the right-hand sides
in Proposition~\ref{prop:proj_twoint}(\ref{prop:proj_twoint1})--(\ref{prop:proj_twoint2}).
In particular, their findings allow us to deduce the following.

\begin{thm}[{Cf.\ \cite[Thm~A]{DV17}}]
  Let $k$ be a number field with ring of integers $\fo$.
  Let $M \subset \Mat_{d\times e}(\fo)$ be a submodule and $V = \fo^d$.
  \begin{enumerate}
  \item The abscissa of convergence $\alpha_M$ of $\zeta_M(s)$ is a rational number.
  \item There exists $\delta > 0$ such that $\zeta_M(s)$ admits
    meromorphic continuation to $\{ s \in \CC : \Real(s) > \alpha_M -
    \delta\}$.
    This continued function has a pole of order $\beta_M$, say, at $s =
    \alpha_M$ but no other poles on 
    the line $\Real(s) = \alpha_M$.
\item
  $\alpha_M$ and $\beta_M$ are (and $\delta$ can be chosen to be) invariant
  under base change of $M$ from $\fo$ to the ring of integers of an arbitrary
  finite extension of $k$.
  \end{enumerate}
\end{thm}

\subsection{Reduced and topological ask zeta functions}
\label{ss:redtop}

Let $k$ be a number field with ring of integers $\fo$.
Recall the notation from~\S\ref{ss:global_alpha}.
Given a suitable family of local zeta functions indexed by places $v\in
\Places_k$, associated ``reduced'' and ``topological'' zeta functions are
obtained by passing to two different limits ``$q_v \to 1$''.
The original topological zeta functions of Denef and Loeser~\cite{DL92} are
singularity invariants attached to polynomials.
Later, du~Sautoy and Loeser~\cite{dSL04} defined topological subobject zeta
functions of algebraic structures.
Reduced subobject zeta functions were introduced by Evseev~\cite{Evs09}.
Topological and reduced representation zeta functions of unipotent groups were
studied by the author in \cite{unipotent} and \cite[\S 7]{padzeta}, respectively.

The techniques from $p$-adic integration used above are similar to those
employed in the study of representation zeta functions of unipotent groups.
As a consequence, we immediately obtain adequate notions of reduced 
and topological ask zeta functions which we now briefly discuss.
Let $M \subset \Mat_{d\times e}(\fo)$ be a submodule.

\paragraph{Topological ask zeta functions.}
Informally, the topological ask zeta function $\zeta_M^{\mathrm{top}}(s) \in \QQ(s)$
of $M$ is the constant term of $(1-q_v^{-1})\zeta_{M_v}(s)$ as a series in $q_v-1$;
for a rigorous definition, combine the formalism 
developed in \cite[\S 5]{topzeta} (and summarised in \cite[\S 4.2]{spp1489}),
Proposition~\ref{prop:proj_twoint}, and \cite[Pf of Lem.\ 3.4]{unipotent}.
For example, Proposition~\ref{prop:Mdxe} implies that
\[
\zeta^{\mathrm{top}}_{\Mat_{d\times e}(\ZZ)}(s) = \frac{s+e}{s(s-d+e)}.
\]
We note that, as in the case of subobject~\cite[Prop.\ 5.19]{topzeta} and representation zeta
functions~\cite[Prop.\ 4.3]{unipotent},
the topological ask zeta function of $M$ only depends on $M \otimes_{\fo} \bar
k$, where $\bar k$ is an algebraic closure of $k$.

\paragraph{Reduced ask zeta functions.}
Informally, the reduced ask zeta function $\Zeta_M^{\mathrm{red}}(T) \in
\QQ(T)$ is obtained from the formal power series $\Zeta_{M_v}(T)$ by applying
a limit ``$q_v\to 1$'' to each coefficient.
In the present context, this process can be formalised just as in the case of
representation zeta functions of unipotent groups (see \cite[\S 7]{padzeta}).
Moreover, Proposition~\ref{prop:proj_twoint} and a variation of \cite[Pf of Thm~7.3]{padzeta}
(which relies heavily on arguments due to Duong and
Voll~\cite{DV17}) show that in fact $\Zeta_M^{\mathrm{red}}(T) = 1/(1-T)$ for any $M$.
This is intuitively plausible:
if $\tilde M \subset \Mat_{d\times e}(\fO_n)$
is a submodule, then the group $(\fO/\fP)^\times$ acts freely on $\tilde M
\setminus \{0\}$ and preserves kernels whence
${\card{\tilde M}} \dtimes \ask {\tilde M} {} \equiv {\card{\tilde V}} \pmod{(q-1)}$,
where $\tilde V = \fO_n^d$.
In particular, one would expect any reasonable limit of $\ask {\tilde M}{}$ as
``$q \to 1$'' to be $1$.

\section{Full matrix algebras, classical Lie algebras, and relatives} 
\label{s:examples_max}

In this section, let $\fO$ be the valuation ring of a non-Archimedean local
field of arbitrary characteristic.
Apart from proving Proposition~\ref{prop:Mdxe},
we compute examples of $\Zeta_{\fg}(T)$,
where $\fg$ ranges over various infinite families of matrix Lie algebras.
At the heart of these computations lies the notion of ``$\orbsize$-maximality''
introduced in \S\ref{ss:O_max}.

\subsection{$\orbsize$-maximality}
\label{ss:O_max}

Let $M \subset \Mat_{d\times e}(\fO)$ be a submodule and $V = \fO^d$.
As we will now see, $\orbsize_M(\xx,y)$ is (generically) as large as
possible if and only if $\Zeta_M(T)$ coincides with $\Zeta_{\Mat_{d\times\genidim(M)}}(T)$.

\begin{lemma}
  \label{lem:O_bound}
  Let $\xx \in V$ and $y\in \fO\setminus\{0\}$.
  Then $\orbsize_M(\xx,y) \le \abs{y}^{-\genidim(M)}\norm{\xx,y}^{\genidim(M)}$.
\end{lemma}
\begin{proof}
  Let $C(\XX) \in \Mat_{\ell\times e}(\fO[\XX])$ with
  $\fO[\XX]^d C(\XX) =  \XX \dtimes (M\otimes\fO[\XX])$.
  We may assume that $C(\xx) \not= 0$ since otherwise $\orbsize_M(\xx,y) = 1$.
  As $0_{\fO^d} \dtimes M = \{ 0_{\fO^e}\}$, the constant terms of
  all non-zero polynomials in $C(\XX)$ vanish whence
  $\norm{C(\xx)} \le \norm{\xx}$.
  Thus, if $C(\xx)$ has equivalence type
  $(\lambda_1,\dotsc,\lambda_r)$,
  then $\abs{\pi^{\lambda_i}} \le \abs{\pi^{\lambda_1}} =
  \norm{C(\xx)} \le \norm{\xx}$ for
  $1\le i \le r$.
  Define $m$ and $n$ via $q^{-m} = \norm{\xx}$ and $n = \nu(y)$.
  Then, by Lemma~\ref{lem:size_ker}(\ref{lem:size_ker4}),
  \begin{align*}
  \orbsize_M(\xx,y) &= q^{rn - \sum_{i=1}^r \min(\lambda_i,n)}
                      \le q^{\genidim(M)(n-\min(m,n))}
                      = \abs{y}^{-\genidim(M)}\norm{\xx,y}^{\genidim(M)}.\qedhere
  \end{align*}
\end{proof}

The above inequality is sharp (cf.\ the comments after Proposition~\ref{prop:local_alpha}):
\begin{lemma}
  \label{lem:O_Mdxe}
  Let $\xx \in V$ and $y \in \fO\setminus\{0\}$.
  Then $\orbsize_{\Mat_{d\times e}(\fO)}(\xx,y) = \abs{y}^{-e}\norm{\xx,y}^e$.
\end{lemma}
\begin{proof}
  Let $(\ee_1,\dotsc,\ee_d) \subset \fO^d$ be the standard basis.
  We may assume that $\xx \not= 0$.
  As $\xx \Mat_{d\times e}(\fO)$ is generated by $\{ \gcd(x_1,\dotsc,x_d)
  \ee_i : i=1,\dotsc,e\}$, the claim follows from Lemma~\ref{lem:size_ker}(\ref{lem:size_ker4}).
\end{proof}

\begin{lemma}
  \label{lem:Omax_null}
  The following are equivalent:
  \begin{enumerate}
  \item \label{lem:Omax_null1}
    $\orbsize_M(\xx,y) =
    \abs{y}^{-\genidim(M)}\norm{\xx,y}^{\genidim(M)}$ for all
    $(\xx,y) \in V \times \fO$ outside a set of measure zero.
  \item \label{lem:Omax_null2}
    $\orbsize_M(\xx,y) = \abs{y}^{-\genidim(M)}\norm{\xx,y}^{\genidim(M)}$ for all
    $\xx \in V$ and all $y \in \fO\setminus\{0\}$.
  \end{enumerate}
\end{lemma}
\begin{proof}
  $\orbsize_M$ is locally constant on $V\times(\fO\setminus\{0\})$ so
  (\ref{lem:Omax_null1}) implies (\ref{lem:Omax_null2}); the converse is clear.
\end{proof}

\begin{defn}
  We say that $M$ is \emph{$\orbsize$-maximal} if it satisfies one of the two
  equivalent conditions in the preceding lemma.
\end{defn}

Proposition~\ref{prop:Mdxe} serves as a blueprint for $\Zeta_M(T)$
whenever $M$ is $\orbsize$-maximal:

\begin{cor}
  \label{cor:O_max}
  $M$ is $\orbsize$-maximal if and only if
  $\Zeta_M(T)
  = \Zeta_{\Mat_{d\times \genidim(M)}(\fO)}(T)$.
\end{cor}
\begin{proof}
  The ``only if'' part follows by combining \eqref{eq:int_O_minors} and
  Lemma~\ref{lem:O_Mdxe}.
  Conversely, suppose that $\orbsize_M(\xx,y) <
  \abs{y}^{-\genidim(M)}\norm{\xx,y}^{\genidim(M)}$
  for some $\xx\in V$ and $y\in \fO\setminus\{0\}$.
  Using the fact that both sides of this inequality are locally constant
  functions of $(\xx,y)$,
  Lemmas~\ref{lem:O_bound}--\ref{lem:O_Mdxe},
  and Theorem~\ref{thm:twoint},
  we conclude that for sufficiently large $s\in\RR$,
  $\zeta_M(s) > \zeta_{\Mat_{d\times \genidim(M)}(\fO)}(s)$.
  In particular, $\Zeta_M(T) \not= \Zeta_{\Mat_{d\times \genidim(M)}(\fO)}(T)$.
\end{proof}

The following is a ``projective'' characterisation of $\orbsize$-maximality.
\begin{lemma}
  \label{lem:O_max_proj}
  $M$ is $\orbsize$-maximal if and only if 
  $\orbsize_M(\xx,y) = \abs{y}^{-\genidim(M)}$
  for all $(\xx,y) \in (V\!\setminus\!\fP V)\times \fP$ outside a set of
  measure zero.
\end{lemma}
\begin{proof}
  Necessity of the given condition being clear,
  suppose that 
  $\orbsize_M(\xx,y) = \abs{y}^{-\genidim(M)}$
  for all $(\xx,y) \in \left( (V\!\setminus\!\fP V)\times \fP\right)\setminus Z$, where
  $Z\subset V\times\fO$ has measure zero.
  Choose $\pi\in\fP\setminus\fP^2$.
  For each $n \ge 0$, we recursively define a set $Z^{(n)}\subset V\times \fO$
  of measure zero such that $\orbsize_M(\pi^n \xx,y) =
  \abs{y}^{-\genidim(M)}\norm{\pi^n,y}^{\genidim(M)}$
  for all $(\xx,y)\in \left((V\!\setminus\!\fP V)\!\times\!
    \fO\right)\!\setminus\! Z^{(n)}$.
  By assumption and since $\orbsize_M(\xx,y) = 1$ for all $\xx\in V$ and $y\in \fO^\times$, 
  we may take $Z^{(0)} := Z$.
  Suppose that $Z^{(n)}$ has been defined with the aforementioned properties and let
  $Z^{(n+1)} := \{ (\xx,y) \in V\times\fO : (\xx,\pi^{-1}y) \in Z^{(n)}\}$;
  note that $Z^{(n+1)}$ has measure zero.
  Let $(\xx,y)\in \left((V\!\setminus\!\fP V)\!\times\!
    \fO\right)\!\setminus\! Z^{(n+1)}$.
  We may assume that $y\in \fP$, say $y = \pi z$.
  Then, since $(\xx,z)\not\in Z^{(n)}$, using
  Lemma~\ref{lem:OK_homg}(\ref{lem:OK_homg1}), we obtain
    \begin{align*}
    \orbsize_M(\pi^{n+1}\xx,y) & = \orbsize_M(\pi^n\xx,z) \\&= 
    \abs{z}^{-\genidim(M)}\dtimes \norm{\pi^{n},z}^{\genidim(M)} \\
      & = \abs{y}^{-\genidim(M)}\dtimes\norm{\pi^{n+1},y}^{\genidim(M)}.
    \end{align*}
    Since $\{ (\pi^n\xx,y) : n \ge 0, (\xx,y)\in Z^{(n)}\}$ has measure zero,
  the claim follows.
\end{proof}

We will repeatedly use the following lemma to prove $\orbsize$-maximality.

\begin{lemma}
  \label{lem:all_monomials}
  Let $\XX = (X_1,\dotsc,X_d)$. Let $C(\XX) \in \Mat_{\ell\times
    e}(\fO[\XX])$ with $\fO[\XX]^d  C(\XX) =  \XX \dtimes (M\otimes\fO[\XX])$.
  Suppose that there exists an $N \ge 0$ such that
  for $i = 1,\dotsc, \genidim(M)$, the ideal of $\fO[\XX]$ generated
  by the $i\times i$ minors of $C(\XX)$ contains each of
  $X_1^N, \dotsc, X_d^N$.
  Then $M$ is $\orbsize$-maximal.
\end{lemma}
\begin{proof}
  Let $Z = \{ \xx \in V : \dim_K(\xx M\otimes K) < \genidim(M)\}$ as in
  Corollary~\ref{cor:formula_O}.
  Let $\xx \in \fO^d\setminus(\fP^d\cup Z)$ and let $y\in \fO\setminus\{0\}$.
  As in \S\ref{sss:O}, let
  $\bm g_i(\XX)$ be the set of non-zero $i\times i$ minors of
  $C(\XX)$.
  Then, by assumption and since $\norm\xx = 1$,
  we have $\norm{\bm g_i(\xx)} = 1$ for $i =
  0,\dotsc,\genidim(M)$ whence $\orbsize_M(\xx,y) = \abs{y}^{-\genidim(M)}$ by
  Corollary~\ref{cor:formula_O}. 
  Thus, $M$ is $\orbsize$-maximal by Lemma~\ref{lem:O_max_proj}.
\end{proof}

For a geometric interpretation of Lemma~\ref{lem:all_monomials} in a global setting, see
Proposition~\ref{prop:constant_orbit_dimension}.

\subsection{Proof of Proposition~\ref{prop:Mdxe} }
\label{ss:Mdxe}

Our proof of Proposition~\ref{prop:Mdxe} and other computations
in \S\ref{ss:classical_Lie} rely on the following.

\begin{lemma}
  \label{lem:Fr}
  Let $a_0,\dotsc,a_r\in \CC$ and write 
  $\sigma_j = a_0 + \dotsb + a_j$.
  Suppose that the integral
  \[
  F_r(a_0,\dotsc,a_r) := \int_{\fO^r\times \fO}\abs{y}^{a_0}
  \norm{x_1,y}^{a_1}\dotsb \norm{x_1,\dotsc,x_r,y}^{a_r}\,\dd\mu_{\fO^r\times\fO}(\xx,y),
  \]
  is absolutely convergent.
  Then
  \[
  F_r(a_0,\dotsc,a_r) =
  \frac{1-q^{-1}}{1-q^{-\sigma_r - r -1}}
  \dtimes
  \prod_{j=0}^{r-1}
  \frac{1 - q^{-\sigma_j - j -2}}
  {1 - q^{-\sigma_j - j -1}}.
  \]
  In particular, in the special case $a_1 = \dotsb = a_{r-1} = 0$, we obtain
  \[
  F_r(a_0,0,\dotsc,0,a_r) =
  \frac{(1-q^{-1})(1- q^{-a_0-r-1})}
  {(1-q^{-a_0-a_r-r-1})(1-q^{-a_0-1})}.
  \]
\end{lemma}
\begin{proof}
  Both claims follow by induction
  from the identities
  (a) $F_0(a_0) = \frac {1-q^{-1}} {1-q^{-1-a_0}}$ and
  (b) $F_r(a_0,\dotsc,a_r) = F_{r-1}(a_0+a_1+1,a_2,\dotsc,a_r)\dtimes
    \frac{1 - q^{-a_0-2}}{1-q^{-a_0-1}}$
    for $r \ge 1$.
  The formula for $F_0(a_0)$ in (a) is well-known and easily proved.
  By performing a change of variables according to whether $\abs{x_1} \le
  \abs{y}$ or $\abs{x_1} > \abs{y}$, we find that
  $F_r(a_0,\dotsc,a_r)
    =
      F_{r-1}(a_0+a_1+1,a_2,\dotsc,a_r) \dtimes \bigl( 1 + \int_{\fP}\abs{y}^{a_0}\dd\mu_{\fO}(y)\bigr)$
  whence (b) follows readily.
\end{proof}

\begin{proof}[Proof of Proposition~\ref{prop:Mdxe}]
  Using \eqref{eq:int_O_minors} and Lemma~\ref{lem:O_Mdxe}, we obtain
  \begin{equation}
    \label{eq:Mdxe_int}
  \Zeta_{\Mat_{d\times e}(\fO)}(q^{-s}) = (1-q^{-1})^{-1} \int_{\fO^d\times \fO}
  \abs{y}^{s+e-d-1} \norm{\xx,y}^{-e}\dd\mu_{\fO^d\times\fO}(\xx,y).
  \end{equation}
  whence the claim follows from Lemma~\ref{lem:Fr}.
\end{proof}

\begin{rem}
  \quad
  \begin{enumerate}
  \item
    Let $M\subset \Mat_{d\times e}(\fO)$ be any submodule.
    By combining Lemma~\ref{lem:O_bound} and~\eqref{eq:Mdxe_int},
    we thus obtain another interpretation of the lower bound in
    Proposition~\ref{prop:local_alpha} in the form 
    $\alpha_M \ge \alpha_{\Mat_{d\times\genidim(M)}(\fO)} = \max\bigl(d-
    \genidim(M),0\bigr)$.
  \item
    We note that \eqref{eq:Mdxe} could also be derived in an
    elementary fashion (without using $p$-adic integration) 
    using Lemma~\ref{lem:WxM} and ad hoc computations with generating
    functions.
    Such an approach quickly becomes cumbersome for more complicated
    examples such as most of those in \S\ref{s:zeta}.
    The author is, moreover, unaware of elementary proofs of general
    results such as Theorems~\ref{thm:rational} and \ref{thm:feqn}.
  \end{enumerate}
\end{rem}

\subsection{Classical Lie algebras and relatives}
\label{ss:classical_Lie}

\paragraph{Reminder.}
Let $R$ be a ring.
Recall the definitions of the \emph{special linear},
\emph{orthogonal}, and \emph{symplectic} Lie algebras
\begin{align*}
  \Sl_d(R) & = \{ a \in \Gl_d(R) : \trace(a) = 0\}, \\
  \So_d(R) & = \{ a \in \Gl_d(R) : a + a^\top = 0\} \quad\text{(assuming $\Char(R)
             \not= 2$), and}\\
  \Sp_{2d}(R) & = \left\{ \begin{bmatrix} a & b\\c & -a^\top\end{bmatrix}:
a,b,c \in \Mat_d(R), b = b^\top, c = c^\top\right\}.
\end{align*}

These are Lie subalgebras of $\Gl_d(R)$ and $\Gl_{2d}(R)$, respectively.
Finally, we let $\tr_d(R)$ and $\Nil_d(R)$ denote the Lie subalgebras of
$\Gl_d(R)$ consisting of upper triangular matrices and strictly upper
triangular matrices, respectively.

\paragraph{}
We now determine $\Zeta_{\fg}(T)$, where $\fg$ is one of the Lie algebras from
above.
Of course, the case $\fg = \Gl_d(\fO)$ is covered by Proposition~\ref{prop:Mdxe}.
Next, clearly, $\Zeta_{\Sl_1(\fO)}(T) = 1/(1-qT)$.
The general case of $\Sl_d(\fO)$ offers nothing new.

\begin{cor}
  \label{cor:sld}
  Let $d > 1$. Then $\Zeta_{\Sl_d(\fO)}(T) = \Zeta_{\Gl_d(\fO)}(T) = \frac{1-q^{-d}T}{(1-T)^2}$.
\end{cor}
\begin{proof}
  It suffices to show that $\xx\dtimes \Sl_d(\fO) \supset \xx \dtimes\Gl_d(\fO)$
  for all $\xx \in \fO^d\setminus\{0\}$.
  Let $x_\ell$ have minimal valuation among the entries of $\xx$.
  Let $\ee_j \in \fO^d$ be the $j$th unit vector.
  By our proof of Lemma~\ref{lem:O_Mdxe}, $\xx\dtimes\Gl_d(\fO)$ is generated by $\{
  x_\ell \ee_j : 1 \le j \le d\}$.
  Note that $\xx\dtimes \Sl_d(\fO)$ is spanned by all $x_i \ee_j$ and $x_i \ee_i -
  x_j \ee_j$  for $1 \le i, j \le d$ with $i \not= j$.
  It thus only remains to show that $x_\ell \ee_\ell \in \xx\dtimes \Sl_d(\fO)$.
  Since $d > 1$, we may choose $j \not= \ell$.
  Since $\abs{x_j} \le \abs{x_\ell}$, 
  $x_\ell \ee_\ell = (x_\ell \ee_\ell - x_j \ee_j) + \frac{x_j}{x_\ell} x_\ell
  \ee_j \in \xx\dtimes \Sl_d(\fO)$ whence the claim follows.
\end{proof}

\begin{prop}
  \label{prop:so}
  Let $\Char(K) \not= 2$.
  Then $\Zeta_{\So_d(\fO)}(T) = \Zeta_{\Mat_{d\times (d-1)}(\fO)}(T) = \frac{1 - q^{1-d}T}{(1-T)(1-qT)}$.
\end{prop}
\begin{rem}
  \label{rem:so}
  \quad
  \begin{enumerate}
  \item \label{rem:so1}
    It is instructive to first determine $\ask{\So_d(\FF_q)}{}$ for odd $q$.
    If $F$ is any field with $\Char(F) \not= 2$, then it is easy to
  see that $\xx \dtimes \So_d(F) = \xx^\perp$ for all $\xx \in
  F^d\setminus\{0\}$, where the orthogonal complement is taken with
  respect to the standard inner product. 
  In particular, if $\xx\not= 0$, then $\dim_F(\xx\dtimes\So_d(F))=d-1$.
  Using Lemma~\ref{lem:WxM},
  we conclude that $\ask{\So_d(\FF_q)}{} = 1 + \frac{q^d-1}{q^{d-1}} = 1 -
  q^{1-d} + q$ for odd $q$;
  this identity
  was first proved probabilistically by Fulman and Goldstein~\cite[Lem.\ 5.3]{FG15}.
  We note that for $\Char(K) \not= 2$ and $\xx\in \fO^d$, while we
  still have an inclusion 
  $\xx\dtimes\So_d(\fO) \subset \xx^\perp$, equality does not, in general, hold;
  indeed, $\xx^\perp$ is always an isolated submodule of $\fO^d$.
\item \label{rem:so2}
  While we assumed that $\Char(K) \not= 2$ in Proposition~\ref{prop:so},
  we do allow $\Char(\fO/\fP) = 2$.
  Note, however, that in this case, $\So_d(\fO)_n (=\So_d(\fO) \otimes
  \fO_n)$ is properly
  contained in the set of all skew-symmetric $d \times d$-matrices over $\fO_n$.
\end{enumerate}
\end{rem}
\begin{proof}[Proof of Proposition~\ref{prop:so}]
  Part (\ref{rem:so1}) of the preceding remark implies that $\genidim(\So_d(\fO)) = d - 1$.
  Given elements $z_1,\dotsc,z_\ell$ of some ring, we recursively
  define an $\binom {\ell} 2 \times \ell$ matrix
  \[
  m(z_1,\dotsc,z_\ell) :=
  \left[
  \begin{array}{c|c}
    \begin{matrix}-z_2 \\-z_3\\\vdots \\-z_\ell\end{matrix}
 &
   \begin{matrix} z_1 \\ & z_1 \\ & & \ddots \\ & & &z_1\end{matrix}
    \\\hline
    \begin{matrix}0 \\ \vdots \\0 \end{matrix}
 &
   m(z_2,\dotsc,z_\ell)
    \end{array}\right];
  \]
  for instance, $m(z_1)$ is the $0\times 1$-matrix and $m(z_1,z_2) =
  \left[-z_2  \,\, z_1\right]$.

  Let $e_{ij} \in \Mat_d(\fO)$ be the elementary matrix with $1$ in
  position $(i,j)$ and zeros elsewhere.
  Then the $e_{ij} - e_{ji}$ for $1\le i < j \le d$ generate $\So_d(\fO)$ as
  an $\fO$-module whence
  the rows of $m(X_1,\dotsc,X_d)$ span $\XX\dtimes \So_d(\fO[\XX])$.
  In other words, the matrix $m(X_1,\dotsc,X_d)$ plays the role of
  $C(\XX)$ in \S\ref{sss:O} for $M = \So_d(\fO)$.
  
  By induction, we may assume that $\pm X_2^i,\dotsc,\pm X_d^i$ are
  $i\times i$ minors of $m(X_2,\dotsc,X_d)$ for all $1\le i \le d-2$
  so that $\pm X_1^j,\dotsc,\pm X_d^j$ are $j \times j$
  minors of $m(X_1,\dotsc,X_d)$ for $1 \le j \le d - 1$.
  Thus, $\So_d(\fO)$ is $\orbsize$-maximal by Lemma~\ref{lem:all_monomials}
  and the claim follows from Corollary~\ref{cor:O_max}.
\end{proof}

For a ring $R$, let $\Sym_d(R) = \{ a \in \Mat_d(R) : a^\top = a\}$.

\begin{prop}
  \label{prop:Symd}
  $\Zeta_{\Sym_d(\fO)}(T) = \Zeta_{\Gl_d(\fO)}(T) = \frac{1-q^{-d}T}{(1-T)^2}$.
\end{prop}
\begin{proof}
  This proof is similar to that of Proposition~\ref{prop:so} 
  and we use the same notation.
  By considering the images of the first unit vector in $\fO^d$ under the
  matrices $e_{11}$ and
  $e_{1j} + e_{j1}$ ($2 \le j \le d$), we find that $\genidim(\Sym_d(\fO)) = d$.
  Given $z_1,\dotsc,z_\ell$, recursively define an $\binom{\ell+1} 2\times \ell$ matrix
  \[
  m'(z_1,\dotsc,z_\ell) :=
  \left[
  \begin{array}{c|c}
    \begin{matrix}z_1 \\z_2 \\\vdots\\z_\ell\end{matrix}
 &
   \begin{matrix} \\z_1 \\ & \ddots \\ &&z_1\end{matrix}
    \\\hline
    \begin{matrix}0 \\ \vdots \\0 \end{matrix}
 &
   m'(z_2,\dotsc,z_\ell)
    \end{array}\right].
  \]
  An induction similar to the one in the proof of Proposition~\ref{prop:so}
  shows that $X_1^i,\dotsc,X_d^i$ are $i\times i$
  minors of $m'(X_1,\dotsc,X_d)$ for $1\le i \le d$.
  As $\XX \dtimes \Sym_d(\fO[\XX])$ is the row span of
  $m'(X_1,\dotsc,X_d)$ over $\fO[\XX]$,
  the claim follows from Lemma~\ref{lem:all_monomials} and Corollary~\ref{cor:O_max}.
\end{proof}

\begin{prop}
  \label{prop:sp}
  $\Zeta_{\Sp_{2d}(\fO)}(T) = \Zeta_{\Gl_{2d}(\fO)}(T) = \frac{1-q^{-2d}T}{(1-T)^2}$.
\end{prop}
\begin{proof}
  We proceed along the same lines as the preceding two proofs.
  Let $e_{ij}$ denote the usual elementary matrix, now of size $2d\times 2d$.
  Using these matrices, it is easy to see that $(1,0,\dotsc,0)
  \dtimes \Sp_{2d}(\fO) = \fO^{2d}$ whence $\genidim(\Sp_{2d}(\fO)) = 2d$.
  As an $\fO$-module, $\Sp_{2d}(\fO)$ is generated by the following
  matrices:
  (i) $e_{ij} - e_{d+j,d+i}$ ($1 \le i,j \le d$), 
  (ii)~$e_{i,d+i},\, e_{d+i,i}$ ($1 \le i \le d$), and
  (iii) $e_{i,d+j} + e_{j,d+i}, \,e_{d+i,j} + e_{d+j,i}$ ($1 \le i < j
    \le d$).
  Write $\XX = (X_1,\dotsc,X_d)$ and $\XX'=(X_1',\dotsc,X_d')$.
  Define $m'(z_1,\dotsc,z_\ell)$ as in the proof of Proposition~\ref{prop:Symd}.
  Then $(\XX,\XX') \dtimes \Sp_{2d}(\fO[\XX,\XX'])$ is generated by
  the rows of

  \[
  \tilde m(\XX,\XX') := 
  \left[
  \begin{array}{c|c}
    \begin{matrix} X_1 \\ & \ddots \\ & & X_1 \end{matrix}
    &
    \begin{matrix}
      -X_1'&\phantom {-X_1'}&\phantom{-X_1'}\\
      \vdots&&\\
      -X_d'&&
    \end{matrix}\\
\hline  \\  \vdots & \ddots\\\\\hline
    \begin{matrix} X_d \\ & \ddots \\ & & X_d \end{matrix}
    &
    \begin{matrix}
      &\phantom{-X_1'} &-X_1'\\
      &&\vdots\\
      &&-X_d'
    \end{matrix}\\
    \hline\hline
    & m'(\XX) \\\hline
    m'(\XX')
  \end{array}
  \right].
  \]
  Using what we have shown about the minors of $m'(\XX)$ 
  in the proof of Proposition~\ref{prop:Symd}, we conclude that
  $X_i^j$ and $(X'_i)^j$ ($i=1,\dotsc,d; j=1,\dotsc,2d)$ arise as 
  $j\times j$ minors of $\tilde m(\XX,\XX')$.
  Again, the claim thus follows from
  Lemma~\ref{lem:all_monomials} and Corollary~\ref{cor:O_max}.
\end{proof}

In contrast to the above examples, 
neither $\Nil_d(\fO)$ nor $\mathfrak{tr}_d(\fO)$ (for $d > 1$) is $\orbsize$-maximal.

\begin{prop}
  \label{prop:n_tr}
  \quad
  \begin{enumerate}
  \item \label{prop:n_tr1}
    $\Zeta_{\Nil_d(\fO)}(T) = \frac{(1 - T)^{d-1}}{(1 - qT)^d}$.
  \item \label{prop:n_tr2}
    $\Zeta_{\mathfrak{tr}_d(\fO)}(T) =
    \Zeta_{\Nil_{d+1}(\fO)}(q^{-1}T) = \frac{(1-q^{-1}T)^d}{(1-T)^{d+1}}$.
  \end{enumerate}
\end{prop}
\begin{proof}
  Since $\Nil_{d+1}(\fO)$ is obtained from $\mathfrak{tr}_d(\fO)$ by adding a zero row
  and a zero column, by Corollary~\ref{cor:hadamard}, it 
  suffices to prove (\ref{prop:n_tr1}).
  Let $\xx \in \fO^d\setminus\{0\}$.
  Then $\xx \, \mathfrak n_d(\fO)$ is generated by
  \[
  \Bigl\{
  \bigl(0,x_1,0,\dotsc,0\bigr), \,\,
  \bigl(0,0,\gcd(x_1,x_2),0,\dotsc,0\bigr), \,\,
  \dotsc, \,\,
  \bigl(0,\dotsc,0,\gcd(x_1,\dotsc,x_{d-1})\bigr)
  \Bigr\};
  \]
  in particular, $\genidim(\Nil_d(\fO)) = d-1$.
  Moreover, by \eqref{eq:int_O_minors} and Lemma~\ref{lem:O_Mdxe}, 
  \[
  \orbsize_{\Nil_d(\fO)}(\xx,y) = \abs{y}^{1-d} \dtimes \norm{x_1,y} \dtimes
  \norm{x_1,x_2,y}\dotsb \norm{x_1,\dotsc,x_{d-1},y}.
  \]
  Hence, using \eqref{eq:int_O_minors} and Lemma~\ref{lem:Fr},
  $(1-q^{-1})\Zeta_{\Nil_d(\fO)}(q^{-s}) = F_{d-1}(s-2,-1,\dotsc,-1)$.
\end{proof}

\subsection{Diagonal matrices}
\label{ss:diag}

Let $\Diag_d(\fO) \subset \Gl_d(\fO)$ be the subalgebra
of diagonal matrices.
Clearly $\Diag_1(\fO) = \Gl_1(\fO)$ so that $\Zeta_{\Diag_1(\fO)}(T) =
\frac{1 - q^{-1}T}{(1-T)^2}$. 
It turns out that the functions $\Zeta_{\Diag_d(\fO)}(T)$
have essentially been computed by Brenti~\cite[Thm~3.4]{Bre94} in a different
context;
the author is grateful to Angela Carnevale for pointing this out.
First recall the definitions of $\BB_n$, $\Des(\sigma)$, and
$\negative(\sigma)$ from \S\ref{ss:permstat}.
For $\sigma \in \BB_n$, let $\descent(\sigma) := \#\Des(\sigma)$;
the function $\descent$ is known as the ``descent statistic''.
Define a polynomial
$h_n(X,Y) := \sum\limits_{\sigma\in \BB_n} X^{\negative(\sigma)} Y^{\descent(\sigma)}$.

\begin{thm}[{\cite[Thm 3.4(ii)]{Bre94}}]
  \label{thm:brenti}
  $\sum\limits_{i=0}^\infty
  (i(X+1) + 1)^n Y^i = \frac{h_n(X,Y)}{(1-Y)^{n+1}}$ for $n \ge 1$.
\end{thm}

The following marks a departure from the simplicity of previous
examples
of $\Zeta_M(T)$.

\begin{cor}
  \label{cor:diagonal}
  $\Zeta_{\Diag_d(\fO)}(T) = \frac{h_d(-q^{-1},T)}{(1-T)^{d+1}}$.
\end{cor}
\begin{proof}
  By Corollary~\ref{cor:hadamard},
  $\Zeta_{\Diag_d(\fO)}(T)$ is the $d$th Hadamard power
  of $\Zeta_{\Diag_1(\fO)}(T)$.
  Since
  $\Zeta_{\Diag_1(\fO)}(T) = \frac{1-q^{-1} T}{(1-T)^2}
  = \sum_{i=0}^\infty (1 + i - iq^{-1})T^i$,
  the claim follows from Theorem~\ref{thm:brenti}.
\end{proof}

\begin{ex}
{\footnotesize
  \quad
  \begin{enumerate}
  \item
    $\begin{aligned}[t]
    \Zeta_{\Diag_2(\fO)}(T)  = \frac
    {1 + T - 4q^{-1}T + q^{-2}T + q^{-2}T^2}
    {(1 - T)^3}\end{aligned}.$
\item
    $\begin{aligned}[t]\Zeta_{\Diag_3(\fO)}(T) = 
    \frac{
      1
      + 6q^{-2}T 
      - 12q^{-1}T 
      + 4T 
      + T^2 
      - q^{-3}T 
      - 4q^{-3}T^2
      + 12q^{-2}T^2 
      - 6q^{-1}T^2 
      - q^{-3}T^3 
    }{(1 - T)^4
    }.\end{aligned}$
  \end{enumerate}
}
\end{ex}

We note that permutation statistics have previously featured in explicit
formulae for representation zeta functions~\cite{SV14,BC17};
see also \cite{CV17,CSV17}.

\section{Constant rank spaces}
\label{s:examples_min}

By a \emph{constant rank space} over a field $F$,
we mean a subspace $M\subset\Mat_{d\times e}(F)$ such that
all non-zero elements of $M$ have the same rank, say $r$;
we then say that $M$ has constant rank $r$.
Such spaces have been studied extensively in the literature
(see e.g.\ \cite{Bea81,Syl86,Wes87,IL99,BFM13}), often in the context of
vector bundles on projective space.
A problem of particular interest is 
to find, for given $d$ and $r$, the largest possible dimension of a subspace
of $\Mat_d(\CC)$ of constant rank $r$.
Apart from trivial examples such as band matrices (see Example~\ref{ex:band}
below),
the construction of  constant rank spaces
(in particular those of large dimension)
seems to be challenging. 
Note that if $M \subset \Mat_{d\times e}(\FF_q)$ has constant rank~$r$ and
dimension $\ell$,
then
\begin{equation}
  \label{eq:ask_constant_rank}
\ask M {} =  q^{-\ell} \bigl(q^d + (q^{\ell}-1)q^{d-r}\bigr)
= q^{d-\ell} + q^{d-r} - q^{d - \ell - r}.
\end{equation}

In \S\ref{ss:K_min}, we consider a natural analogue,
$\kersize$-minimality, of the concept of $\orbsize$-maximality
studied in \S\ref{ss:O_max}. 
We then derive interpretations of these notions in a global
setting in \S\ref{ss:global_extremality}---in particular, we will see
that $\kersize$-minimality is related to constant rank spaces.

\subsection{$\kersize$-minimality}
\label{ss:K_min}

Let $\fO$ be the valuation ring of a non-Archimedean local field $K$ of
arbitrary characteristic.
Let $M \subset \Mat_{d\times e}(\fO)$ be a submodule.
Recall the definition of $\kersize_M$ from Definition~\ref{d:OK}.

\begin{lemma}
  Let $F\colon \fO^\ell \to M$ be an $\fO$-module isomorphism,
  $\ww \in \fO^\ell$, and $y \in \fO\setminus\{0\}$.
  Then $\kersize_M(F(\ww),y) \ge \abs{y}^{\genrank(M)-d} \norm{\ww,y}^{-\genrank(M)}$.
\end{lemma}
\begin{proof}
  We may assume that $\ww\not= 0$.
  Let $a := F(\ww)$ have equivalence type $(\lambda_1,\dotsc,\lambda_r)$ (see
  \S\ref{ss:size_ker}); of course,
  $r \le \genrank(M)$.
  Let $m := \min(\nu(w_1),\dotsc,\nu(w_\ell))$, $n = \nu(y)$, and $a := F(\ww)$.
  Then $m \le \min(\nu(a_{ij}) : 1\le i,j\le d) = \lambda_1 \le \lambda_2 \le
  \dotsb \le \lambda_r$
  and Lemma~\ref{lem:size_ker}(iii) shows that
  \begin{align*}
  \kersize_M(F(\ww),y) &\ge q^{r\min(m,n) + (d-r)n} \ge q^{\genrank(M) \min(m,n) +
    (d-\genrank(M))n}\\&
  = \abs{y}^{\genrank(M)-d} \norm{\ww,y}^{-\genrank(M)}.\qedhere
  \end{align*}
\end{proof}

The following is analogous to Lemma~\ref{lem:Omax_null}.

\begin{lemma}
  For an $\fO$-module isomorphism $F\colon \fO^\ell \to M$, the following are equivalent:
  \begin{enumerate}
  \item $\kersize_M(F(\ww),y) = \abs{y}^{\genrank(M)-d}
  \norm{\ww,y}^{-\genrank(M)}$ for all $(\ww,y) \in \fO^\ell\times \fO$
  outside a set of measure zero.
  \item
    $\kersize_M(F(\ww),y) = \abs{y}^{\genrank(M)-d}
  \norm{\ww,y}^{-\genrank(M)}$ for all 
  $\ww \in \fO^\ell$ and $y \in \fO\setminus\{0\}$.
  \qed
  \end{enumerate}
\end{lemma}

\begin{defn}
  We say that $M$ is \emph{$\kersize$-minimal} if there exists an $\fO$-module
  isomorphism $F\colon \fO^\ell \to M$ which satisfies one of the equivalent
  conditions from the preceding lemma.
\end{defn}

Clearly, if $\xx \in \fO^n$ and $a \in \GL_n(\fO)$, then $\norm{\xx a} =
\norm{x}$.
We conclude that 
if the condition in the preceding definition is satisfied for some
isomorphism $F\colon \fO^\ell\to M$, then it holds for all of them.
Lemma~\ref{lem:Fr} and arguments as in the proof of Corollary~\ref{cor:O_max}
now imply the following.

\begin{prop}
  \label{prop:Z_constant_rank}
  Let $r = \genrank(M)$ and $\ell = \dim_K(M \otimes K)$.
  Then $M$ is $\kersize$-minimal if and only if
  \[
  \pushQED{\qed}
  \Zeta_M(T) = \frac{1-q^{d-\ell-r}T}{(1-q^{d-\ell}T)(1-q^{d-r}T)}.\qedhere
  \popQED
  \]
\end{prop}

The following sufficient condition for $\kersize$-minimality is proved
similarly to Lemma~\ref{lem:all_monomials}.

\begin{lemma}
  \label{lem:all_monomials2}
  Let $(a_1,\dotsc,a_\ell)$ be an $\fO$-basis of $M$.
  Suppose that there exists $N \ge 0$ such that
  for $1 \le i \le \genrank(M)$,
  the ideal generated by the $i\times i$ minors of
  $X_1 a_1+\dotsb + X_\ell a_\ell \in \Mat_{d\times
    e}(\fO[X_1,\dotsc,X_\ell])$
  contains $X_1^N,\dotsc,X_\ell^N$.
  Then $M$ is $\kersize$-minimal. \qed
\end{lemma}

\begin{ex}[Band matrices]
  \label{ex:band}
  Let $r \ge 1$ and define
  \[
  B_r = \left\{
    \begin{bmatrix}
      x_1 \\
      x_2 & \ddots \\
      \vdots & \ddots & x_1 \\
      x_r & \vdots & x_2 \\
      & \ddots & \vdots\\
      & & x_r
    \end{bmatrix}
    : x_1,\dotsc,x_r \in \fO
    \right\} \subset \Mat_{(2r-1)\times r}(\fO).
  \]
  By Lemma~\ref{lem:all_monomials2} and
  Proposition~\ref{prop:Z_constant_rank} (with $d = 2r-1$ and $\ell = r$),
  $\Zeta_{B_r}(T) = \frac{1-q^{-1} T}{(1-q^{r-1}T)^2}$.
\end{ex}

\subsection{A global interpretation}
\label{ss:global_extremality}

Henceforth, let $k$ be a number field with ring of integers $\fo$;
recall the notation from \S\ref{ss:global_alpha}.
Let $\bar k$ be an algebraic closure of $k$.
Let $M \subset \Mat_{d\times e}(\fo)$ be a submodule.
The following can be proved similarly to Proposition~\ref{prop:genrank} (and Proposition~\ref{prop:genidim}).
\begin{lemma}
  \label{lem:global_genrank_genidim}
  \quad
  \begin{enumerate}
  \item \label{lem:global_genrank_genidim1}
    $\max\limits_{a\in M}\rank_k(a)
    = \max\limits_{\bar a \in M\otimes_{\fo} \bar k}\rank_{\bar k}(\bar a)
    = \genrank(M_v)$ for all $v \in \Places_k$.
  \item \label{lem:global_genrank_genidim2}
    $\max\limits_{\xx \in \fo^d}\dim_k(\xx \dtimes (M\otimes_{\fo} k))
    = \max\limits_{\bar\xx \in \bar k^d}\dim_{\bar k}(\bar\xx \dtimes (M\otimes_{\fo} \bar
    k)) = \genidim(M_v)$ for all $v \in \Places_k$. \qed
  \end{enumerate}
\end{lemma}

Extending Definitions~\ref{d:genidim} and \ref{d:genrank},
we let $\genrank(M)$ and $\genidim(M)$ be the common number in (\ref{lem:global_genrank_genidim1}) and
(\ref{lem:global_genrank_genidim2}), respectively.
We will now see that $\kersize$-minimality is closely related to constant rank
spaces.

\begin{prop}
  \label{prop:constant_rank}
  \quad
  \begin{enumerate}
  \item \label{prop:constant_rank1}
    Let $(a_1,\dotsc,a_\ell) \subset M$ be a $k$-basis of $M\otimes_{\fo} k$.
    Let $I_i$ be the ideal of $k[X_1,\dotsc,X_\ell]$ generated
    by the $i\times i$ minors of $X_1 a_1 + \dotsb + X_\ell a_\ell$.
    Then $M \otimes_{\fo} \bar k$ is a constant rank space 
    if and only if there exists $N \ge 0$ such that $X_1^N,\dotsc,X_\ell^N \in
    I_i$ for $i = 1,\dotsc,\genrank(M)$.
  \item \label{prop:constant_rank2}
    If $M \otimes_{\fo} \bar k$ is a constant rank space, then $M_v$ is
    $\kersize$-minimal for almost all $v\in \Places_k$. 
  \end{enumerate}
\end{prop}
\begin{proof}
  \quad
  \begin{enumerate}
  \item
    Let $\Var(I) \subset \bar k^\ell$ be the algebraic set corresponding to 
    $I \normal k[\XX] := k[X_1,\dotsc,X_\ell]$ and
    let $r = \genrank(M)$.
    Then $M\otimes_{\fo}\bar k$ is a constant rank space if and only if 
    $\Var(I_i) \subset \{0\}$ for $i =1,\dotsc,r$
    which, by Hilbert's Nullstellensatz,
    is equivalent to $\sqrt{I_i} \supset \langle X_1,\dotsc,X_r\rangle$ for
    $i = 1,\dotsc,r$.
  \item
    This follows from (\ref{prop:constant_rank1}) and Lemma~\ref{lem:all_monomials2};
    note that if $f \in \fo[\XX]$ belongs to the $k[\XX]$-ideal generated by
    $g_1,\dotsc,g_r\in \fo[\XX]$, then $f$ also belongs to the
    $\fo_v[\XX]$-ideal generated by the $g_i$, at least for almost all $v\in
    \Places_k$.
    \qedhere
  \end{enumerate}
\end{proof}

In view of Lemma~\ref{lem:WxM}, we say that a subspace $M'
\subset\Mat_{d\times e}(F)$ (where $F$ is a field) has
\emph{constant orbit dimension} if all $F$-spaces $\xx M'$ for $\xx \in
F^d\setminus\{0\}$ have the same dimension.
The following counterpart of Proposition~\ref{prop:constant_rank} 
is then proved in the same way.
\begin{prop}
  \label{prop:constant_orbit_dimension}
  \quad
  \begin{enumerate}
  \item
  Let $(a_1,\dotsc,a_\ell) \subset M$ be a $k$-basis of $M\otimes_{\fo} k$.
  Let $J_i$ be the ideal of $k[\XX] := k[X_1,\dotsc,X_d]$ generated
  by the $i\times i$ minors of
  \[
  \begin{bmatrix}
    \XX a_1 \\
    \vdots\\
    \XX a_\ell
  \end{bmatrix}.
  \]
  Then
  $M \otimes_{\fo} \bar k$ has constant orbit dimension.
  if and only if there exists $N\ge 0$ such that $X_1^N,\dotsc,X_d^N \in J_i$
  for $i = 1,\dotsc,\genidim(M)$.
  \item If $M\otimes_{\fo}\bar k$ has constant orbit dimension, then $M_v$ is
    $\orbsize$-maximal for almost all $v\in \Places_k$. \qed
  \end{enumerate}
\end{prop}

\section{Smooth determinantal hypersurfaces}
\label{s:det}

In a series of papers, Voll~\cite{Vol04,Vol05,Vol09}
developed geometric techniques for studying the normal subgroup growth of
finitely generated torsion-free nilpotent groups of class~$2$.
Under suitable genericity assumptions on the Pfaffian hypersurface
attached to such a group, he produced an explicit formula~\cite[Thm~3]{Vol05}
for almost all of its local normal subgroup zeta functions in terms of
numbers of rational points of the aforementioned hypersurface.

The following is an analogue of Voll's result for ask zeta functions.
Here, the role of the Pfaffian hypersurface of a group is
played by the determinantal hypersurface associated with a matrix of linear
forms, a classical topic in algebraic geometry (see Remark~\ref{rem:classical_det}).
Throughout this section, $K$ is a non-Archimedean local field of arbitrary
characteristic with valuation ring $\fO$.
Recall that $\RF = \fO/\fP$ denotes the residue field of $K$.

\begin{thm}
  \label{thm:det}
  Let $a_1,\dotsc,a_\ell \in \Mat_d(\fO)$, where $\ell \ge 1$.
  Let $\XX = (X_1,\dotsc,X_\ell)$ consist of algebraically independent
  variables over $\fO$.
  Write $a(\XX) = X_1 a_1 + \dotsb + X_\ell a_\ell \in \Mat_d(\fO[\XX])$ and
  let $M = \{ a(\xx) : \xx
  \in \fO^\ell\} \subset \Mat_d(\fO)$.
  Let $F(\XX) := \det(a(\XX))$.
  Suppose that the following smoothness condition is satisfied:
  \begin{equation}
    \label{eq:smooth}
    \tag{\textsf{SM}}
    \text{For all } \bar \xx\in \RF^\ell, \text{ if }
    F(\bar\xx) = \frac{\partial F (\bar\xx)}{\partial X_1} = \dotsb =
    \frac{\partial F (\bar\xx)}{\partial X_\ell} = 0,
    \text{ then } \bar\xx = 0.
  \end{equation}
  Let $H := \Proj\!\left(\fO[\XX]/(F(\XX))\right) \subset \PP^{\ell-1}_{\fO}$.
  Then:
  \[
  \Zeta_M(T) = \frac{1-q^{-\ell}T}{(1-T)(1-q^{d-\ell}T)}  + \#H(\RF) \dtimes (q-1)^2
  \dtimes \frac{q^{-\ell}T}{(1-T)^2(1-q^{d-\ell}T)}.
  \]
\end{thm}

A proof of Theorem~\ref{thm:det} will be given below.
We henceforth use the notation of Theorem~\ref{thm:det} and assume that
condition \eqref{eq:smooth} is satisfied.
Note that the latter assumption is certainly satisfied if $H$ is smooth as a scheme over $\fO$.

Let $I_i(\XX)$ be the ideal of $\fO[\XX]$ generated by the $i\times i$ minors of
$a(\XX)$.
For an ideal $J(\XX) \normal \fO[\XX]$ and an element $b$ of an
(associative, commutative, unital) $\fO$-algebra $B$, we write
$J(b) = \{ f(b) : f \in J(\XX) \} \normal B$.

\begin{lemma}[{Cf.\ \cite[Pf of Thm~2.2]{CT79}}]
  \label{lem:detrk}
  \quad
  \begin{enumerate}
  \item 
    \label{lem:detrk1}
    $\displaystyle
    \frac{\partial F(\XX)}{\partial X_i}
    \in I_{d-1}(\XX)$ for $i = 1,\dotsc,\ell$.
  \item\label{lem:detrk2}
    Let $\bar\xx \in \RF^\ell \setminus\{0\}$.
    Then $\rank_{\RF}(a(\bar\xx)) \in \{d-1,d\}$.
  \end{enumerate}
\end{lemma}
 \begin{proof}
   See \cite[p.\ 426]{CT79} for (\ref{lem:detrk1}).
   For (\ref{lem:detrk2}),
   if $\rank_{\RF}(a(\bar\xx)) < d-1$ for $\bar\xx\in \RF^\ell$,
   then $F(\bar\xx) = 0$ and $I_{d-1}(\bar\xx) = \{0\}$.
   Part (\ref{lem:detrk1}) and \eqref{eq:smooth} then imply that
   $\bar\xx = 0$.
 \end{proof}

\begin{cor}
  \label{cor:det_Mrk}
  If $d \ge 2$, then
  $\fO^\ell \to M, \, \xx \mapsto a(\xx)$
    is an isomorphism of $\fO$-modules.
\end{cor}
\begin{proof}
   Let $\xx \in \fO^\ell$ with $a(\xx) = 0$ and
   suppose that $\xx \not= 0$.
   Choose $\pi \in \fP\setminus\fP^2$.
   Then $\xx = \pi^m \yy$ for some $m \ge 0$ and an element $\yy\in \fO^\ell$
   whose image in $\RF^\ell$ is non-zero.
   Then $a(\yy) = 0$ but also $\rank_{\RF}(a(\yy)\otimes\RF)
   \ge d-1 > 0$ by Lemma~\ref{lem:detrk}(\ref{lem:detrk2}), a contradiction.
\end{proof}

\begin{proof}[Proof of Theorem~\ref{thm:det}]
  First suppose that $d = 1$.
  Then $a(\XX) = F(\XX)$ is linear.
  Moreover, condition \eqref{eq:smooth} implies that $M = \Mat_1(\fO)$
  and $\#H(\RF) = (q^{\ell-1}-1)/(q-1)$.
  The claim is now easily verified using a direct computation and
  Proposition~\ref{prop:Mdxe}.

  Let $d \ge 2$.
  Write $U = \fO^\ell$ and
  fix $\pi \in \fP\setminus \fP^2$.
  Let $\xx  \in U \setminus \fP U$ and $y\in \fP\setminus\{0\}$.
  By Lemma~\ref{lem:detrk}(\ref{lem:detrk2}),
  for each $i = 1,\dotsc,d-1$,
  some $i\times i$ minor of $a(\xx)$ belongs to $\fO^\times$.
  Hence, $I_0(\xx) = \dotsc = I_{d-1}(\xx) = \fO$.
  As $F(\XX) \not= 0$,
  $\genrank(M) = d$.
  Using Corollaries~\ref{cor:formula_K} and~\ref{cor:det_Mrk}, 
  \[
  \kersize_M(a(\xx),y)
  = \prod_{i=1}^d
  \frac{\norm{I_{i-1}(\xx)}}
  {\norm{I_{i}(\xx) \cup y I_{i-1}(\xx)}}
  = 
  \norm{F(\xx),y}^{-1}.
  \]

  Write $t = q^{-s}$.
  By Proposition~\ref{prop:proj_twoint}(\ref{prop:proj_twoint2}),
  \begin{equation}
    \label{eq:shorthand}
    (1-q^{d-\ell} t) \dtimes \Zeta_M(t) = 
    1 + (1-q^{-1})^{-1} \int_{(U\setminus\fP U)\times \fP} \omega,
  \end{equation}
  where we wrote $\omega$ as a shorthand for ${\abs{y}^{s-1}}{\norm{F(\xx),y}}^{-1}\!\dd\mu_{U\times\fO}(\xx,y)$.
  In order to evaluate the integral in \eqref{eq:shorthand},
  we decompose the domain of integration into sets 
  of the form $(\xx_0 + \fP U)\times \fP$ for $\xx_0 \in U\setminus\fP U$.
  If $F(\xx_0) \not\equiv 0 \pmod \fP$,
  then clearly
  \begin{align*}
    \int_{(\xx_0 + \fP U)\times \fP}\omega
    & = \int_{(\xx_0 + \fP U)\times \fP} \abs{y}^{s-1}
      \,\dd\mu_{U\times\fO}(\xx,y) 
     = (1-q^{-1}) A(q,t),
  \end{align*}
  where $A(q,t) := q^{-\ell} t /(1-t)$.
  Suppose that $F(\xx_0) \equiv 0 \pmod \fP$.
  Using \eqref{eq:smooth} and Hensel's lemma~\cite[Ch.\ III, \S 4.5, Cor.~2]{Bou89},
  a measure-preserving change of coordinates transforms the
  map induced by $F(\XX)$ on $\xx_0 + \fP U$ into the map induced by
  $X_1$, say,
  on $\fP U$. Thus,
  \begin{align*}
  \int_{(\xx_0 + \fP U)\times \fP}\omega
   & = \int_{\fP U\times \fP}\frac{\abs{y}^{s-1}}{\norm{x_1,y}}
     \dd\mu_{U\times\fO}(\xx,y) 
    = q^{1-\ell} \dtimes \int_{\fP\times\fP} 
     \frac{\abs{y}^{s-1}}{\norm{x,y}}
     \dd\mu_{\fO\times\fO}(x,y)\\
   & =
     q^{1-\ell} \dtimes
     \sum_{m,n=1}^{\infty}
     \mu_{\fO\times\fO}(\pi^m \fO^\times \times \pi^n \fO^\times)
     \dtimes q^{n+\min(m,n)}t^n \\&
    =
     (1-q^{-1})^2 \dtimes q^{1-\ell}
     \sum_{m,n=1}^\infty  
     q^{\min(m,n)-m}\dtimes t^n.
  \end{align*}
  
  By an elementary calculation,
  \[
  \sum_{m,n=1}^\infty q^{\min(m,n)-m}\dtimes t^n
  = \frac{t(1-q^{-1}t)}{(1-q^{-1}) (1-t)^2}
  \]
  and thus
  $\int\limits_{(\xx_0 + \fP U)\times \fP}\omega
  = (1-q^{-1}) B(q,t)$,
  where $B(q,t) = q^{1-\ell} t \dtimes (1-q^{-1}t)/(1-t)^2$.

  The condition $F(\xx_0) \equiv 0\pmod\fP$ is equivalent to
  the image of $\xx_0$ in $\PP^{\ell-1}(\RF)$ being contained in $H(\RF)$.
  Therefore,
  \begin{align*}
    (1-q^{d-\ell}t) \Zeta_M(t) & = 1 + (1-q^{-1})^{-1} \int_{(U \setminus \fP U) \times \fP}
                 \omega\\
    & = 1 +  (q-1) \left(\# \PP^{\ell-1}(\RF) - \# H(\RF)\right) A(q,t) 
      +  (q-1) \# H(\RF) B(q,t) \\
    & = 1 + {(q-1) \#\PP^{\ell-1}(\RF)}\dtimes \frac
      {q^{-\ell}t}{1-t}
      + \#H(\RF) (q-1) \bigl(B(q,t)-A(q,t)\bigr)
  \end{align*}
  and the claim follows from
  $1 + (1-q^{-\ell})\frac t{1-t} = \frac{1-q^{-\ell}t}{1-t}$
  and $B(q,t) - A(q,t) = (q-1)  \frac {q^{-\ell}t} {(1-t)^2}$.
\end{proof}

\begin{rem}
  \label{rem:det_const_rk}
  By the Chevalley-Warning Theorem, $\#H(\RF) = 0$ implies that
  $\ell \le d$; see e.g.\ \cite[Ch.~1, \S 2, Cor.\ 1]{Ser73}.
  Suppose that $\#H(\RF) = 0$.
  Then $M$ is readily seen to be $\kersize$-minimal (cf.\ the proof of
  Theorem~\ref{thm:det}) and $M\otimes K$ has constant rank $d$.
  The formula $\Zeta_M(T) = \frac{1-q^{-\ell}T}{(1-T)(1-q^{d-\ell}T)}$ given by
  Theorem~\ref{thm:det} in this case agrees with
  Proposition~\ref{prop:Z_constant_rank}. 
\end{rem}

\begin{ex}[Diagonal $2\times 2$ matrices]
  \label{ex:diag_det}
  Let $a(\XX) = \diag(X_1,\dotsc,X_\ell)$.
  Then condition \eqref{eq:smooth} is satisfied if and only if $\ell
  \le 2$.
  Using Theorem~\ref{thm:det} with $\ell = 2$ and $\#H(\RF) = 2$, we
  recover the special case $d = 2$ of Corollary~\ref{cor:diagonal}.
\end{ex}

\begin{ex}[Univariate polynomials]
  \label{ex:det_uni}
  \quad
  \begin{enumerate}
  \item
    \label{ex:det_uni1}
    (Local case.)
    Let $d \ge 1$ and
    $f(X) = X^d + c_{d-1}X^{d-1} + \dotsb + c_1 X + c_0 \in \fO[X]$.
    Let
    \[
    b_f = 
    \begin{bmatrix}
      0 & 1\\
      & \ddots & \ddots \\
      & & 0 & 1\\
      -c_0 & \hdots & -{c_{m-2}} & -c_{d-1}
    \end{bmatrix}
    \in \Mat_d(\fO)
    \]
    be the companion matrix of $f(X)$.
    Let
    $a_{f}(X,Y) = X \dtimes 1_d - Y \dtimes b_f \in
    \Mat_d(\fO[X,Y])$,
    where $1_d$ denotes the $d\times d$ identity matrix.
    Then $\det(a_{f}(X,Y)) = Y^{d} \dtimes f(Y^{-1} X)$ is the
    homogenisation of $f(X)$.
    Let $M_f(\fO) = \{ a_f(x,y) : x,y\in \fO\} \subset \Mat_d(\fO)$;
    note that $M_f(\fO)$ has $\fO$-rank $2$ and that $\genrank(M_f(\fO)) = d$.

    We now assume that the image of $f(X)$ in $\RF[X]$ has no repeated roots in $\RF$.
    Then condition \eqref{eq:smooth} is satisfied for $a(X,Y) = a_f(X,Y)$ and
    $F(X,Y) = \det(a_f(X,Y))$.
    We may therefore apply Theorem~\ref{thm:det} (with $\ell = 2$) to obtain
    \begin{equation}
      \label{eq:univariate_det}
      \Zeta_{M_f(\fO)}(T) =
      \frac{
        (1-T)(1-q^{-2}T) + n_f(\RF)(1-q^{-1})^2T
      }
      {(1-T)^2(1-q^{d-2}T)},
    \end{equation}
    where $n_f(\RF) =\#\{ x\in \RF : f(x) = 0\} \in \{0,1,\dotsc,d\}$.

    If the image of $f(X)$ in $\RF[X]$ is irreducible, then we are in the special
    case $n_f(\RF) = \#H(\RF) = 0$
    discussed in Remark~\ref{rem:det_const_rk}.
    At the other extreme, if $f(X)$ splits into $d$
    linear factors
    over $\RF$ (which is possible if and only if $q \le d$), then $n_f(\RF) = d$.
  \item
    \label{ex:det_uni2}
    (Global case.)
    Let $k$ be a number field.
    Let $f(X) = X^d + c_{d-1} X^{d-1} + \dotsb + c_0 \in \ZZ[X]$ be the
    minimal polynomial of an integral primitive element of $k/\QQ$.
    For a (rational) prime $p$, the reduction $f_p(X) \in \FF_p[X]$ of $f(X)$ modulo $p$ is
    separable if and only if $p$ does not divide the discriminant $\disc(f(X))$ of $f(X)$.
    Let $M_f \subset \Mat_d(\ZZ)$ be the module of matrices generated by the identity
    matrix $1_d$ and the companion matrix of $f(X)$.
    Then, using the notation from (\ref{ex:det_uni1}), we may identify
    $M_f \otimes_{\ZZ} \ZZ_p = M_f(\ZZ_p)$.
    In particular, if $\ndivides p {\disc(f(X))}$, then we obtain a
    formula~\eqref{eq:univariate_det} for $\Zeta_{M_f \otimes_{\ZZ}\ZZ_p}(T)$
    which depends on the number $n_f(\FF_p)$ of roots of $f(X)$ in $\FF_p$.
  \end{enumerate}
\end{ex}

\begin{rem}
  For sufficiently large primes~$p$,
  Voll~\cite[Prop.\ 3]{Vol04} gave an explicit formula for almost all local
  normal subgroup zeta functions associated with an indecomposable $\mathfrak
  D^*$-group attached to a power of an irreducible 
  polynomial $f(X)$ over $\QQ$; for background on $\mathfrak D^*$-groups, see \cite{GS84}.
  Voll's formula depends on the number of roots of $f(X)$ modulo $p$ and has
  essentially the same shape as \eqref{eq:univariate_det};
  we note that the matrix $\mathbf B$ in \cite[Eqn~(16)]{Vol04} (and in
  \cite[\S 6]{GS84}) plays the same role as $a_f(X,Y)$ above.
\end{rem}

\begin{ex}[$Y^2 = X^3-X$]
  Let $E \subset \PP^2_{\ZZ}$ be defined by the
  homogenisation of $Y^2 = X^3 - X$.
  Then $E \otimes_{\ZZ} F$ is an elliptic curve
  for every field $F$ of characteristic distinct from $2$.
  The curve $E \otimes_{\ZZ} \QQ$ has been previously used to
  show that various group-theoretic counting problems exhibit
  arithmetically ``wild'' behaviour.
  In particular, using determinantal representation in the sense of the
  present section, du~Sautoy~\cite{dS01} constructed a nilpotent group of
  Hirsch length $9$ whose local normal subgroup zeta function at a
  prime~$p$ depends on the number $\#E(\FF_p)$ of $\FF_p$-rational points of $E$.
  The precise shapes of these local zeta functions were first determined by Voll~\cite{Vol04}.
  Due to  the ``wild'' behaviour of $\#E(\FF_p)$ as a function of $p$ (see
  \cite[Pf of Thm~2.1]{dS01}),
  du~Sautoy's result disproved earlier predictions on the growth of (normal)
  subgroups of nilpotent groups. 
  His construction has since been used to demonstrate that other
  group-theoretic counting problems can be ``wild''
  (e.g.\ the enumeration of representations~\cite[Ex.~2.4]{Vol11} 
  or of ``descendants''~\cite{dSVL12}). We will now see that the
  present setting is no exception.
  Let
  \[
  a(X,Y,Z) = \begin{bmatrix} Z & X & Y \\ X & Z & 0 \\ Y & 0 &
    X\end{bmatrix}
  \]
  and $M = \{ a(x,y,z) : x,y,z\in \ZZ\} \subset \Mat_3(\ZZ)$.
  Then $E$ is defined by $\det(a(X,Y,Z)) = 0$.
  Suppose that $q$ is odd 
  so that $F(X,Y,Z) := \det(a(X,Y,Z))$ satisfies
  condition \eqref{eq:smooth}.
  By applying Theorem~\ref{thm:det},
  we thus obtain
  \[
  \Zeta_{M\otimes_\ZZ\fO}(T) = \frac{1-q^{-3} T}{(1-T)^2}
  + \#E(\RF) \dtimes (q-1)^2 \frac{q^{-3}T}{(1-T)^3}.
  \]

  In particular,
  Theorem~\ref{thm:denef} accurately reflects the
  general dependence of $\Zeta_{M_v}(T)$ on a place $v \in \Places_k$
  for a module of matrices $M$ over the ring of integers $\fo$ of a number
  field~$k$.
  However, just as in the study of zeta functions of groups, it
  is presently unclear if anything meaningful can be said about the varieties
  $V_i\otimes_{\fo} k$ ``required'' to produce formulae~\eqref{eq:denef}
   as $M$ varies over all modules of matrices over $\fo$.
\end{ex}

\begin{rem}
  \label{rem:classical_det}
  Determinantal representations of projective hypersurfaces
  (i.e.\ representations of defining polynomials as determinants of matrices of linear forms)
  over the complex numbers (or over algebraically closed fields) have been
  studied extensively; see e.g.\ \cite[\S\S 4.1, 9.3]{Dol12}, \cite{Bea00} and \cite{KV12}.
  In particular, in the smooth case, only curves and cubic surfaces over $\CC$ generically
  admit determinantal representations.
  Ishitsuka~\cite[Cor.~8.3]{Ish17} showed
  that over a local field (of arbitrary characteristic), every smooth plane
  cubic admits a determinantal representation.
  He further showed~\cite[Thm~1.1(i)]{Ish17}  that
  the same is true of a positive proportion (measured by height) of smooth plane cubics
  over $\QQ$.
\end{rem}

\begin{rem}
  Let $k$ be a number field with ring of integers $\fo$;
  recall the notation from~\S\ref{ss:global_alpha}.
  Let $a(\XX) = a_1 X_1 + \dotsb + a_\ell X_\ell\in \Mat_d(\fo[\XX])$ for
  $\ell \ge 1$ and $a_1,\dotsc,a_\ell\in \Mat_d(\fo)$.
  Let $F(\XX) := \det(a(\XX))$ and $M := \{ a(\xx) : \xx\in\fo^\ell\}$.
  Define $\mathsf H := \Proj\!\left(\fo[\XX]/(F(\XX))\right)$ and suppose that $\mathsf H
  \otimes_{\fo} k$ is smooth over $k$.
  For almost all $v\in\Places_k$, Theorem~\ref{thm:det} then provides
  a formula for $\Zeta_{M_v}(T)$ in terms of $\#\mathsf H(\RF_v)$.
  In this special case,
  for almost all $v\in\Places_k$,
  the functional equation in Theorem~\ref{thm:feqn} follows immediately from
  the identity
  \[
  \# \mathsf H(\RF_v)\Bigm\vert_{q_v\to q_v^{-1}} = q^{2-\ell}_v \dtimes
  \#\mathsf H(\RF_v),
  \]
  a consequence of the Weil conjectures applied to
  the smooth projective variety
  $\mathsf H\otimes_{\fo} \RF_v$;
  cf.\ \cite[Pf\ of Thm~4]{DM91}.
\end{rem}

\section{Orbits and conjugacy classes of linear groups}
\label{s:bounded_denominators}

In this section, we use $p$-adic Lie theory to relate ask, orbit-counting, and
conjugacy class zeta functions.
In \S\ref{ss:saturable}, we recall properties
of saturable pro-$p$ groups and Lie algebras.
In \S\ref{ss:linear_orbits}, we prove that orbit-counting zeta functions over
$\ZZ_p$ are rational.
In \S\ref{ss:cent_stab} we compare group stabilisers and Lie
centralisers under suitable hypotheses and this allows us to deduce
Theorem~\ref{thm:bounded_denominators} in~\S\ref{ss:pf_bounded_denominators}.
Finally, in \S\ref{ss:unipotent_orbits}, we prove Theorem~\ref{thm:unipotent}.

\subsection{Reminder: saturable pro-$p$ groups and Lie algebras}
\label{ss:saturable}

We briefly recall Lazard's~\cite{Laz65} notion of ($p$-)saturability
of groups and Lie algebras using Gonz\'alez-S\'anchez's~\cite{Gon07}
equivalent formulation.

Let $\fg$ be a Lie $\ZZ_p$-algebra whose underlying $\ZZ_p$-module is free of
finite rank.
A \emph{potent filtration} of $\fg$ is a central series
$\fg = \fg_1 \supset \fg_2 \supset \dotsb$ of ideals (i.e.\
$[\fg_i,\fg] \subset \fg_{i+1}$ for all $i$) with
$\bigcap\limits_{i=1}^\infty \fg_i = 0$ and such 
that
$[\fg_i,_{p-1},\fg] :=
[\fg_i,\underbrace{\fg,\dotsc,\fg}_{p-1}] := [\dotso[[\fg_i,\fg],\fg],\dotsc,\fg] \subset p\fg_{i+1}$
for all $i \ge 1$.
We say that $\fg$ is \emph{saturable} if it admits a potent filtration.

\begin{prop}
  (\textup{\cite[Prop.\ 2.3]{AKOV13}})
  Let $\fO$ be the valuation ring of a non-Archimedean local field
  $K \supset \QQ_p$.
  Let $\rind(K/\QQ_p)$ denote the ramification index of $K/\QQ_p$.
  Let $\fg$ be a Lie $\fO$-algebra whose underlying $\fO$-module is free of
  finite rank.
  Let $m > \frac {\rind(K/\QQ_p)}{p-1}$.
  Then $\fg^m$ ($= \fP^m\fg$) is saturable as a $\ZZ_p$-algebra.
\end{prop}

We note that, as before, subalgebras of an $\fO$-algebra are
understood to be $\fO$-subalgebras; whenever we consider
$\ZZ_p$-subalgebras, we will explicitly state as much.

Similarly to the case of Lie algebras, a torsion-free finitely generated
pro-$p$ group $G$ which admits a central series $G = G_1 \ge G_2
\ge\dotsb$ of closed subgroups with $\bigcap\limits_{i=1}^\infty G_i = 1$
and $[G_i,_{p-1} G] \le G_{i+1}^p$ is \emph{saturable}.

If $\fg$ is a saturable Lie $\ZZ_p$-algebra, then the underlying topological
space of $\fg$ can be endowed with the structure of a saturable pro-$p$ group
using the  Hausdorff series. Conversely, every saturable pro-$p$ group gives
rise to a saturable Lie $\ZZ_p$-algebra and these two functorial operations
furnish mutually quasi-inverse equivalences between the categories of
saturable Lie $\ZZ_p$-algebras and saturable pro-$p$ groups
(defined as full subcategories of all Lie $\ZZ_p$-algebras and pro-$p$
groups, respectively); see \cite[\S 4]{Gon07} for an overview and
\cite[Ch.\ 4]{Laz65} for details. 
While the general interplay between subalgebras and subgroups
is subtle, we note the following fact.

\begin{lemma}
  \label{lem:same_idx}
  Let $\fg$ be a saturable Lie $\ZZ_p$-algebra and let $\mathfrak h$ be a
  saturable subalgebra of finite additive index.
  Let $G$ and $H$ be the saturable pro-$p$ groups associated with $\fg$ and
  $\mathfrak h$ via the Hausdorff series (so that $H \le G$).
  Then $\idx{G:H} = \idx{\mathfrak g:\mathfrak h}$.
\end{lemma}
\begin{proof}
  This can be proved in the same way as \cite[Lem.\ 3.2(4)]{GS09}.
\end{proof}

\subsection{Orbits of $p$-adic linear groups}
\label{ss:linear_orbits}

Let $\fO$ be the valuation ring of a non-Archimedean local field $K$.
Recall that $\fP$ denotes the maximal ideal of $\fO$
and that $q$ and $p$ denote the size and characteristic of the residue
field of $K$, respectively.
Further recall the definition of $\Zeta^\oc_G(T)$
(Definition~\ref{d:cc_oc}(\ref{d:cc_oc2})).
Although we will not need it in the sequel, since it might be of
independent interest, we note the following rationality statement for
$\Zeta^\oc_G(T)$.

\begin{thm}
  \label{thm:orbit_rationality}
  Let $G \le \GL_d(\ZZ_p)$.
  Then $\Zeta_{G}^{\oc}(T) \in \QQ(T)$.
  More precisely, there are $a_1,\dotsc,a_r \in \ZZ$ and $b_1,\dotsc,b_r \in \NN$ such that
  $\prod_{i=1}^r (1-p^{a_i}T^{b_i}) \Zeta^\oc_G(T) \in \QQ[T]$.
\end{thm}

\begin{rem}
  \label{rem:oc_positive_similarity}
  \quad
  \begin{enumerate}
  \item
    \label{rem:oc_positive_similarity1}
    The author does not know if the conclusion of Theorem~\ref{thm:orbit_rationality}
    remains valid if $G$ is allowed to be a subgroup of $\GL_d(\FF_q\llb z\rrb)$.
    The proof of Theorem~\ref{thm:orbit_rationality} below combines
    basic $p$-adic Lie theory and a powerful model-theoretic result due to
    Cluckers~\cite[App.~A]{HMRC17}.
    Both of these ingredients are only available in characteristic zero.
    In the latter case, this reflects the mysterious nature of the model theory of
    local fields of (small) positive characteristic.
  \item 
    \label{rem:oc_positive_similarity2}
    Avni et al.~\cite[Thms E, A.5]{AKOV16b} gave 
    formulae for the ``similarity class zeta functions''
    associated with the groups $\GL_d(\fO)$ and $\mathrm{GU}_d(\fO)$ for $d =
    2,3$.
    In addition to being consistent with Theorem~\ref{thm:orbit_rationality},
    their formulae are valid for all local fields $K$ subject to the sole
    assumption that $q$ be odd in case of $\mathrm{GU}_d(\fO)$.
    As explained in (\ref{rem:oc_positive_similarity1}),
    the techniques employed here seem incapable of establishing rationality
    results in such great generality.
  \end{enumerate}
\end{rem}

\begin{lemma}
  Let $\bar G$ be the closure of $G \le \GL_d(\fO)$ in $\GL_d(\fO)$.
  Then $\Zeta^\oc_{\bar G}(T) = \Zeta^\oc_G(T)$.
\end{lemma}
\begin{proof}
  The open subgroups $\Gamma_i := \{ a \in \GL_d(\fO) : a \equiv 1 \pmod {\fP^i}\}$
  form a fundamental system of neighbourhoods of the identity
  in $\GL_d(\fO)$.
  The claim follows since $G \Gamma_i = \bar G \Gamma_i$ for all $i \ge 0$.
\end{proof}

\begin{proof}[Proof of Theorem~\ref{thm:orbit_rationality}]
  Define an equivalence relation $\sim_n$ on $\ZZ_p^d$
  via
  \[
  \xx \sim_n \yy \quad:\Leftrightarrow\quad
  \exists g \in G.\, \xx \equiv \yy g \pmod{p^n}.
  \]

  Our theorem will follow immediately from \cite[Thm~A.2]{HMRC17} 
  once we have established that $\sim_n$ is definable (definably in
  $n$) in the subanalytic language used in \cite[App.\ A]{HMRC17}.
  By the preceding lemma, we may assume that $G = \bar G$.
  It then follows from the well-known structure theory of
  $p$-adic analytic groups (see \cite{Laz65,DdSMS99})
  that there exists an open saturable (or, more restrictively, uniform) subgroup $H
  \le G$ of the form $H = \exp(p^2 \mathfrak h)$, where
  $\mathfrak h \subset \Gl_d(\ZZ_p)$ is a suitable saturable (or uniform)
  $\ZZ_p$-subalgebra.
  Let $T$ be a transversal for the right cosets of $H$ in $G$.
  Then, for $\xx,\yy\in \ZZ_p^d$, $\xx \sim_n \yy$ if and only if
  \[
  \bigvee_{t\in T} \exists a\in \mathfrak h.\, \xx \exp(p^2a) 
  \equiv \yy t^{-1} \pmod{p^n}.
  \]
  The claim thus follows since the $d\times d$-matrix exponential
  $\exp(p^2X)$ is given by $d^2$ power series in $d^2$ variables which
  all converge on $\Gl_d(\ZZ_p)$.
\end{proof}

\begin{rem}
  Let $K/\QQ_p$ be a finite extension.
  Then we may regard $G \le \GL_d(\fO)$ as
  a $\ZZ_p$-linear group of degree $d\idx{K:\QQ_p}$ via the regular
  representation of $\fO$ over $\ZZ_p$.
  In particular, 
  Theorem~\ref{thm:orbit_rationality}
  implies that $\Zeta^\oc_{G}(T)$ is rational
  provided that $K/\QQ_p$ is unramified.
\end{rem}

The following questions are inspired by Theorems~\ref{thm:denef}--\ref{thm:transfer} and \cite[Thm~C]{BDOP13}.

\begin{qu}
  \label{qu:uniform}
  Let $k$ be a number field with ring of integers $\fo$.
  Let $\GG \le \GL_d\otimes k$ be an algebraic group over $k$.
  Let $\sG$ denote the schematic closure of $\GG$ in $\GL_d\otimes
  \fo$.
  \begin{enumerate}
  \item \label{qu:uniform1}
    Do there exist $V_1,\dotsc,V_r$ and $W_1,\dotsc,W_r\in \QQ(X,T)$ as
    in Theorem~\ref{thm:denef} such that
    \[
    \Zeta_{\sG(\fo_v)}^\oc(T) = \sum_{i=1}^r \#V_i(\RF_v)\dtimes W_i(q_v,T)
    \]
    for almost all places $v\in \Places_k$?
  \item
    Do we have
    $\Zeta_{\sG(\fo_v)}^\oc(T) = \Zeta_{\sG(\RF_v\llb z\rrb)}^\oc(T)$ for almost all $v\in\Places_k$?
  \end{enumerate}
\end{qu}

By combining Theorem~\ref{thm:denef} and
Corollary~\ref{cor:unipotent_oc_cc} below, we see that
Question~\ref{qu:uniform}(\ref{qu:uniform1}) has a positive answer if $\GG$ is unipotent.
We conclude this subsection with some elementary examples of $\Zeta^\oc_G(T)$.
\begin{ex}
  \quad
  \begin{enumerate}
  \item
    It is easy to see that the rule
    $\xx \mapsto
    \min(\nu(x_1),\dotsc,\nu(x_d),n)$ induces a bijection $\fO_n^d/\GL_d(\fO)
    \to \{ 0,\dotsc,n\}$
    whence
    \[
    \Zeta^\oc_{\GL_d(\fO)}(T) = \sum_{n=0}^{\infty} (n+1)T^n = \frac 1{(1-T)^2}.
    \]
  \item
    Let $p \not= 2$.
    The number of orbits of $\langle -1\rangle \le \GL_1(\fO)$ on $\fO_n$ is $1 + (q^n-1)/2$
    so that
    \[
    \Zeta^\oc_{\langle -1\rangle}(T) =
    \frac 1{1-T} + \frac 1 2\Bigl(
    \frac 1 {1-qT} - \frac 1{1-T}
    \Bigr)
    =
    \frac{2 -qT - T }{2(1 - qT)(1 - T)}.
    \]
\item
  Let $G = \langle \left[\begin{smallmatrix} 0 & 1\\1& 0\end{smallmatrix}\right] \rangle\le \GL_2(\fO)$.
  It is easy to see that
  \begin{align*}
    \Zeta^\oc_G(T) & =
    \sum_{n=0}^\infty \frac{q^n(q^n+1)}2T^n
      = \frac{2 -q^2 T - qT}{2(1 - q^2 T)(1 - qT)}.
  \end{align*}
  \end{enumerate}
\end{ex}

We note that the fact that the preceding examples as well as
the formula in \cite[(1.12)]{AKOV16b} all satisfy functional equations under
the operation ``$(q,T) \to
(q^{-1},T^{-1})$'' does not seem to be explained by any of the results in the
present article (e.g.\ Theorem~\ref{thm:feqn}).

\subsection{Lie centralisers and group stabilisers}
\label{ss:cent_stab}

Let $\fO$ be the valuation ring of a local field $K \supset \QQ_p$.
Let $\rind(K/\QQ_p)$ denote the ramification index of $K/\QQ_p$.
As expected, for suitable matrix algebras and groups,
the equivalence between saturable pro-$p$ groups
and Lie $\ZZ_p$-algebras recalled in \S\ref{ss:saturable} 
can be made explicit using exponentials and logarithms.

In line with our previous notation (see \S\ref{ss:rescaling}), we write
$\Gl_d^m(\fO) := \{ a\in \Gl_d(\fO) : a \equiv 0 \pmod{\fP^m}\}$.
Moreover, we write $\GL_d^m(\fO) := \{ a \in \GL_d(\fO) : a \equiv 1 \pmod {\fP^m}\}$.

\begin{prop}[{\cite[Lem.\ B.1]{Klo05}}]
  Let $m > \frac {\rind(K/\QQ_p)}{p-1}$.
  The formal exponential and logarithm series 
  converge on $\Gl_d^m(\fO)$ and $\GL_d^m(\fO)$, respectively, and
  define mutually inverse bijections 
  $\exp\colon \Gl_d^m(\fO) \to \GL_d^m(\fO)$ and $\log\colon \GL_d^m(\fO) \to
  \Gl_d^m(\fO)$. 
\end{prop}

Hence, if $m > \frac {\rind(K/\QQ_p)}{p-1}$ and $\fg \subset \Gl_d^m(\fO)$ is a saturable
subalgebra, then we may identify the saturable pro-$p$ group associated with
$\fg$ in the sense of \S\ref{ss:saturable} with $\exp(\fg) \le \GL_d^m(\fO)$.

For $\xx \in \fO^d$ and $n \ge 0$, we write $\xx \bmod {\fP^n}$ for the
image of $\xx$ in $\fO_n^d$.

\begin{lemma}
  \label{lem:cent_sat}
  Let $\fg \subset \Gl_d^{\rind(K/\QQ_p)}(\fO)$ ($= p\,\Gl_d(\fO)$) be a saturable
  subalgebra.
  Then $$\cent_{\fg}(\xx \bmod {\fP^n}) = \bigl\{ a \in \fg : \xx a \equiv 0 \pmod{\fP^n}\bigr\}$$
  is a saturable subalgebra of $\fg$
  for all $\xx \in \fO^d$.
\end{lemma}
\begin{proof}
  Write $\rind = \rind(K/\QQ_p)$ and
  $\mathfrak c^n := \cent_{\fg}(\xx \bmod {\fP^n})$;
  obviously, $\mathfrak c^n$ is a subalgebra of $\fg$.
  Let $\fg = \gamma_1(\fg) \supset \gamma_2(\fg) \supset \dotsb$ be the lower central
  series of $\fg$.
  Then $\gamma_p(\mathfrak c^n) \subset \mathfrak c^{n+\mathfrak{e}}$.
  Indeed, each element, $a$~say, of $\gamma_p(\mathfrak c^n)$ is a sum
  of matrix products $c(p h)$ for suitable $c\in\mathfrak c^n$ and $h\in
  \Gl_d(\fO)$ and clearly, $\xx c(p h) \equiv 0
  \pmod{\fP^{n+{\mathfrak e}}}$.
  Let $\fg = \fg_1 \supset \fg_2\supset \dotsb$ be a potent filtration of $\fg$.
  Then $\mathfrak c^n = \mathfrak c^n \cap \fg_1 \supset \mathfrak
  c^n \cap \fg_2 \supset \dotsb$ is 
  a central series of $\mathfrak c^n$ with 
  $\bigcap\limits_{i=1}^\infty (\mathfrak c^n \cap \fg_i) = 0$.
  It is in fact a potent filtration 
  for $p\fg \cap \mathfrak c^{n+\mathfrak e}  = p\mathfrak c^n$ 
  whence
  $[(\mathfrak c^n \cap \fg_i),_{p-1} \mathfrak c^n]
  \subset p \fg_{i+1} \cap \gamma_p(\mathfrak c^n) \subset p\fg \cap
  \mathfrak c^{n+\mathfrak{e}} = p\mathfrak c^n$.
\end{proof}

\begin{lemma}
  \label{lem:exp_unit}
  Let $m > \frac {\rind(K/\QQ_p)} {p-1}$ and $a\in \Gl_d^m(\fO)$.
  Then there exists $u \in \GL_d^1(\fO)$ with $\exp(a) = 1 + au$.
\end{lemma}
\begin{proof}
  Let $g(X) := \sum_{i=0}^\infty \frac 1{(i+1)!} X^i$ so that $\exp(X) = 1 + X
  g(x)$ in $\QQ\llb X\rrb$.
  Let $a \in \Gl_d^m(\fO)$ be non-zero.
  Let $b$ be an entry of $a$ of minimal valuation.
  Then $\nu_K(b^{i}/(i+1)!)$ is a lower bound for the valuation
  of each entry of $\frac 1{(i+1)!} a^i$ for $i \ge 0$.
  It is well-known that $\nu_p((i+1)!) \le i/(p-1)$
  (see e.g.\ \cite[Lem.\ 4.2.8(1)]{Coh07}) so that $\nu_K((i+1)!) \le
  \rind(K/\QQ_p)i/(p-1)$.
  Therefore, $\nu_K(b^i/(i+1)!) \ge i (\nu_K(b) - \rind(K/\QQ_p)/(p-1)) =: f(i)$.
  Clearly, $f(i) \to \infty$ as $i \to \infty$ and $f(i) > 0$ for $i > 0$. 
  Hence, the series $g(a)$ converges in $\Mat_d(\fO)$ to an element
  of $\GL_d^1(\fO)$.
\end{proof}

By combining the preceding two lemmas, we obtain the following.
\begin{cor}
  \label{cor:cent_stab}
  Let $m > \max\left(\rind(K/\QQ_p)-1,\frac {\rind(K/\QQ_p)}{p-1}\right)$,
  let $\fg \subset \Gl_d^m(\fO)$ be a saturable subalgebra,
  and let $\xx \in \fO^d$.
  Then $\exp(\cent_{\fg}(\xx \bmod {\fP^n})) = 
  \Stab_{\exp(\fg)}(\xx \bmod \fP^n)$.
\end{cor}
\begin{proof}
  Let $a \in \fg$ and
  write $\exp(a) = 1 + au$ for $u \in \GL_d^1(\fO)$.
  Then
  \begin{align*}
  a \in \cent_{\fg}(\xx \bmod{\fP^n})
    \,& \Leftrightarrow \, \xx a \equiv 0 \pmod {\fP^n} \\
    \,& \Leftrightarrow \, \xx au \equiv 0 \pmod {\fP^n}\\
    \,& \Leftrightarrow \, \xx \exp(a) \equiv  \xx (1+au)   \equiv \xx \pmod
                              {\fP^n}
                              \\
    \,& \Leftrightarrow \,
                                  \exp(a) \in \Stab_{\exp(\fg)}(\xx \bmod{\fP^n}).
\qedhere
  \end{align*}
\end{proof}

\subsection{Proof of Theorem~\ref{thm:bounded_denominators}}
\label{ss:pf_bounded_denominators}

Let $\fO$ be the valuation ring of a local field
$K\supset \QQ_p$. Recall that $\rind(K/\QQ_p)$ denotes the ramification index of $K/\QQ_p$.

\begin{prop}
  \label{prop:Zg_ZG}
  Let $m > \max\left(\rind(K/\QQ_p)-1, \frac {\rind(K/\QQ_p)}{p-1}\right)$,
  let $\fg \subset \Gl_d^m(\fO)$ be a saturable subalgebra,
  and let $G = \exp(\fg)$.
  Then $\Zeta_{\fg}(T) = \Zeta^\oc_{G}(T)$.
\end{prop}
\begin{proof}
  Write $V = \fO^d$ and $V_n = V\otimes\fO_n$.
  Recall that $\fg_n$ denotes the image of $\fg$ under the natural map
  $\Gl_d(\fO) \to \Gl_d(\fO_n)$.
  For $n \ge 0$, 
  by combining Lemma~\ref{lem:WxM}, Lemma~\ref{lem:same_idx}, and
  Corollary~\ref{cor:cent_stab}, we obtain
  \begin{align*}
    \card{V_n/G} & = \sum_{\xx \in V_n} \card{\xx G}^{-1} = \sum_{\xx \in
                   V_n}\idx{G:\Stab_G(\xx)}^{-1} \\
    & = \sum_{\xx \in V_n} \idx{\fg:\cent_{\fg}(\xx)}^{-1}
      = \sum_{\xx \in V_n} \idx{\fg_n:\cent_{\fg_n}(\xx)}^{-1} \\
    & = \sum_{\xx \in V_n} \card{\xx \,\fg_n}^{-1} = \ask{\fg_n}{V_n}.\qedhere
  \end{align*}
\end{proof}

\begin{proof}[Proof of Theorem~\ref{thm:bounded_denominators}]
  Let $m > \max\left(\rind(K/\QQ_p)-1, \frac {\rind(K/\QQ_p)}{p-1}\right)$.
  Propositions~\ref{prop:rescale} and \ref{prop:Zg_ZG} show that
  $q^{dm}\Zeta_{\fg}(T) = \Zeta^m_{\fg^m}(T) = \Zeta^{\oc,m}_{\exp(\fg^m)}(T) \in \ZZ\llb T\rrb$.
\end{proof}

\subsection{Orbits and conjugacy classes of unipotent groups}
\label{ss:unipotent_orbits}

Let $\fO$ be the valuation ring of a local field $K\supset \QQ_p$.
Recall that $\Nil_d(\fO)\subset\Gl_d(\fO)$ is the Lie algebra of all
strictly upper triangular matrices and
that $\Uni_d$ denotes the group scheme of upper unitriangular $d\times d$
matrices.
The following is well-known.

\begin{prop}
  \label{prop:nil_exp_log}
Let $p \ge d$.
\begin{enumerate}
  \item \label{prop:nil_exp_log1}
    All $\ZZ_p$-subalgebras of $\Nil_d(\fO)$ and closed subgroups of
    $\Uni_d(\fO)$ are saturable.
  \item \label{prop:nil_exp_log2}
    $\exp$ and $\log$ define polynomial bijections
    between $\Nil_d(\fO)$ and $\Uni_d(\fO)$.
  \end{enumerate}
\end{prop}
\begin{proof}
  Noting that subalgebras of $\Nil_d(\fO)$ and closed subgroups of
  $\Uni_d(\fO)$ are nilpotent of class at most $d-1 < p$, 
  their lower central series constitute potent filtrations.
  This proves (\ref{prop:nil_exp_log1}). Part (\ref{prop:nil_exp_log2}) follows since $a^d = 0$ for $a \in
  \Nil_d(\fO)$.
\end{proof}

A simple variation of Proposition~\ref{prop:nil_exp_log} yields the following.
\begin{cor}
  \label{cor:Lie_nilpotent_matrices}
  Let $\fg \subset \Gl_d(\fO)$ be a subalgebra.
  Suppose that $\fg$ is nilpotent of class at most $c$ and that
  $a^{c+1} = 0$ (in $\Mat_d(\fO)$) for all $a \in \fg$.
  Further suppose that $p > c$.
  \begin{enumerate}
  \item $\fg$ is saturable.
  \item $G := \exp(\fg)$ is a saturable subgroup of $\GL_d(\fO)$.
  \item $\exp$ and $\log$ define mutually inverse polynomial bijections
    between $\fg$ and $G$. 
    \qed
  \end{enumerate}
\end{cor}

By going through the proof of Theorem~\ref{thm:bounded_denominators},
we now easily obtain the following.

\begin{cor}
  \label{cor:nil_O_algebra}
  Let the notation be as in Corollary~\ref{cor:Lie_nilpotent_matrices}.
  Then $\Zeta^\ak_{\fg}(T) = \Zeta^\oc_{\exp(\fg)}(T)$ and
  thus, in particular, $\Zeta^\ak_{\fg}(T) \in \ZZ\llb T\rrb$. \qed
\end{cor}

This proves the first part of Theorem~\ref{thm:unipotent}.
We note that Corollary~\ref{cor:diagonal} shows that we may not, in
general, relax the assumptions in Corollary~\ref{cor:nil_O_algebra} 
and merely assume that $\fg\subset\Gl_d(\fO)$ is a nilpotent subalgebra.
The following completes the proof of Theorem~\ref{thm:unipotent}.

\begin{prop}
  \label{prop:cc}
  Let $\fg \subset \Gl_d(\fO)$ be an isolated subalgebra
  consisting of nilpotent matrices.
  Suppose that $p \ge d$.
  Then $\Zeta^\cc_{\exp(\fg)}(T) = \Zeta^\ak_{\ad(\fg)}(T)$.
\end{prop}
\begin{proof}
  By Engel's theorem, $\fg$ is $\GL_d(K)$-conjugate to a subalgebra of
  $\Nil_d(K)$ and hence nilpotent of class at most $d-1$;
  in particular, $\ad(a)^d = 0$ for all $a \in \fg$.
  By Corollary~\ref{cor:Lie_nilpotent_matrices},
  $G := \exp(\fg)$ is a saturable subgroup of $\GL_d(\fO)$.

  Let $\Ad\colon G \to \GL(\fg)$
  denote the adjoint representation of $G$;
  hence, $\Ad(g)\colon \fg \to \fg, \,a \mapsto \log(\exp(a)^g) = a^g $
  for $g\in G$.
  Recall that $G_n$ denotes the image of $G$ in $\GL_d(\fO_n)$
  and $\fg_n$ that of $\fg$ in $\Gl_d(\fO_n)$.
  Clearly, $a \equiv 0\pmod{\fP^n}$ if and only if $\exp(a) \equiv 1
  \pmod{\fP^n}$ for $a\in \fg$.
  We may thus identify conjugacy classes of $G_n$ with $\Ad(G)$-orbits on $\fg_n$.
  As $\fg$ is isolated within $\Gl_d(\fO)$,
  we may identify $\fg_n = \fg \otimes \fO_n$
  and obtain $\Zeta^\cc_G(T) = \Zeta^\oc_{\Ad(G) \acts \fg}(T)$.

  By Corollary~\ref{cor:Lie_nilpotent_matrices},
  $\ad(\fg)$ is a saturable subalgebra of $\Gl(\fg)$.
  The Hausdorff series shows that
  $\log(\exp(b)^{\exp(a)})
  = \sum_{i=0}^\infty \frac 1 {i!}[b,_i a]$
  for $a,b\in \fg$ (see \cite[Eqn~(3)]{GS09})
  whence
  $\Ad(\exp(a)) = \exp(\ad(a))$ for all $a\in \fg$.
  Thus, $\Ad(G) = \exp(\ad(\fg))$ and 
  Corollary~\ref{cor:nil_O_algebra} shows that
  $\Zeta^\oc_{\Ad(G) \acts \fg}(T) = \Zeta^\ak_{\log(\Ad(G))\acts \fg}(T) =
  \Zeta^\ak_{\ad(\fg)\acts \fg}(T)$.
\end{proof}

\begin{rem}
  \label{rem:isolated}
  The conclusion of Proposition~\ref{prop:cc} does not generally hold if
  $\fg$ is not isolated.
  For a simple example, take $\fg = \fP\dtimes \Nil_2(\fO)$.
  Then $\ad(\fg) = \{0\} \subset \Gl(\fg) \approx \Gl_1(\fO)$ and thus
  $\Zeta^{\ak}_{\ad(\fg)}(T) = 1/(1-qT) = 1 + qT + \mathcal O(T^2)$.
  On the other hand, the reduction of $\exp(\fg)$ modulo $\fP$ is trivial
  whence $\Zeta^{\cc}_{\exp(\fg)}(T) = 1 + T + \mathcal O(T^2)$.
\end{rem}

Using the well-known equivalence between unipotent algebraic groups
and nilpotent finite-dimensional Lie algebras over a field of
characteristic zero (see \cite[Ch.~IV]{DG70}), 
Corollary~\ref{cor:nil_O_algebra} and Proposition~\ref{prop:cc}
now imply the following global result.

\begin{cor}
  \label{cor:unipotent_oc_cc}
  Let $k$ be a number field with ring of integers $\fo$.
  Let $\GG \le \Uni_d\otimes_{\ZZ}\, k$ be a unipotent algebraic group
  over $k$ and let $\sG \le \Uni_d\otimes_{\ZZ}\, \fo$ be the
  associated $\fo$-form of $\GG$ (i.e.~the schematic
  closure of~$\GG$).
  Let $\bm\fg \subset \Nil_d(k)$ be the Lie algebra of
  $\GG$ and $\fg = \bm\fg \cap \Nil_d(\fo)$.
  Then for almost all $v\in \Places_k$,
  $\Zeta^\oc_{\sG(\fo_v)}(T) = \Zeta^\ak_{\fg_v}(T)$ and
  $\Zeta^\cc_{\sG(\fo_v)}(T) = \Zeta^\ak_{\ad(\fg_v)}(T)$.
  \qed
\end{cor}

Using Theorem~\ref{thm:feqn}, we further establish the following functional
equations for orbit-counting and conjugacy class zeta functions arising from
unipotent algebraic groups.

\begin{cor}
  \label{cor:unipotent_feqn}
  Let the notation be as in Corollary~\ref{cor:unipotent_oc_cc}.
  Then for almost all $v\in \Places_k$,
  \[
  \Zeta^\oc_{\sG(\fo_v)}(T) \Biggm\vert_{(q_v,T) \to (q_v^{-1},T^{-1})} = 
  (-q_v^d T) \dtimes \Zeta^\oc_{\sG(\fo_v)}(T)
  \]
  and
  \[
  \pushQED{\qed}
  \Zeta^\cc_{\sG(\fo_v)}(T) \Biggm\vert_{(q_v,T) \to (q_v^{-1},T^{-1})}  = 
  (-q_v^{\dim_k(\GG)} T) \dtimes \Zeta^\cc_{\sG(\fo_v)}(T).
  \qedhere
  \popQED                                                                          
  \]
\end{cor}

\begin{cor}
  Let $k$ be a number field with ring of integers $\fo$.
  Let $\GG \le \Uni_d\otimes_{\ZZ}\, k$ 
  and $\HH \le \Uni_e\otimes_{\ZZ}\, k$ 
  be unipotent algebraic groups over $k$ with
  $\fo$-forms $\sG$ and $\sH$ as above.
  Suppose that for almost all $v \in \Places_k$,
  $\Zeta^\cc_{\sG(\fo_v)}(T) = \Zeta^\cc_{\sH(\fo_v)}(T)$.
  Then $\dim_k(\GG) = \dim_k(\HH)$. \qed
\end{cor}
\begin{proof}
  Corollary~\ref{cor:unipotent_feqn}  and \cite[\S 4]{stability} allow us to
  recover $\dim_k(\GG)$ from $\bigl(\Zeta^\cc_{\sG(\fo_v)}(T)\bigr)_{v \in
    \Places_k\setminus S}$
  for any finite set $S \subset \Places_k$.
\end{proof}

We note that there are examples of non-isomorphic groups
$\GG$ and $\HH$ (of the same dimension) which satisfy
$\Zeta^\cc_{\sG(\fo_v)}(T) = \Zeta^\cc_{\sH(\fo_v)}(T)$ for almost all $v \in
\Places_k$;
see Table~\ref{tab:cc6} in \S\ref{ss:ex_cc}.

\section{Further examples}
\label{s:zeta}

\subsection{Computer calculations: \textsf{Zeta}}

The author's software package \textsf{Zeta}~\cite{Zeta} 
for Sage~\cite{Sage} can compute numerous types  of ``generic local''
zeta functions in fortunate (``non-degenerate'') cases. 
The techniques used by \textsf{Zeta} were developed over the course of
several papers; see \cite{padzeta}, in particular, and \cite{spp1489} for an
overview and references to other pieces of software that \textsf{Zeta} relies upon.
When performing computations, \textsf{Zeta} proceeds by attempting to explicitly
compute certain types of $p$-adic integrals.
Fortunately, the integrals in \eqref{eq:int_K_minors} and
\eqref{eq:int_O_minors} can both be encoded in terms of the
``representation data'' introduced in \cite[\S 5]{unipotent} whence
the author's computational techniques apply verbatim to the functions
$\Zeta_{M\otimes_{\ZZ} \fO}(T)$, where $M$ is $\ZZ$-defined.
In detail, given a submodule $M \subset \Mat_{d\times e}(\ZZ)$,
\textsf{Zeta} can be used to attempt to construct a rational
function $W(X,T) \in \QQ(X,T)$ with the following property:
for almost all primes $p$ and all finite extensions $K/\QQ_p$,
$\Zeta_{M\otimes_{\ZZ}\fO_K}(T) = W(q_K,T)$;
we note that for various reasons, \textsf{Zeta} may well fail to
construct $W(X,T)$ even if it exists.
Given $M \subset \Mat_{d\times e}(\ZZ)$, \textsf{Zeta} can
also be used to attempt to construct a formula as in Theorem~\ref{thm:denef}.
We note that while the techniques used by \textsf{Zeta} can, at least in
principle, be used to construct an explicit number $C_M$ such that all
primes $p$ which needed to be excluded above satisfy $p < C_M$, such a
number $C_M$ is not presently determined.

The remainder of this section is devoted to a number of examples
of functions $\Zeta_M(T)$ and $\Zeta^\cc_G(T)$ (via
Theorem~\ref{thm:unipotent}) computed with the help of \textsf{Zeta}. 
Throughout, $\fO$ denotes the valuation ring of a non-Archimedean local field
$K \supset \QQ_p$ with residue field size $q$.

\subsection{Examples of ask zeta functions}

\begin{ex}[Small poles and unbounded denominators]
  \label{ex:unbounded_denom}
  Let
  \[
  M = \left\{
    \begin{bmatrix}
      a & b & a \\
      b & c & d \\
      a & d & c
    \end{bmatrix}
    : a,b,c,d \in \fO
    \right\}.
  \]
  Then for sufficiently large $p$,
  \[
  \Zeta_{M}(T) = 
  \frac{1 + 5 q^{-1}T - 12q^{-2}T + 5 q^{-3}T + q^{-4}T^2}{(1 - q^{-1}T)(1 - T)^2}.
  \]
  Hence, the real poles of $\zeta_M(s)$ are $-1$ and $0$; it is easy
  to see that $\genidim(M) = 3$ (see Definition~\ref{d:genidim}).
  This example illustrates that, 
  in contrast to the case of $\Mat_{d\times e}(\fO)$,
  $d - \genidim(M)$ is generally not a lower bound for the real poles
  of $\zeta_M(s)$.
  Note that $\Zeta_M(T)$ has unbounded denominators---the author has
  found comparatively few modules of square matrices with this
  property (and initially suspected they did not exist).
\end{ex}

\begin{ex}
  \label{ex:non_Lie}
  Suppose that $p \not= 2$ and let
  \[
  M = \left\{
    \begin{bmatrix}
      0 & x_2 & -x_3 & 0 & 0 & x_1 \\
      0 & 0 & x_1 & \frac {x_2} 2 & -\frac {x_3} 2 & x_5 \\
      0 & 0 & 0& x_1 & 0 & x_4 \\
      0 & 0 & 0 & 0 & x_1 & x_3 \\
      0 & 0 & 0 & 0 & 0 & x_2 \\
      0 & 0 & 0 & 0& 0& 0
    \end{bmatrix} :
    x_1,\dotsc,x_5 \in \fO
    \right\}.
  \]
  For sufficiently large $p$,

  {\small
  \begin{align*}
  \Zeta_M(T) = \bigl(\,
    &
      (+ q^{36} T^{19} - 4 q^{35} T^{19} - q^{34} T^{20} + q^{35} T^{18} + 8 q^{34} T^{19} -
      2 q^{34} T^{18} - 2 q^{33} T^{19} - q^{34} T^{17}
    \\ &
          - 6 q^{33} T^{18} - q^{32} T^{19}
      + 3 q^{33} T^{17} + 5 q^{32} T^{18} + 3 q^{32} T^{17} + 6 q^{31} T^{18} -
      q^{32} T^{16} - 12 q^{31} T^{17}
    \\ &- 2 q^{30} T^{18} - 9 q^{30} T^{17} - q^{31} T^{15}
      + 14 q^{30} T^{16} + 14 q^{29} T^{17} + 4 q^{30} T^{15} + 5
         q^{29} T^{16}
    \\ &- 3 q^{28} T^{17} - 14 q^{29} T^{15} - 41 q^{28} T^{16} + q^{29} T^{14} +
      12 q^{28} T^{15} + 26 q^{27} T^{16} - 2 q^{28} T^{14}
    \\& + 46 q^{27} T^{15} -
      4 q^{26} T^{16} - 7 q^{27} T^{14} - 73 q^{26} T^{15} + 2 q^{27} T^{13} -
      24 q^{26} T^{14} + 32 q^{25} T^{15}
    \\& - 2 q^{26} T^{13} + 103 q^{25} T^{14} -
      3 q^{24} T^{15} + 6 q^{25} T^{13} - 98 q^24 T^{14} - q^25 T^12 -
      89 q^24 T^{13} 
    \\&+ 29 q^{23} T^{14} + 8 q^{24} T^{12} + 176 q^{23} T^13 -
      2 q^{22} T^14 + 35 q^{23} T^{12} - 115 q^{22} T^{13} + q^{23}
        T^{11} 
    \\&-178 q^{22} T^{12} + 25 q^{21} T^{13} - 15 q^{22} T^{11} + 223 q^{21} T^{12} -
      2 q^{20} T^{13} + 119 q^{21} T^{11}
    \\& - 100 q^{20} T^{12} + q^{21} T^{10} -
      262 q^{20} T^{11} + 16 q^{19} T^{12} - 39 q^{20} T^{10} + 214 q^{19} T^{11} -
      q^{18} T^{12}
    \\&+ 176 q^{19} T^{10} - 61 q^{18} T^{11} + 3 q^{19} T^9 -
      280 q^{18} T^{10} + 3 q^{17} T^{11} - 61 q^{18} T^9 + 176 q^{17}
        T^{10}
    \\&-      q^{18} T^8 + 214 q^{17} T^9 - 39 q^{16} T^{10} + 16 q^{17} T^8 -
      262 q^{16} T^9 + q^{15} T^{10} - 100 q^{16} T^8
    \\&+ 119 q^{15} T^9 -
      2 q^{16} T^7 + 223 q^{15} T^8 - 15 q^{14} T^9 + 25 q^{15} T^7 -
      178 q^{14} T^8 + q^{13} T^9
    \\&- 115 q^{14} T^7 + 35 q^{13} T^8 -
      2 q^{14} T^6 + 176 q^{13} T^7 + 8 q^{12} T^8 + 29 q^{13} T^6 -
      89 q^{12} T^7
    \\&- q^{11} T^8 - 98 q^{12} T^6 + 6 q^{11} T^7 - 3 q^{12} T^5 +
      103 q^{11} T^6 - 2 q^{10} T^7 + 32 q^{11} T^5 - 24 q^{10} T^6 
    \\& + 2 q^9 T^7 - 73 q^{10} T^5 - 7 q^9 T^6 - 4 q^{10} T^4 + 46 q^9 T^5 -
      2 q^8 T^6 + 26 q^9 T^4 + 12 q^8 T^5 + q^7 T^6
    \\&- 41 q^8 T^4 -
      14 q^7 T^5 - 3 q^8 T^3 + 5 q^7 T^4 + 4 q^6 T^5 + 14 q^7 T^3 +
      14 q^6 T^4 - q^5 T^5
    \\&- 9 q^6 T^3 - 2 q^6 T^2 - 12 q^5 T^3 -
      q^4 T^4 + 6 q^5 T^2 + 3 q^4 T^3 + 5 q^4 T^2 + 3 q^3 T^3 - q^4 T
        \\& - 6 q^3 T^2 - q^2 T^3 - 2 q^3 T - 2 q^2 T^2 + 8 q^2 T + q T^2 -
      q^2 - 4 q T + T\bigr)
    \\&\,/
        \bigl(
        q^2
        (1 - q^{10} T^5)
        (1 - q^8 T^4)
        (1 - q^5 T^3)
        (1 - q^4 T^2)^2
        (1 - q^3 T^2)
        (1 - q^2 T)
        (1 - q T)^2
        \bigr).
  \end{align*}}

Since $\Zeta_M(T) = 1 + (2 q^2 + 4 q + 4q^{-1} - q^{-2} - 8)T +
\mathcal O(T^2)$, we see that, 
in contrast to $\orbsize$-maximal or $\kersize$-minimal cases (see \S\ref{s:examples_max}--\ref{s:examples_min}),
the complexity of $\ask{M_1}{V_1}$ is in general a poor indicator of that of $\Zeta_M(T)$.

We note that by Corollary~\ref{cor:nil_O_algebra} and since
$\Zeta_M(T) \not\in \ZZ\llb T\rrb$, the module $M$ cannot be a Lie subalgebra of $\Nil_6(\fO)$. 
Indeed, this is readily verified directly even though $M$ is 
listed among Lie algebras in \cite[Table 5]{Roj16}.
\end{ex}

Rojas's article~\cite{Roj16} provides numerous examples of Lie subalgebras $\fg
\subset \Nil_d(\ZZ)$, say, for $d \le 6$. 
For many of these, we may use \textsf{Zeta} to compute
$\Zeta_{\fg\otimes_{\ZZ}\ZZ_p}(T)$.
Here, we only include one example.

\begin{ex}
  Let $p\not= 2$ and let
  \[
  \fg = \left\{
    \begin{bmatrix}
      0 & x_1 & \frac{x_2} 2 & -\frac{x_3} 2 & x_5 \\
      0 & 0 & x_1 & 0 & x_4 \\
      0 & 0 & 0 & x_1 & x_3 \\
      0 & 0 & 0 & 0 & x_2 \\
      0 & 0 & 0 & 0 & 0
    \end{bmatrix} :
    x_1,\dotsc,x_5 \in \fO
    \right\}.
  \]
  Then $\fg$ is a Lie subalgebra of $\Nil_5(\fO)$ of nilpotency class $4$, listed 
  as $L_{5,6}$ (de~Graaf's~\cite{dG07} notation) in
  \cite[Table~3]{Roj16}.
  For sufficiently large $p$,
  \begin{align*}
  \Zeta_{\fg}(T)  =
            \bigl(\,&+q^8 T^7 - 3 q^8 T^6 + q^8 T^5 + q^7 T^6 + 2 q^7 T^5 -
                   2 q^6T^5 - 2 q^6T^4 - q^5T^5 + 6 q^5T^4
    \\ & - 3 q^4 T^4 - 3 q^4 T^3 + 6 q^3 T^3 - q^3 T^2 -
                   2 q^2 T^3 - 2 q^2 T^2 + 2 q T^2 + qT + T^2 
    \\ & - 3 T + 1\bigr) / \bigl( (1-q^5 T^3) (1 - q^4 T^2) (1 - q^2 T) (1-qT)^2 \bigr).
  \end{align*}
\end{ex}

Numerous examples (including the case of $\Mat_{d\times e}(\fO)$) show
that $\zeta_M(s)$ may have a pole at zero and
Example~\ref{ex:unbounded_denom} shows that negative poles can arise
even for modules of square matrices.
In contrast, 
all of the author's computations are consistent with the following
question having a positive answer.

\begin{qu}
  \label{qu:positive_poles}
  Let $k$ be a number field with ring of integers $\fo$.
  Let $\fg \subset \Nil_d(\fo)$ be a Lie subalgebra.
  Is it the case that for almost all $v \in \Places_k$, every real pole
  of $\zeta_{\fg_v}(s)$ is positive?
\end{qu}

Supposing that Question~\ref{qu:positive_poles} indeed has a positive answer,
if $\GG \le \Uni_d\otimes_{\ZZ}\, k$ is an algebraic group over $k$ with
associated $\fo$-form $\sG \le \Uni_d\otimes_{\ZZ}\, \fo$ (see
Corollary~\ref{cor:unipotent_oc_cc}),
then we may evaluate the meromorphic continuation of
$\Zeta^\oc_{\sG(\fo_v)}(q_v^{-s})$ at $s = 0$ for almost all $v \in \Places_k$.
Inspired by similar questions regarding the behaviour at zero of
local subalgebra~\cite[Conj.~IV]{topzeta}, submodule~\cite[Conj.~E]{cyclic},
and representation~\cite[Qu.\ 8.5]{padzeta} zeta functions,
it would then be interesting to see if one can interpret the resulting
rational numbers, say in terms of properties of the orbit space $\fo_v^d/\sG(\fo_v)$.

\subsection{Examples of conjugacy class zeta functions}
\label{ss:ex_cc}

Let $k$ be a number field with ring of integers $\fo$.
Morozov~\cite{Mor58} classified nilpotent Lie algebras of dimension at
most $6$ over an arbitrary field of characteristic
zero---equivalently, he classified unipotent algebraic groups of
dimension at most $6$ over these fields.
A recent computer-assisted version of this classification (valid for
fields of characteristic $\not= 2$) is due to de~Graaf~\cite{dG07}.
We use his notation and let $L_{d,i}$ (or $L_{d,i}(a)$) denote the
$i$th Lie $k$-algebra (with parameter $a$) given in \cite[\S 4]{dG07}.

Table~\ref{tab:cc5} provides a complete list of  ``generic conjugacy
class zeta function'' associated with nilpotent Lie $k$-algebras of
dimension at most $5$ in the following sense: for each such 
algebra $\bm\fg$, let $\GG$ be its associated unipotent algebraic
group over $k$. After choosing an embedding $\GG \le
\Uni_d\otimes_{\ZZ}\, k$, we obtain an $\fo$-form $\sG$ of $\GG$ as in
Corollary~\ref{cor:unipotent_oc_cc}.
Then for almost all $v \in \Places_k$ and all finite extensions
$K/k_v$, $\Zeta^\cc_{\sG(\fO)}(T)$ is given in Table~\ref{tab:cc5}.

In contrast to dimension at most $5$, \textsf{Zeta} is unable to
compute generic conjugacy class zeta functions associated with every
nilpotent Lie $k$-algebra of dimension $6$.
Nevertheless, Table~\ref{tab:cc6} contains numerous examples of such functions;
we only included examples corresponding to $\oplus$-indecomposable algebras.
Clearly, generic conjugacy class zeta functions of direct products of
algebraic groups are Hadamard products of the zeta functions corresponding to the factors.
We note that $L_{3,2} \approx \Nil_3(K)$ and $L_{6,19}(-1) \approx \Nil_4(K)$.
A formula for $\Zeta^\cc_{\Uni_3(\fO)}(T)$ was previously given
in~\cite[\S 8.2]{BDOP13}.
This formula is incorrect due to a sign mistake.
More substantially, the computation in \cite[\S 8.2]{BDOP13} relies on
\cite[Prop.~6.2]{BDOP13} and what seems to be a variation of the integral
formalism developed in \cite{BDOP13} for unipotent groups; this is however not
explained.
Said integral formalism in \cite{BDOP13} appears to be essentially different
from the methods developed and applied here.

\paragraph{Possible further directions.}
A refinement of Higman's conjecture (see \S\ref{s:intro}) predicts that
$\concnt(\Uni_d(\FF_q))$ is a polynomial in $q-1$ with non-negative
coefficients.
In recent years, the same question has been studied for groups of
$\FF_q$-rational points of unipotent radicals of Borel subgroups of more
general algebraic groups such as Chevalley groups of small rank; see, in particular,
work of Goodwin et al.~\cite{Goo06,GR09a,GMR14,GMR16}.
In this spirit, an elementary calculation using the formulae in
Table~\ref{tab:cc5}--\ref{tab:cc6} shows that the coefficients of the
generic conjugacy class zeta functions associated with $\Nil_3(K)$ and $\Nil_4(K)$ are
polynomials with non-negative coefficients in $q-1$, generalising the known
cases of the corresponding coefficients of $T$.
The same is true of the generic conjugacy class zeta functions associated with
$L_{4,3}$; the latter algebra is isomorphic to the nilradical of a Borel
subalgebra of $\Sp_4(K)$.
It would be interesting to further explore to what extent such
non-negativity properties are satisfied by the coefficients of ask zeta
functions in the setting of Goodwin et al.

For another intriguing problem, let $\mathfrak f_{c,d}$ be the free nilpotent
Lie ring of class $c$ on $d$ generators and write $\mathfrak f_{c,d}(R) :=
\mathfrak f_{c,d} \otimes_{\ZZ} R$.
O'Brien and Voll~\cite[\S 5]{O'BV15} gave a combinatorial description of
$\concnt (\exp(\mathfrak f_{c,d}(\FF_q)))$ under mild assumptions on $q$.
The generic conjugacy class zeta functions associated with $\mathfrak f_{3,2}$
and $\mathfrak f_{2,3}$ can be found in Tables \ref{tab:cc5}
and~\ref{tab:cc6}, respectively.
Lins computed the conjugacy class zeta functions associated with $\mathfrak
f_{2,d}$ for all $d$; see \cite[Cor.~1.11]{Lin18}.
It seems challenging to determine the conjugacy class zeta functions
$\mathfrak f_{c,d}$ in general.

\begin{table}[h]
  \centering
  \begin{tabular}{r|c}
    $\bm\fg$ & $\Zeta^\cc_{\sG(\fO))}(T)$
    \\ \hline
    $L_{1,1}$ & $1/(1-qT)$ \\  \hline
    $L_{2,1}$ & $1/(1-q^2T)$ \\\hline
    $L_{3,1}$ & $1/(1-q^3T)$ \\
    $L_{3,2} \approx \Nil_3(K)$ & $(1-T)/((1 - q^2T)(1 - qT))$ \\ \hline
    $L_{4,1}$ & $1/(1-q^4T)$ \\
    $L_{4,2}$ & $(1 - qT)/((1 - q^3T)(1 - q^2T))$ \\
    $L_{4,3}$ & $(1-T)/(1 - q^2T)^2$ \\ \hline
    $L_{5,1}$ & $1/(1-q^5T)$ \\
    $L_{5,2}$ & $(1 - q^2T)/((1 - q^4T)(1 - q^3T))$ \\
    $L_{5,3}$ & $(1 - qT)/(1 - q^3T)^2$ \\
    $L_{5,4}$ & $(1 - T)/((1 - q^4T)(1 - qT))_{\phantom{0_0}}$ \\
    $L_{5,5}$ & $\frac
                { 1
                - T 
                - qT 
                + q^2T 
                + q^2T^2 
                - q^3T^2 
                - q^4T^2 
                + q^4T^3
                }{(1 - q^5T^2)(1 - q^3T)(1 - qT)}$ \\
    $L_{5,6}$ & $\frac
                {1
                - 2T 
                + qT^2 
                + q^2T 
                - 2q^3T^2 
                + q^3T^3 
                }{(1 - q^5T^2)(1 - q^2T)(1 - qT)}$\\
    $L_{5,7}$ & $(1 - T)/((1 - q^3T)(1 - q^2T))$ \\
    $L_{5,8}$ & $(1 - qT)/(1 - q^3T)^2$ \\
    $L_{5,9} \approx \mathfrak f_{3,2}(K)$ & $(1 - T)/((1 - q^3T)(1 - q^2T))$
  \end{tabular}
  \caption{Complete list of generic conjugacy class zeta functions of
    unipotent algebraic groups of dimension at most $5$}
  \label{tab:cc5}
\end{table}

\begin{table}[h]
  \centering
  \begin{tabular}{r|c}
    $\bm\fg$ & $\Zeta^\cc_{\sG(\fO)}(T)$
    \\ \hline
    $L_{6,10}$, $L_{6,25}$; $L_{6,26} \approx \mathfrak f_{2,3}(K)$ & $(1 - qT)/((1 - q^4T)(1 - q^3T))$\\
    $L_{6,11}$, $L_{6,12}$, $L_{6,20}$ & $\frac{
                 1
                 - 2 q T
                 + q^2 T 
                 + q^4 T^2 
                 - 2 q^5 T^2 
                 + q^6 T^3 
                 }{(1 - q^6T^2)(1 - q^3T)^2}$ \\
    $L_{6,16}$ & $(1 - qT)(1 - T)/((1 - q^2T)^2(1 - q^3T))$ \\
    $L_{6,17}$ & $\frac{
                 1
                 - T
                 - qT
                 + q^2T
                 + q^3T^2
                 - q^4T^2
                 - q^5T^2
                 + q^5T^3
                 }{(1 - q^6T^2)(1 - q^3T)(1 - q^2T)}$ \\
    $L_{6,18}$ & $(1 - T)/((1 - q^2T)(1 - q^4T))$ \\
    $L_{6,19}(0)$ & $
                    \frac{
                     1
                    + T
                    - 3qT
                    - q^2T 
                    + q^3T^2 
                    + 3 q^4T^2
                    - q^5T^2
                    - q^5T^3
                    }
                    {(1 - q^3T)^3(1 - q^2T)}$ \\
    $L_{6,19}(-1) \approx \Nil_4(K)$,
    $L_{6,21}(0)$ & $(1 - qT)^2/((1 - q^3T)^2(1 - q^2T))$ \\

    $L_{6,21}(1)$
             & $
                                 \frac{
                                 1
                                 - T
                                 - qT
                                 + q^2T
                                 + q^2T^2
                                 - q^3T^2
                                 - q^4T^2
                                 + q^4T^3
                                 }{(1 - q^5T^2)(1 - q^3T)(1 - q^2T)}$ \\
    $L_{6,22}(0)$ & $\frac{
                    1
                    - q T
                    - q^2 T
                    + q^3 T
                    + q^4 T^2
                    - q^5 T^2
                    - q^6 T^2
                    + q^7 T^3
                    }
                    {(1 - q^7 T^2)(1 - q^4T)(1 - q^2T)}$\\
    $L_{6,23}$, $L_{6,24}(0)$ & $
                 \frac{
                 1
                 - 2qT
                 + q^3T 
                 + q^3T^2
                 - 2q^5T^2
                 + q^6T^3
                 }{(1 - q^7T^2)(1 - q^3T)(1 - q^2T)}$
  \end{tabular}
  \caption{Examples of generic conjugacy class zeta functions of
    unipotent algebraic groups of dimension $6$}
  \label{tab:cc6}
\end{table}

\clearpage
{
  \bibliographystyle{abbrv}
  \tiny
  \bibliography{ask}
}

\end{document}